\documentclass[12pt]{amsart}
\usepackage{latexsym,amsmath, amscd, amssymb, amsthm,  euscript, mathrsfs, hyperref, stmaryrd}
\usepackage[dvipsnames]{xcolor}
\usepackage[all]{xy}
\usepackage{tikz}
\usetikzlibrary{arrows}
\usetikzlibrary{cd}
\usepackage{makecell}
\usepackage{enumitem}

\newtheorem{thm}[subsection]{Theorem}
\newtheorem{thm/def}[subsection]{Theorem/Definition}
\newtheorem{cor}[subsection]{Corollary}

\newtheorem{lem}[subsection]{Lemma}
\newtheorem*{lem*}{Lemma}

\newtheorem{prop}[subsection]{Proposition}
\newtheorem*{prop*}{Proposition}
\newtheorem{problem}[subsection]{Problem}
\theoremstyle{definition}
\newtheorem{notation}[subsection]{Notation}
\newtheorem{defn}[subsection]{Definition}
\newtheorem{alg}[subsection]{Algorithm}

\theoremstyle{definition}

\theoremstyle{definition}

\newtheorem{rem}[subsection]{Remark}
\newtheorem*{rem*}{Remark}

\numberwithin{equation}{subsection}

\newtheorem{pg}[subsection]{}

\newtheorem{conclusion}[subsection]{Summary}

\newcommand{\Q}{\mathbf{Q}}

\newcommand{\mc}{\mathcal }

\newcommand{\Sp}{\text{\rm Spec}}
\newcommand{\Spec}{\text{\rm Spec}}

\newcommand{\Hom}{\underline {\text{\rm Hom}}}

\newcommand{\mls}{\mathscr}

\newcommand{\mk}{\mathrm{mk}}
\newcommand{\nd}{\mathrm{nd}}

\newcommand{\et }{\text{\rm \'et}}

\newcommand{\gp}{\text{\rm gp}}

\newcommand{\tors}{\mathrm{tors}}

\newcommand{\A}{\beta}

\newcommand{\bD}{\mathbf{D}}
\newcommand{\bE}{\mathbf{E}}

\newcommand{\bone}{\mathbf{1}}

\newcommand{\sC}{\mathscr{C}}

\newcommand{\fM}{\mathfrak{M}}
\newcommand{\fC}{\mathfrak{C}}

\newcommand{\cX}{\mc X}

\newcommand{\ba}{\mathbf{a}}

\newcommand{\bC}{\mathbf{C}}
\newcommand{\bb}{\mathbf{b}}

\usepackage{xcolor}

 \newcommand\RACHELoff{\newcommand{\Commentr}[1]{}}
\newcommand{\Rachel}[1]{\Commentr{#1}}

\newcommand\TEMPon{\newcommand{\Commentt}[1]{\noindent\color{Maroon}{\texttt 4/3: }##1\color{black}}}

 \newcommand\MARTINoff{\newcommand{\Commentp}[1]{}}
 \newcommand{\Martin}[1]{\Commentp{#1}}

\newcommand{\Tors}{\text{\rm Tors}}
\newcommand{\uTors}{\underline {\text{\rm Tors}}}
\newcommand{\ab}{\text{\rm ab}}

\newcommand{\MS}[2]{M_{{#2}\rightarrow {#1}}}

\AtBeginDocument{%
   \def\MR#1{}
}

\begin{document}



\MARTINoff

\RACHELoff

\TEMPon

\title{Twisted stable maps with colliding points}
\author{Martin Olsson and Rachel Webb}

\begin{abstract}
We study moduli spaces of stable maps from pointed curves, where the points are allowed to coincide,  with target a tame Deligne-Mumford stack.  This generalizes the Abramovich-Vistoli theory of twisted stable maps as well as work of Hassett, Alexeev and Guy, and Bayer and Manin, who studied stable maps to projective varieties 
from curves with weighted marked points.
\end{abstract}

\maketitle
\setcounter{tocdepth}{1}
\tableofcontents

\section{Introduction}
Our work in this article is a common generalization of Abramovich-Vistoli's theory \cite{AV} of twisted stable maps, and Alexeev-Guy \cite{AlexeevGuy} and Bayer-Manin's \cite{bayermanin} theory of weighted stable maps (building on Hassett's theory \cite{Hassett} of weighted stable curves). A key technical input is the theory of generalized log twisted curves introduced in \cite{OWI}.
Our main results, stated informally here and discussed  in greater detail below, are as follows. Let $\mc X$ be a separated tame Deligne-Mumford stack of finite type over a Noetherian base scheme $B$ and fix a suitable set of rational numbers $\ba = (a_1, \dots, a_n)$, called weights,  with $a_i \in (0, 1]$.
\begin{itemize}
\item  We define an algebraic stack of stable maps $\mls K_{g, \ba}(\mc X)$ to $\mc X$ from generalized log twisted curves with $\ba$-weighted marked points.
\item We show that if $\mc X$ has projective coarse moduli $X$ with fixed polarization, the substack $\mls K_{g, \ba}(\mc X, d)$ of maps of degree $d$ is proper. If moreover $\mc X$ is a global quotient, so is $\mls K_{g, \ba}(\mc X, d).$
\item We show that, if $\ba \leq \bb$, there is a functorial closed substack $\mls K_{g, \bb}^{\ba-\mathrm{stab}}(\mc X) \subseteq \mls K_{g, \bb}(\mc X)$, containing the substack classifying marked smooth curves, where a stabilization morphism
\(
\mls K_{g, \bb}^{\ba-\mathrm{stab}}(\mc X) \to \mls K_{g, \ba}(\mc X)
\)
is defined. 
\end{itemize}

To illustrate the fundamental difficulty in developing our theory, consider the following example.
Let $A$ be a discrete valuation ring with field of fractions $K$, and let $C \to \Sp(A)$ be a smooth curve with two sections $x, y$ that are distinct in the generic fiber but intersect at a point $x_0$ in the closed fiber. Let $G$ be a finite group of order invertible in $A$,
and let $P_K \to C_K$ be a $G$-cover, possibly ramified over $x_K$ and $y_K$,  whose quotient $P_K/G$ is a twisted curve $\mls C_K$ with coarse space $C_K$ and cyclic isotropy groups at $x_K$ and $y_K$.
\begin{problem}\label{prob:overview}
How can $\mls C_K$ and the induced morphism $\mls C_K \to BG$ be extended to a stack $\mls C$ over $A$ and representable morphism $\mls C \to BG$, such that the coarse moduli space of $\mls C$ is a morphism $\mls C \to C$ that is an isomorphism away from the sections $x$ and $y$?
\end{problem}

In fact there could be many stacks $\sC$ that solve this problem as stated; one should also require that $\sC \to \Spec(A)$ is a ``family of generalized twisted curves'' in some sense. 
Here we require $\mls C$ to be a family of generalized log twisted curves as defined in \cite{OWI}. Since such curves have abelian stabilizer groups we cannot always find one that solves Problem \ref{prob:overview}.

\subsection{Stable maps}
Let $\mc X$ be a separated tame Deligne-Mumford stack of finite type over a Noetherian base scheme $B$. As mentioned above, our theory of stable maps to $\mc X$ uses the \textit{generalized log twisted curves} introduced in \cite{OWI}. The important features of the definition are the following: 
\begin{itemize}
\item A generalized log twisted curve over a $B$-scheme $S$ is a tuple $\bC$ that includes the data of a nodal curve $C\to S$
and sections $s_i: S \to C$ of the smooth locus, not necessarily disjoint, as well as some log data.
\item The tuple $\bC$ has an associated stack $\mls C$ with coarse space $C$ and abelian stabilizer groups that are nontrivial only at nodes and markings.
\end{itemize}

Fix nonnegative integers $g, n$ and a set of weights 
$\ba = (a_1, \ldots, a_n)$. A \textit{prestable map} to $\mc X$ of type $(g, \ba)$ is a pair $(\bC, f: \mls C \to \mc X)$ such that the underlying map of coarse spaces $C \to X$ is $\ba$-weighted prestable in the sense of \cite{AlexeevGuy} and \cite{bayermanin}. That is, for any set markings that coincide in some geometric fiber of $C$, the sum of their associated weights $a_i$ is at most 1.

To motivate our definition of stability, consider the following example.  Let $k$ be an algebraically closed field and let $G$ be a finite group of order invertible in $k$. Let 
$$
(C, \{s_i\}_{i=1}^n, f:\mls C\rightarrow BG)
$$
 be an $n$-marked  twisted stable map in the sense of \cite{AV}, with coarse moduli space $C/k$.  When we introduce weights $\ba $ we may have a rational tail $E$ with marked points, say $s_1, \dots, s_r$ (after reordering), with $\sum _{i=1}^ra_i\leq 1$, and ideally (to mimic the definition of weighted stable maps in \cite{AlexeevGuy} \cite{bayermanin}) we would like to contract $E$. However this is not always possible.    By the theory of generalized log twisted curves \cite{OWI} we have a natural morphism of stacks $\mls C \to \mls C'$ that contracts the preimage of $E$ in $\sC$ (with its reduced structure), but critically, the morphism $f$ may not descend to a morphism $f':\mls C'\rightarrow BG$.  A basic obstruction is that the stabilizer group of $\mls C'$ at the image of $\mls E$ is abelian.  Therefore, if $P|_{\mls E}\rightarrow \mls E$ is the $G$-torsor corresponding to the map $f|_{\mls E}:\mls E\rightarrow BG$, then the associated  monodromy representation of the fundamental group  must be abelian if the map $f'$ exists; equivalently, the map $f|_{\mls E}$ must factor through $BA\rightarrow BG$ for an abelian subgroup $A\subset G$.  In this case we say that $f|_{\mls E}$ admits an \emph{abelian contraction} (see \ref{D:abeliancontractioni} for the definition in general).

 \begin{rem} Note that if all $a_i = 1$, the component $E$ satisfies $\sum _{i=1}^ra_i\leq 1$ if and only if there is at most one marked point, and in this case the fundamental group of $\mls E$ is abelian (being a quotient of the fundamental group of $\mathbf{P}^1-\{0, \infty \}$).  So $f|_{\mls E}$ always has an abelian contraction in this case.
 \end{rem}

 This leads to the following definition of stability of maps to the more general target $\mc X$.

\begin{defn}
A prestable map $(\bC, f: \mls C \to \mc X)$ of type $(g, \ba)$ is \emph{stable} if $f$ is representable and for all geometric points $\bar s \to S$, the following hold:
\begin{itemize}
\item The automorphism group scheme of the underlying coarse prestable map $(C_{\bar s}, \{s_i\}_{i=1}^n, C_{\bar s} \to X)$ is finite.
\item If $E \subset C_{\bar s}$ is an irreducible component such that $\mls E \to \cX$ has an abelian contraction, then either the normalization $\widetilde E$ of $E$ has positive genus, or
\[
\#\{e\in \widetilde E \mid \text{$e$ maps to a node of $C_{\bar s}$}\}+\sum _{i, s_i(\bar s)\in E}a_i>2.
\]
\end{itemize}
\end{defn}
 The second condition says that a component of $C_{\bar s}$ is stabilized by either having positive genus, by having enough (weighted) special points, or by \textit{failing} to have an abelian contraction.

 \subsection{Properties of the stack of stable maps} Define $\mls K_{g, \ba}(\mc X)$ to be the fibered category over the category of $B$-schemes which to any $T \to B$ associates the groupoid of stable maps of type $(g, \ba)$ over $T$. Without imposing further assumptions on $\mc X$, we can say the following (see \ref{T:7.4} and \ref{T:proper}).
 \begin{thm}\label{T:intro1}
The fibered category $\mls K_{g, \ba}(\mc X)$ is an algebraic stack locally of finite type over $B$ with finite diagonal. If $\mc X$ is proper the $\mls K_{g, \ba}(\mc X)$ satisfies the existence part of the valuative criterion.
 \end{thm}

Now assume additionally that the coarse moduli space $X$ of $\mc X$ is a projective $B$-scheme with a relatively ample line bundle $\mls O_X(1)$. The \emph{degree} of a stable map $(\bC, f:\mls C \to \mc X)$ is the degree of the underlying map $\bar f: C \to X$ of coarse spaces; that is, the degree of $\bar f^*\mls O_X(1)$. Let $\mls K_{g, \ba}(\mc X, d) \subseteq \mls K_{g, \ba}(\mc X)$ be the open and closed substack of stable maps of degree $d$. The following theorem is proved in \ref{T:proper} and \ref{T:14.2}.
\begin{thm}\label{T:intro2}
The fibered category $\mls K_{g, \ba}(X, d)$ is a proper algebraic stack with finite diagonal. If moreover $B$ is quasicompact and $\mc X$ is a global quotient stack, then $\mls K_{g, \ba}(\mc X, d)$ is also a global quotient stack.
\end{thm}
We note that when $\ba = \bone^n$ is the vector $(1, 1, \ldots, 1)$ with $n$ entries, the stack $\mls K_{g, \bone^n}(\mc X, d)$ is the Abramovich-Vistoli stack of twisted stable maps. When $\mc X$ is equal to its coarse space the stack $\mls K_{g, \ba}(\mc X, d)$ is the stack of weighted stable maps considered by Alexeev-Guy and Bayer-Manin.

\subsection{Stabilization morphisms} Our proofs of Theorems \ref{T:intro1} and \ref{T:intro2} 
heavily use the existence of stabilization morphisms between the moduli spaces $\mls K_{g, \ba}(\mc X)$ for varying $\ba$. 
Let $\mc X$ be a separated tame Deligne-Mumford stack of finite type over a Noetherian base scheme $B$. Fix nonnegative integers $g, n$ and two tuples of weights $\ba = (a_1, \ldots, a_n)$ and $\bb = (b_1, \ldots, b_n)$ with $a_i \leq b_i$ for all $i$. A prestable map $(\bC, f:\mls C \to \mc X)$ of type $(g, \bb)$ \textit{admits an $\ba$-stabilization} if it has a factorization through a stable map of type $(g, \ba)$.
We prove the following in \ref{C:stab}.

\begin{thm}\label{T:intro3}
The subcategory $\mls K_{g, \bb}^{\ba-\mathrm{stab}}(\mc X) \subseteq \mls K_{g, \bb}(\mc X)$ of $\bb$-stable maps that admit an $\ba$-stabilization is a closed substack containing the locus of points classifying marked smooth curves. Moreover an $\ba$-stabilization of a $\bb$-stable map is unique when it exists, so there is a natural morphism
\begin{equation}\label{eq:stab-intro}
\mls K_{g, \bb}^{\ba-\mathrm{stab}}(\mc X) \to \mls K_{g, \ba}(\mc X)
\end{equation}
sending a $\bb$-stable map to its $\ba$-stabilization.
\end{thm}
Note that this generalizes \cite[Thm 3.1]{AlexeevGuy} and \cite[Prop 3.1.1]{bayermanin} (which themselves generalize \cite[Prop 4.5]{Hassett}) to a stacky setting, with the interesting complication that the morphism \eqref{eq:stab-intro} is only defined on a canonical closed substack of $\mls K_{g, \bb}(\mc X)$.
The reason for this is the following result. Let $Q:\mls C \to \mls D$ be a morphism associated to an \textit{initial contraction} of generalized log twisted curves (see \cite[6.14]{OWI}) over $S$ and let $C \to D$ be the induced coarse moduli map. Let $f:\mls C \to \mc X$ be a morphism coarse coarse moduli map factors as $C \to D \to X$. The \emph{fibered category of factorizations} of $f$ associates to $T \to S$ the category of pairs $(g: \mls D_T \to \mc X, \alpha)$ where $\alpha$ is a 2-isomorphism from $g \circ Q_T$ to $f_T$. The following appears as \ref{T:9.2} in the body of the article.

\begin{thm}\label{T:intro4}
Assume that for each geometric point $\bar y \to D$ the fiber morphism $f_{\bar y}: \mls C_{\bar y} \to \mc X$ has an abelian contraction. Then the fibered category of factorizations of $f$ is represented by a closed subscheme of $S$.
\end{thm}
In particular this result says there is a maximal closed substack of $S$ where a factorization of $f$ exits and is unique. This theorem is a generalization of the result \cite[9.1.1]{AV} for twisted stable maps.

Finally, we mention an application of the stabilization morphisms \eqref{eq:stab-intro}. Assume $\mc X$ has projective coarse moduli space, so $\mls K_{g, \bone^n}(\mc X, d)$ is the Abramovich-Vistoli stack of twisted stable maps. From Theorem \ref{T:intro3} we get a closed substack $\mls K^{\ba-\mathrm{stab}}_{g, \bone^n}(\mc X, d) \subseteq \mls K_{g, \bone^n}(\mc X, d)$ for all $\ba \in (0, 1]^n\cap \mathbf{Q}^n$. The following result, proved in \ref{C:small}, identifies a compactification of smooth twisted stable maps contained in the Abramovich-Vistoli stack. 
\begin{thm}
Define $\mls K^{\mathrm{stab}}_{g, \bone^n}(\mc X, d) := \cap_{\ba \in (0, 1]^n\cap \mathbf{Q}^n} \mls K_{g, \bone^n}^{\ba-\mathrm{stab}}(\mc X, d)$. Then $\mls K^{\mathrm{stab}}_{g, \bone^n}(\mc X, d)$ is a closed substack of $\mls K_{g, \bone^n}(\mc X, d)$ containing all maps from smooth source curves.
\end{thm}

\subsection{Hypotheses and terminology}
We fix a separated tame Deligne-Mumford stack $\mc X$ of finite type over a noetherian base scheme $B$ with coarse space $\mc X \to X \to B$ (though we relax the separation assumption for the preliminaries in Section \ref{sec:prelims}). In particular, $\mc X \to B$ is of finite presentation.

\begin{rem} One could relax the assumption that $B$ is noetherian.  However, since we are assuming that $\mc X$ is of finite presentation over $B$ this does not add significantly more.  Indeed if $\mc X\rightarrow B$ is a tame Deligne-Mumford stack of finite presentation, then after shrinking on $B$ and using a standard limit argument there exists a finite type $\mathbf{Z}$-scheme $B_0$ and a tame Deligne-Mumford stack $\mc X_0/B_0$ such that $\mc X$ is the base change of $\mc X_0$ along a morphism $B\rightarrow B_0$.  
\end{rem}

\begin{rem}
It is natural to expect a theory of stable maps when $\mc X$ is a tame Artin stack, generalizing \cite{AOV} to include curves with weighted markings. However, the correct notion of an abelian contraction appears quite subtle in this case. Even when $\mc X$ has abelian stabilizers, many of the arguments are significantly more complicated than in the tame Deligne-Mumford setting. Hence we restrict ourselves to the latter in this article. 
\end{rem}

\begin{rem}
We do not impose a blanket assumption that $\mc X$ has projective coarse moduli as this is not necessary for many results in this paper. In particular, our theorems \ref{T:intro3} and \ref{T:intro4} do not require this hypothesis.
Indeed, it appears well-known to experts that this hypothesis is not needed for the analogous results for twisted stable maps appearing in \cite[Thm 3.1]{AlexeevGuy}, \cite[Prop 3.1.1]{bayermanin}, and \cite[9.1.1]{AV}, even though as stated the classical results require it. 

\end{rem}

We define a \textit{tame \'etale group scheme} over $S$ to be  a finite \'etale group scheme whose order in the fiber at any point of $s$ is coprime to the residue characteristic. In addition, we use the terminology discussed in \cite[1.10]{OWI}.

\subsection{Acknowledgements}
Olsson was partially supported by NSF FRG grant DMS-2151946 and the Simons Collaboration on Perfection in Algebra, Geometry, and Topology. Webb was partially supported by an NSF Postdoctoral Research Fellowship, award number 200213, and a grant from the Simons Foundation. Webb is grateful to the Isaac Newton Institute for their hospitality.

\section{Preliminaries}\label{sec:prelims}

This section contains definitions and results prerequisite to our definition of stable maps and discussion of their contractions. Throughout $\mc X$ is a tame Deligne-Mumford stack over a Noetherian base $B$ and $X$ is the coarse space of $\mc X$. 

\subsection{Abelian contractions}
Let $k$ be a field and let $\mls C$ be a tame Deligne-Mumford stack over $k$.
Let $\Sp(k) \to B$ be a morphism and let $f: \mls C \to \cX$ be a morphism over $B$. 

\begin{defn}\label{D:abeliancontractioni}
An \textit{abelian contraction of $f$} is a factorization
\begin{equation}\label{E:factorit}
\xymatrix{
\mls C\ar@/^2pc/[rr]^-f\ar[r] &BA \ar[r]^-x& \cX}
\end{equation}
with $x:BA \to \cX$ a representable $B$-morphism for some finite abelian group $A$, where $BA$ denotes the classifying stack of $A$ over $\Sp (k)$.
\end{defn}

\begin{rem}\label{rem:tilde}
An abelian contraction of $f: \mls C \to \cX$ defines canonical $k$-points $\tilde x: \Sp(k) \to \cX$, $\Sp(k) \to X$, and a canonical $k$-group  $G_{\tilde x}$ as follows. The canonical $k$-point $\tilde x$ is the composition $\Sp (k)\rightarrow BA\xrightarrow{x} \mc X$ and $G_{\tilde x}$ is the automorphism group  of this composition. The $k$-point $\Sp(k) \to X$ of the coarse space is the image of $\tilde x$ and it is the unique $k$-point through which the induced map of coarse spaces $C \to X$ factors.
\end{rem}
 
\begin{rem}\label{rmk:contraction subgroup}
If $f: \mls C \to \mc X$ has an abelian contraction with associated $k$-point $\tilde x \in \mc X(k)$, then there is a factorization
\begin{equation}\label{eq:factormore}
\mls C \to BA \to BG_{\tilde x} \to \mc X
\end{equation}
of $f$. In particular, the map $\mls C \to BG_{\tilde x}$ corresponds to a $G_{\tilde x}$-torsor $P$ over $\mls C$ and the factorization of $f$ through $BA$ corresponds to a reduction of $P$ to an $A$-torsor. To see the factorization \eqref{eq:factormore}, note that since $BA$ is reduced the map $f$ factors through the maximal reduced closed substack $\mls G\subset \cX \times _{X }\Sp (k)$ where the map $\Sp(k) \to X$ is the image of $\tilde x$. The stack $\mls G \to \Sp(k)$ is a gerbe \cite[\href{https://stacks.math.columbia.edu/tag/06QK}{Tag 06QK}]{stacks-project} and  has a section induced by $\tilde x$. It follows from \cite[\href{https://stacks.math.columbia.edu/tag/06QG}{Tag 06QG}]{stacks-project} that $\mls G$ is canonically isomorphic to $BG_{\tilde x}.$ 
\end{rem}

\begin{rem}
\label{rmk:has contraction}
If $f: \mls C \to \cX$ factors through a closed embedding $\cX' \to \cX$, then $f$ has an abelian contraction if and only if the induced map $\mls C \to \cX'$ does. This is because the only nonempty closed substack of $BA$ is itself, so if $\mls C \to BA \to \cX$ is a factorization, then the pullback of $BA$ to $\cX'$ is $BA$ again. 
\end{rem}

\label{P:2.11} We now prove several properties of abelian contractions. First, existence of abelian contractions does not depend on the base field.

\begin{lem}\label{lem:abelian gerbe}
Let $\mls C$ be an algebraic stack of finite type over a $k$-point of $B$  and let $f: \mls C \to \cX$ be a morphism over $B$.  Let $k \subset k'$ be an extension of fields and let $f_{k'}: \mls C_{k'} \to \cX$ be the basechange to $\Spec(k')$. 
\begin{enumerate}
\item [(i)] If $k$ is algebraically closed then $f_{k'}$ has an abelian contraction  if and only if $f$ does.
\item [(ii)] If $k'$ is the algebraic closure of $k$, then if $f_{k'}$ has an abelian contraction there is a finite subextension $k''$ of $k'/k$ such that $f_{k''}$ has an abelian contraction.
\end{enumerate}
\end{lem}
\begin{proof}
Note that the ``if'' direction of part (i) is immediate. For the ``only if'' direction, suppose $f_{k'}$ factors through a representable morphism $BA_{k'} \to \cX$ for some abelian group  $A_{k'}$. 
The field $k'$ can be written as the colimit of its finitely generated sub-$k$-algebras, so by \cite[B.2 and B.3]{Rydh2} there is a finitely generated $k$-algebra $R \subset k'$ such that there is an abelian group scheme $A_R$ over $\Sp(R)$ such that $f_R: \mls C_R \to \cX$ factors through a representable morphism $BA_R \to \cX$. Since $k$ is algebraically closed and $R$ is finitely generated there exists a $k$-point $R \to k$. Base changing along this point yields the desired factorization over $k$.

The proof of (ii) is similar, except that we write $k'$ as the union of subextensions of $k$ that are finite over $k$ (as modules). In this case there is no need to find a section of the extension. 
\end{proof}

 We now discuss abelian contractions in a family setting. If $\mls C$ is a tame Deligne-Mumford stack over a base scheme $S$ and $\bar x: \Sp(k) \to C$ is a geometric point of the coarse space, let $\mls G_{\bar x}$ be the maximal reduced closed substack of the fiber product $\mls C\times_C \Sp(k)$.

\begin{lem}\label{L:factorlem}  Let $f:\mls C\rightarrow [U/G]$ be a morphism of tame Deligne-Mumford stacks over $S$, where $G$ is a finite group acting on a scheme $U$, and  let  $\bar y\rightarrow  U$  be a geometric point whose image in $U$ is fixed by $G$. Let $\bar x \to C$ be a geometric point mapping to the image of $\bar y$ in the coarse space of $[U/G]$. Finally, let $A\subset G$ be a subgroup.

(i) If $g_{\bar x}:\mls G_{\bar x}\rightarrow [U/A]$ is a factorization of $f_{\bar x}:\mls G_{\bar x}\rightarrow [U/G]$ through the quotient map $[U/A]\rightarrow [U/G]$ then $g_{\bar x}$ lifts uniquely to a factorization morphism $g_{(\bar x)}:\mls C_{(\bar x)}\rightarrow [U/A]$ of the restriction of $f$ to $\mls C_{(\bar x)}:= \mls C\times _C\Sp (\mls O_{C, \bar x})$.

(ii) If $g_{\bar x}:\mls G_{\bar x}\rightarrow [U/A]$ is a factorization of $f_{\bar x}:\mls G_{\bar x}\rightarrow [U/G]$ through the quotient map $[U/A]\rightarrow [U/G]$ then there exists an \'etale neighborhood $V\rightarrow C$ of $\bar x$ such that the base change $f_V:\mls C_V\rightarrow \mc X$ factors through $[U/A]$. 
\end{lem}
\begin{proof}
    Statement (ii) follows from (i) and a standard limit argument.

To see statement (i) note that the morphism $[U/A]\rightarrow [U/G]$ is finite \'etale.
Let $Z\rightarrow \mls C_{(\bar x)}$ denote the fiber product $\mls C_{(\bar x)}\times _{[U/G]}[U/A],$ and let $Z'\subset Z$ be the connected substack containing the image of the map $\mls G_{\bar x}\rightarrow Z$ defined by $g_{\bar x}$.  It then suffices to verify that $Z'\rightarrow \mls C_{(\bar x)}$ is an isomorphism.  This follows from noting that  by the same argument as in \cite[Theorem 3.1]{MR0268188} the functor from finite \'etale algebras over $\mls C_{(\bar x)}$ to finite \'etale algebras over $\mls G_{\bar x}$ is an equivalence of categories.
\end{proof}

Let $\mls C \to \mc X$ be a morphism of tame Deligne-Mumford stacks over $B$ such that the induced map of coarse spaces $C \to X$ factors as $C \to D \to X$ for some algebraic space $D$. If $\bar x: \Sp(k) \to D$ is a geometric point let $\mls C_{\bar x}$ denote the maximal reduced closed substack of $\mls C \times_D \Sp(k)$.

\begin{cor}\label{C:openprestable}
Let $\bar x \to D$ be a geometric point such that $\mls C_{\bar x} \to \mc X$ admits an abelian contraction. Then there exists an open neighborhood $V \subset D$ of the image of $\bar x$ such that for each geometric point $\bar y \to V$ the induced map $\mls C_{\bar y} \to \mc X$ has an abelian contraction.
\end{cor}
\begin{proof}
We may replace $B$ by the coarse space $X$ and $\mc X$ by its pullback $\mc X \times_X D$, so $X=D=B$. Moreover we may replace $D$ by any \'etale neighborhood of $\bar x$, so we can assume $\mc X = [U/G]$ with $G$ a finite group acting on $U$ fixing a point mapping to the image of $\bar x$ in the coarse space of $[U/G]$. Since $\mls C_{\bar x} \to [U/G]$ has an abelian contraction  there is an abelian subgroup  $A \subset G$ such that $\mls C_{\bar x} \to [U/G]$ factors through $BA$.  Hence $\mls C_{\bar x} \to [U/G]$ factors through $[U/A] \to [U/G].$

By \ref{L:factorlem} we may, after replacing $D$ by an \'etale neighborhood, assume $\mls C \to [U/G]$ factors through $[U/A] \to [U/G]$.
If $\bar y:\Sp(K) \to D$ is any geometric point, the map $\mls C_{\bar y} \to [U/G]$ factors through the maximal reduced closed substack of $[U/A] \times_D \Sp(K)$. If $\tilde y: \Sp(K) \to U$ lifts $\bar y$ 
then the maximal reduced substack of $[U/A] \times_D \bar y$ is isomorphic to $B\widetilde{A}$, where $\widetilde{A}$ is the automorphism group scheme of $\tilde y$ and a subgroup scheme of the abelian group  $A$, viewed as a group scheme over $K$. Hence we have an abelian contraction of $\mls C_{\bar y} \to [U/G]$.
\end{proof}

We now apply the discussion to contractions of curves. 
Let $\bC$ be an $n$-marked  generalized log twisted curve over a $B$-scheme $S$ and let $\mls C/S$ be the associated stack (see \cite[\S 3]{OWI}).  Let $q:C\rightarrow D$ be a contraction of $n$-marked prestable curves over $S$ (meaning that $\mls O_D \to Rq_*\mls O_C$ is an isomorphism; see \cite[4.2]{OWI}). Assume $\mls C \to \mc X$ is a morphism over $B$ such that the induced map of coarse spaces $C \to X$ factors as $C \xrightarrow{q} D \to X$. As before if $\bar x: \Sp(k) \to D$ is a geometric point then $\mls C_{\bar x}$ is the maximal reduced closed substack of $\mls C \times_D \Sp(k).$

\begin{lem}\label{L:2.14b} Let $\bar s\rightarrow S$ be a geometric point such that for every geometric point $\bar x\rightarrow D_{\bar s}$ in the smooth locus the restriction of $f$ to $\mls C_{\bar x}$ admits an abelian contraction.  Then there exists a Zariski neighborhood $U\rightarrow S$ of $\bar s$ 
such that for every geometric point $\bar x\rightarrow D_U$ of the smooth locus the restriction of $f$ to $\mls C_{\bar x}$ admits an abelian contraction. 
\end{lem}
\begin{proof}
We may replace $S$ with an \'etale neighborhood of $\bar s$. By doing so and adding markings to $\mathbf{C}$ we can assume $q$ is an isomorphism away from nodes of $D$ and the induced markings $t_1, \ldots, t_n: S \to D$ of $D$, where $t_j = q\circ s_j$. By \ref{C:openprestable}, for each marked point of $D_{\bar s}$ there is a neighborhood $V_j \subset D$ of the $j$th marked point such that for each geometric point $\bar x \to V_j$, the induced map $\mls C_{\bar x} \to \mc X$ has an abelian contraction. Let $U_j$ denote the (open) image of $V_j$ in $S$. By further shrinking $U_j$ we can assume $V_j$ contains the image of $t_j|_{U_j}$. Let $U$ be the intersection of the $U_j$. Since fibers of $\mls C \to D$ are representable at smooth unmarked points of $D$, the subscheme $U$ has the required properties.
\end{proof}

\subsection{Torsors on stacky $\mathbf{P}^1$}\label{S:section8}

Our discussion of contracting maps in Section \ref{S:mapcontract} will require properties of torsors on stacky rational curves.
Throughout this subsection we work over an algebraically closed field $k$.

Let $(\mathbf{P}^1, \{s_i\}_{i=1}^n, \mls N)$ be a generalized log twisted curve with underlying coarse curve $\mathbf{P}^1 = \mathbf{P}^1_k$ and assume that $s_n$ is distinct from the other points (but $s_1, \dots, s_{n-1}$ can coincide).  Let $\mls P\rightarrow \mathbf{P}^1$ be the associated stack.  Changing coordinates on $\mathbf{P}^1$ if necessary, we may assume that $s_n$ maps to $\infty \in \mathbf{P}^1$. So over $\infty $ the stack $\mls P$ is given by the $m$-th root construction for some integer $m>0$.  

Using \cite[2.22]{OWI} we see that there exists an admissible submonoid $N\subset \mathbf{Q}^{n-1}_{\geq 0}$ defining $\mls N|_{\mathbf{A}^1}$.  We let $\mls N'$ be the admissible subsheaf of $\Q^{n-1}_{\geq 0}$ on $\mathbf{P}^1$ induced by $N$ and we let $\mls P'$ be the stack associated to $(\mathbf{P}^1, \{s_i\}_{i=1}^{n-1}, \mls N')$, so there is a morphism $\pi: \mls P \to \mls P'$ which is an isomorphism away from $\infty$, and in a neighborhood of $\infty$ is given by the $m$th root construction at that point. 

We set $X := N^{gp}/\mathbf{Z}^{n-1}$ and note that summation $\mathbf{Q}^{n-1} \to \mathbf{Q}$ induces a homomorphism
\begin{equation}\label{E:chidef}
\chi: X \to \mathbf{Q}/\mathbf{Z}.
\end{equation}
We set $X_m = \chi^{-1}(\frac{1}{m}\mathbf{Z}/\mathbf{Z})$.

\subsection{Torsion line bundles on $\mls P$}

If $\mls X$ is a stack, let $\mls Pic(\mls X)$ denote its Picard category and $\mls Pic(\mls X)_{\tors}$ the full subcategory on torsion line bundles. Note that the of isomorphism classes of  $\mls Pic(BD(X))$ is equal to $X$ and that $\mls Pic(BD(X_m))$  is a full subcategory of $\mls Pic(BD(X))$. Moreover, by \cite[3.8]{OWI} the stack $BD(X)$ is the product of the residue gerbes at the distinct stacky points of $\mls P$, and hence there is a natural morphism $i: BD(X) \to \mls P$ (this is independent of $m$). 

\begin{lem}\label{L:pic}
The restriction functor $i^*: \mls Pic(\mls P)_{\tors} \to \mls Pic(BD(X))$ is fully faithful with essential image $\mls Pic(BD(X_m))$.
\end{lem}

Before proving the lemma, we investigate the map $\chi$ used to define $X_m$. We begin with the following observation.

\begin{lem}\label{lem:pic-for-diag} Let $\pi :\mls X\to X$ be a morphism of tame Deligne-Mumford stacks that \'etale locally on $X$ is a coarse moduli morphism.  

(i) If $\mc L$ is a line bundle on $\mls X$ such that the character at every stacky point of $\mls X$ factors through a character of the isotropy group of the corresponding point of $X$, then $\pi _*\mc L$ is a line bundle on $X$ and the canonical map $\pi^*\pi_* \mc L \to \mc L$ is an isomorphism.

(ii) If $\mc M$ is a line bundle on $X$ then the adjunction map $\mc M\rightarrow \pi _*\pi ^*\mc M$ is an isomorphism.
\end{lem}
\begin{proof}
    The assertion is \'etale local on $X$ so it suffices to consider the case when $X$ is a scheme and $\pi $ is a coarse moduli space morphism.  In this case the result follows from \cite[6.1]{integral}. 
\end{proof}

If $\mc L$ is a line bundle on $\mls P'$, we define its \textit{degree} as follows: let $M$ be an integer such that $\mc L^{\otimes M}$ descends to a line bundle on $\mathbf{P}^1$, and set
$\deg_{\mls P'} \mc L = (1/M)\deg_{\mathbf{P}^1} \mc L^{\otimes M}.$

Note that for the morphism $\pi:\mls P \to \mls P'$, pushforward sends line bundles to line bundles: this is true away from $\infty$ because $\pi$ is an isomorphism there, and it holds in a neighborhood of $\infty$ because here $\pi$ is given by an $m$th root construction (see \cite[3.1.1, 3.1.2]{cadman} for a description of the behavior of $\pi_*$ for $m$th root constructions).

\begin{lem}\label{L:chi}
The composition
\[
\mls Pic(\mls P)_{\tors} \to \mls Pic (BD(X)) \to X \xrightarrow{\chi} \mathbf{Q}/\mathbf{Z}
\]
sends a line bundle $\mc L$ on $\mls P$ to the class of $(-1)\cdot\deg_{\mls P'} \pi_*\mc L.$
\end{lem}
\begin{proof}
By definition of $\mls P'$ (see \cite[3.6]{OWI}), we have a symmetric monoidal functor $N \to \mls Div^+(\mls P')$, where $\mls Div^+(\mls P')$ is the groupoid of morphisms $\mc L \to \mls O_{\mls P'}$ where $\mc L$ is a line bundle on $\mls P'$. For $\ell \in N$ we write $\mc L_{\ell}$ for the associated line bundle and note that $\mc L_{\ell}|_{BD(X)} = \ell$. If $\ell \in N \subset \Q^{n-1}_{\geq 0}$ is a vector with one entry equal to 1 and the rest equal to zero, then $\mc L_\ell \simeq \mls O_{\mathbf{P}^1}(-1)$.

To prove the lemma, it is enough to show that if $\mc L$ is a line bundle on $\mls P'$, then 
$\chi(\mc L|_{BD(X)})$ is the class of $-\deg_{\mls P'} \mc L$. 
This is true when $\mc L= \mc L_\ell$: If $M$ is an integer such that $M\ell =(a_i)_{i=1}^{n-1} $ is in $\mathbf{Z}^{n-1}$, 
then $(\mc L_{\ell})^{\otimes M}$ is isomorphic to $\mls O_{\mathbf{P}^1}(-\sum_{i=1}^{n-1} a_i)$ and so
\[
-\deg_{\mls P'} \mc L_\ell = \frac{1}{M} \deg_{\mls P'} (\mls O_{\mathbf{P}^1}(\sum_{i=1}^{n-1} a_i)) = \sum_{i=1}^{n-1} a_i/M
\]
whose  class is $\chi(\ell).$ Finally, we claim that an arbitrary line bundle $\mc L$ on $\mls P'$ is of the form $\mls O_{\mathbf{P}^1}(d) \otimes \mc L_\ell$ for some integer $d$ and $\ell \in N$. Indeed, given $\mc L$, we define $\ell \in X$ to be the character associated to $\mc L|_{BD(X)}$. Then $\mc L \otimes (\mc L_\ell)^{-1}$ has trivial stabilizer action at each of the marked points of $\mls P'$ and hence is pulled back from $\mathbf{P}^1$; i.e., $\mc L \otimes (\mc L_\ell)^{-1}$ is isomorphic to $\mls O_{\mathbf{P}^1}(d)$ for some integer $d$.
\end{proof}

\begin{proof}[Proof of \ref{L:pic}]

We first show that the essential image of $i^*$ is in $\mls Pic(BD(X_m))$. Let $\mc L$ be a torsion line bundle on $\mls P$. Then $\mc L^{\otimes m}$ is also torsion, say $\mc L^{\otimes mM} \simeq \mls O_{\mls P}$. By \ref{lem:pic-for-diag}, since the stabilizer action of $\mc L^{\otimes m}$ at $\infty$ is trivial, we have $\pi^*\pi_*(\mc L^{\otimes m}) \simeq \mc L^{\otimes m}$. Hence we have $\pi^*(\pi_*(\mc L^{\otimes m}))^{\otimes M} \simeq \mls O_{\mls P}$ and since $\pi^*$ is fully faithful by \ref{lem:pic-for-diag} we have $\pi_*(\mc L^{\otimes m})$ is torsion. It follows that $\deg_{\mls P'}(\pi_*(\mc L^{\otimes m}))=0$ and by \ref{L:chi} the image of $\mc L^{\otimes m}$ in $\mathbf{Q}/\mathbf{Z}$ is zero. Hence the image of $\mc L$ is in $\frac{1}{m} \mathbf{Z}/\mathbf{Z}$ as required.

We next show that the essential image of $i^*$ is equal to $\mls Pic(BD(X_m))$. Let $\ell = (a_1, \ldots, a_{n-1}) \in N$ be an element mapping to $X_m$ and let $\mc L_\ell$ be the associated line bundle on $\mls P'$ (see the proof of \ref{L:chi}). Write $m\sum_{i=1}^{n-1}a_i = md + s$ for unique integers $d, s$ with $s \in [0, m-1]$. Let $\mls O_{\mls P}(\infty/m)$ be the ideal sheaf of $B\mu_m \subset \mls P$, so $\mls O_{\mls P}(\infty/m)^{\otimes m}$ is isomorphic to $\mls O_{\mathbf{P}^1}(-1)$. Set $\mc N = \mls O_{\mathbf{P^1}}(d) \otimes \mc L_\ell \otimes \mls O_{\mls P}(\infty/m)^{s}$. Then we compute
\[
\mc N^{\otimes m} \simeq \mls O_{\mathbf{P}^1}(md) \otimes \mls O_{\mathbf{P}^1}(-md -s) \otimes \mls O_{\mathbf{P}^1}(s) \simeq \mls O_{\mathbf{P}^1}.
\]
So $\mc N$ is a torsion bundle on $\mls P$ mapping to $\mc L$ (since the characters of $\mls O_{\mathbf{P^1}}(d)$ and $\mls O_{\mls P}(\infty/m)^{s}$ at every stacky point of $\mls P$ away from $\infty$ are trivial).

To see that $i^*$ is injective on isomorphism classes, let $\mc L$ be a torsion line bundle on $\mls P$. If $i^*\mc L$ is trivial then the fiber of $\mc L$ has trivial stabilizer action at each of the stacky points of $\mls P$ not equal to $\infty$. So $\mc L$ is the pullback of a line bundle $\mc L''$ on a stacky $\mathbf{P}^1$ given by taking an $m$th root stack at a single point. Since this pullback is fully faithful by \ref{lem:pic-for-diag}, the bundle $\mc L''$ is also torsion and hence trivial (since the Picard group of this curve is $\mathbf{Z}$). So $\mc L$ is trivial.

Finally, $i^*$ induces an isomorphism on automorphisms of any object. The automorphisms of a line bundle on $\mls P$ (resp. on $BD(X)$) are sections of $\mathbf{G}_{m, \mls P}$ (resp. of $\mathbf{G}_{m, BD(X)}$), but these are just the constants $k^\times$ since global sections of $\mls O_{\mls P}$ (resp. $\mls O_{BD(X)}$) can be computed after pushforward to the coarse space. 

\end{proof}

\subsection{Torsors on $\mls P$}\label{S:generalcase}
From \ref{L:pic} we have an equivalence of categories
\begin{equation}\label{eq:equiv}
\mls Pic(\mls P)_{\tors} \to \mls Pic(BD(X_m)).
\end{equation}
A choice of inverse equivalence defines a morphism 
\[
\rho: \mls P \to BD(X_m)
\]for which the restriction functor on Picard groups is inverse to \eqref{eq:equiv}.  Let $G$ be a finite group and let $\Tors ^G(\mls P)$ (resp. $\Tors ^G(BD(X))$) denote the category of $G$-torsors over $\mls P$ (resp. $BD(X)$). Let $\Tors ^{G, \ab}(\mls P)$ (resp. $\Tors ^{G, \ab}(BD(X))$) denote the full subcategory of torsors whose classifying map to $BG$ admits an abelian contraction.

\begin{prop}\label{P:3.15} The pullback functor
$$
\rho^*: \Tors ^{G}(BD(X_m))\rightarrow \Tors ^{G, \ab}(\mls P)
$$
is an equivalence of categories.
\end{prop}
\begin{proof}
If $A\subset G$ is an abelian subgroup then the category of $A$-torsors on $\mls P$ is equivalent to the groupoid of morphisms of Picard categories $Y\rightarrow \mls Pic(\mls P)_{\text{tors}}$, but $\mls Pic(\mls P)_{\text{tors}}\simeq \mls Pic (BD(X_m))$ by \ref{L:pic}. Hence $\rho^*$ is essentially surjective.

For the full faithfulness, note that a $G$-torsor is given by a vector bundle $\mc V$ with various maps between tensor products of $\mc V$ with itself and $\mls O_G$ (the vector bundle $\mc V$ is the coordinate ring of the torsor, its ring structure is given by a section of $\mc V$ (the unit) and multiplication $\mc V \otimes \mc V \to \mc V$ making certain diagrams commute, and the $G$-action is given by a map $\mc V \to \mc V \otimes \mls O_G$ (making more diagrams commute). A morphism of $G$-torsors is given by a morphism of the associated vector bundles that commutes with morphisms described. It follows that to prove that $\rho^*$ is fully faithful for $G$-torsors it suffices to show that it is fully faithful for vector bundles, so it suffices to show that if $\mc V$ is a vector bundle on $BD(X_m)$ then $\mc V\rightarrow \rho _*\rho ^*\mc V$ is an isomorphism. Since any vector bundle on $BD(X_m)$ is a direct sum of line bundles it furthermore suffices to prove this for a line bundle $\mc L_x$ for $x\in X_m$.  Now the pullback of $\rho _*\rho ^*\mc L_x$ along the covering $\Sp (k)\rightarrow BD(X_m)$ is given by $H^0(P_0, \rho ^*\mc L_x|_{P_0})\simeq H^0(\mls P, \oplus _{y\in X_m}(\mc L_y\otimes \mc L_x)).$  By \ref{P:10.16} below this is the one-dimensional $k$-vector space given by $\mc L_x|_{\Sp (k)}$. 
\end{proof}

\begin{lem}\label{P:10.16} Let $\ell\in X_m$ be a nonzero element and let $\mc L_\ell$ be the line bundle on $\mls P$ given by $\rho^*$ applied to the bundle on $BD(X_m)$ corresponding to $\ell$.  Then $H^0(\mls P, \mc L_\ell) = 0$.
\end{lem}
\begin{proof}
Consider the projection $\pi :\mls P\rightarrow \mathbf{P}^1$. Suppose $H^0(\mls P, \mc L_\ell) \neq 0$, and let $s \in H^0(\mls P, \mc L_\ell)$ be a section that is not identically zero. Then $s^{\otimes m} \in H^0(\mls P, \mc L_\ell^{\otimes m})$ is not identically zero, but $\mc L_\ell^{\otimes m}$ is trivial, so since $ H^0(\mls P, \mls O_{\mls P}) = H^0(\mathbf{P}^1, \pi _*\mls O_{\mls P}) = k$ we see $s^{\otimes m}$ is nowhere zero. It follows that $s$ is nowhere zero and hence defines a trivialization of $\mc L_\ell$. Since $\rho^*$ induces an equivalence inverse to \eqref{eq:equiv}, we have that $\ell \in X_m$ is trivial.
\end{proof}

\subsection{Torsors on $\mls P$: A special case}

\begin{pg} We now specialize \ref{S:generalcase} to the case when $s_1 =s_2 =\cdots =s_{n-1}$ (but still $s_1\neq s_n$). The morphism $i:BD(X) \to \mls P$ is now the residual gerbe at the $n-1$ coincident markings. As before let $G$ be a finite group.
\end{pg}

\begin{prop}\label{P:newabelian} Any object $T$ of $\Tors ^G(\mls P)$ admits an abelian contraction.
\end{prop}
\begin{proof}
Let $T\rightarrow \mls P$ be a $G$-torsor and let $A\subset G$ be an abelian subgroup such that the pullback of $T$ to $BD(X)$ is the pushout of an $A$-torsor over $BD(X)$ (for example $A$ could be the image of $D(X)$ in $G$).  Consider the quotient $\overline T:= T/A$, which is \'etale over $\mls P$, and let $\mls P^{\prime \prime }\rightarrow \mathbf{P}^1$ be the stack obtained by taking the $m$-th root of $\infty $ (and no other stack structure). The pullback of $\overline T$ to $BD(X)$ is isomorphic to $G/A$, since $T|_{BD(X)}$ is a pushout from $A$, and therefore the stabilizer action of $D(X)$ on the fiber of $\mls O_{\overline T}$ is trivial.  It follows that $\overline T$ descends to an \'etale morphism $\overline T^{\prime \prime }\rightarrow \mls P^{\prime \prime }$.  Since $\mls P^{\prime \prime }$ is normal so is $\overline T^{\prime \prime }$, and the coarse moduli space of $\overline T^{\prime \prime }$ is a curve over $\mathbf{P}^1$ unramified everywhere except possibly at $\infty $.  It follows that the coarse space of $\overline T^{\prime \prime }$ is a disjoint union of copies of $\mathbf{P}^1$ from which it follows that $\overline T^{\prime \prime }$ is a disjoint union of copies of $\mls P''$.   We conclude that $\overline T^{\prime \prime }\simeq G/A\times \mls P^{\prime \prime }$ and therefore $\overline T \simeq G/A \times \mls P$. Fixing such a trivialization of $\overline T$ we obtain an $A$-torsor $T_0\rightarrow \mls P$ whose pushout is $T$.  In particular, $T$ admits an abelian contraction. 
\end{proof}

\section{Stable maps}\label{S:section5b}
Let $\mc X$ be a separated tame Deligne-Mumford stack of finite type over a noetherian base $B$. We now turn to the problem of defining stability for maps from stacks $\mls C$ associated to generalized log twisted curves to $\mc X$. Roughly speaking, we would like stability of a map $\mls C \to \mc X$ to be determined by the underlying map of coarse spaces, but when $\mc X$ has nonabelian stabilizers this does not yield a proper stack of maps (due to the fact that our curves $\mls C$ have only abelian stabilizers). We use the notion of an abelian contraction in our definition of stability to address this issue.

\subsection{Review of the case when $\mc X$ is an algebraic space}\label{S:review}
When $\mc X = X$ is an algebraic space, weighted stable maps from representable curves to $X$ were defined in \cite{AlexeevGuy, bayermanin}, extending the notion of weighted stable curves in \cite{Hassett}. In this section we recall that theory.
Fix a nonnegative integer $g$ and numbers $a_1, \ldots, a_n \in (0, 1]\cap \mathbf{Q}$. 
We refer to the $a_i$ as the \emph{weights}.


\begin{pg}\label{P:2.1}

For a $B$-scheme $S$ a
  \emph{prestable map to $X$ of type $(g, \ba)$ over $S$} is a collection of data
\begin{equation}\label{E:prestablemap}
(\,C\rightarrow S, \;\{s_i:S\rightarrow C\}_{i=1}^n, \;f:C\rightarrow X\,),
\end{equation}
where
\begin{enumerate}
\item [(i)] $(C/S, \{s_i\}_{i=1}^n)$ is an $n$-marked prestable curve over $S$ of genus $g$.
\item [(ii)] $f$ is a morphism of $B$-schemes.
\item [(iii)] For every geometric point $\bar s\rightarrow S$ and $x\in C_{\bar s}$ we have
$$
\sum _{i, \, s_i(\bar s) = x}a_i\leq 1.
$$
\end{enumerate}
An isomorphism of prestable maps of type $(g, \ba)$ over $S$ is an isomorphism of the associated $n$-marked prestable curves over $S$ that commutes with the morphisms to $X$.
A prestable map to $X$ over $S$ of type $(g, \ba)$ is \emph{stable} if, in addition to (i)-(iii), the following condition holds for every geometric point $\bar s\rightarrow S$:
\begin{enumerate}
\item [(iv)] If $E\subset C_{\bar s}$ is an irreducible component such that $f_{\bar s}(E)$ has dimension zero, then either the normalization $\widetilde E$ of $E$ has positive genus, or
\begin{equation}\label{eq:numerics}
\#\{e\in \widetilde E \mid \text{$e$ maps to a node of $C_{\bar s}$}\}+\sum _{i, s_i(\bar s)\in E}a_i>2.
\end{equation}
\end{enumerate}
A \textit{(pre)stable curve} of type $(g, \ba)$ is a (pre)stable map of type $(g, \ba)$ to the target $\mc X = B$. 

We let $\mls K_{g, \ba}(X)$ denote the stack over $B$ whose fiber over $S\rightarrow B$ is the groupoid of stable maps of type $(g, \ba)$
$$
(\,C\rightarrow S,\; \{s_i:S\rightarrow C\}_{i=1}^n,\; f:C\rightarrow X_S\,)
$$
over $S$.  
In the case when $X/B$ is projective and $\mls O_X(1)$ is an ample invertible sheaf on $X$ we can further consider the degree of $f^*\mls O_X(1)$.  We let 
\[\mls K_{g, \ba}(X, d)\subset \mls K_{g, \ba}(X)\]
denote the substack of  stable maps for which $f^*\mls O_X(1)$ has degree $d$ in every fiber, and refer to such a stable map as a \emph{stable map of type $(g, \ba)$ and degree $d$}.
\end{pg}

\begin{thm}[{\cite[1.9]{AlexeevGuy} and \cite[1.1.4]{bayermanin}}]\label{T:2.4} 
Let $(X, \mls O_X(1))$ be a projective scheme over $B$ with a relatively ample invertible sheaf.  

(i) $\mls K_{g, \ba}(X)$ is an Artin   stack locally of finite type over $B$ with finite diagonal.

(ii) $\mls K_{g, \ba}(X, d)$ is an open and closed substack of $\mls K_{g, \ba}(X)$ and 
$$
\mls K_{g, \ba}(X) = \coprod _d\mls K_{g, \ba}(X, d).
$$

(iii) For all $d$ the stack $\mls K_{g, \ba}(X, d)$ is a proper algebraic  stack over $B$ with finite diagonal and admitting a presentation as a global quotient stack. 
\end{thm}

\qed


\begin{rem}
Using the method of \cite{AGOT} one can show that the coarse space $K_{g, \ba}(X, d)$ of $\mls K_{g, \ba}(X, d)$ is a projective scheme.
\end{rem}

\begin{rem}
In order for $\mls K_{g, \ba}(X, d)$ to be nonempty, we must have either $d \neq 0$ or $2g - 2 + \sum_{i=1}^n a_i >0.$
\end{rem}

\subsection{Prestable maps with target a stack}\label{S:section5}

Let $\mc X$ be a separated tame Deligne-Mumford stack of finite type over a noetherian scheme $B$.
As before, fix discrete data $(g, \ba)$ consisting of a nonnegative integer $g$ and weights $a_1, \ldots, a_n \in (0, 1]\cap \mathbf{Q}$.

\begin{defn}\label{def:weighted twisted prestable map}
A \emph{prestable map to $\cX/B$ of type $(g, \ba)$}  over a $B$-scheme $S$ is the data $(\bC, f)$, where
\begin{enumerate}
\item[(i)] $\bC$ is a generalized log twisted curve over $S$ of genus $g$, with underlying marked prestable curve $(C/S, \{s_i\}_{i=1}^n)$;
\item[(ii)] $f: \mls C \to \cX$ is a morphism of stacks over $B$, where $\mls C$ is the stack associated to $\bC$ (see \cite[\S 3]{OWI});
\item[(iii)] For every geometric point $\bar s \in S$ and $x \in C_{\bar s}$ we have
$$
\sum _{i,\, s_i(\bar s) = x}a_i\leq 1.
$$
\end{enumerate}
\end{defn}

An isomorphism of twisted prestable maps of type $(g, \ba)$ over $S$ is an isomorphism of generalized twisted curves over $S$ such that the induced isomorphism of associated stacky curves commutes with the given morphisms to $\cX$.

\begin{lem}\label{L:prestableopen} Let $\mathbf{C}$ be an $n$-marked  generalized log twisted curve over a scheme $S$ and let $f:\mls C\rightarrow \mc X$ be a morphism of stacks.  Then there exists an open subset $U\subset S$ such that a morphism of schemes $T\rightarrow S$ factors through $U$ if and only if the base change $(\bC _T, f_T)$ of $(\bC , f)$ to $T$ is a prestable map of type $(g, \ba )$.
\end{lem}
\begin{proof} 
Condition (iii) is evidently an open condition.


\end{proof}

\subsection{Stability}



\begin{notation}\label{not:component stack}
Let $\bC$ be a generalized log twisted curve over a field $k$ with coarse space $C$ and associated stack $\mls C$. If $E \subset C$ is any irreducible component, we define the \textit{stack associated to E} to be the reduced closed substack $\mls E \subset \mls C$ whose underlying set is the closure in $|\mls C|$ of the generic point of $E$ (see \cite[\href{https://stacks.math.columbia.edu/tag/0509}{Tag 0509}]{stacks-project}). Observe that the coarse space of $\mls E$ is $E$. However, $\mls E$ is not equal to the pullback of $\mls C$ to $E$, since the later may not be reduced. 
\end{notation}

\begin{defn}\label{def:weighted twisted stable map}
A prestable map $(\bC, f)$ to $\cX/B$ of type $(g, \ba)$ over a $B$-scheme $S$ is \emph{stable}  if $f$ is representable and for every geometric point $\bar s\rightarrow S$ the following additional conditions hold:
\begin{enumerate}
\item[(iv)] 
The automorphism group scheme of the underlying prestable map $(C_{\bar s}, \{s_i\}_{i=1}^n, f)$ is finite.
\item[(v)] If $E \subset C_{\bar s}$ is an irreducible component such that $\mls E \to \cX$ has an abelian contraction, then either the normalization $\widetilde E$ of $E$ has positive genus, or
\[
\#\{e\in \widetilde E \mid \text{$e$ maps to a node of $C_{\bar s}$}\}+\sum _{i, s_i(\bar s)\in E}a_i>2.
\]
\end{enumerate}
\end{defn}

\begin{rem}\label{rem:finite auts}
(a) Condition (iv) in \ref{def:weighted twisted prestable map} is equivalent to requiring that for every irreducible component $E$ of $C_{\bar s}$ whose image in $X$ is zero-dimensional, $E$ either has positive genus or contains at least three special points (preimages of nodes and distinct markings). 

(b) Condition (v) is roughly saying that we only require stability of the underlying coarse moduli map on components that have an abelian contraction. In particular, the map on coarse spaces $C\rightarrow X$ need not be a stable map in general.
\end{rem}

\begin{rem} If the target stack $\mc X$ has abelian stabilizer group schemes, then a twisted prestable map $(\mathbf{C}, f)$ of type $(g, \ba)$ is \emph{stable} if and only if $f$ is representable and the induced map on coarse spaces
$$
C\rightarrow X
$$
is a stable map of type $(g, \ba)$ in the sense of \ref{P:2.1}.
 \end{rem}

\begin{lem}\label{L:5.13}
Let $k \subset k'$ be an extension of  algebraically closed fields, and let $(\bC, f)$ be a prestable map of type $(g, \ba)$ over $k$ such that $f$ is also representable. Then $(\bC, f)$ is stable if and only if its base change to $k'$ is stable.
\end{lem}
\begin{proof}
By \cite[\href{https://stacks.math.columbia.edu/tag/038H}{Tag 038H}]{stacks-project} and \cite[\href{https://stacks.math.columbia.edu/tag/0C56}{Tag 0C56}]{stacks-project} the morphism $C_{k'} \to C$ induces bijections of irreducible components and nodes. The result follows from this, \ref{rem:finite auts}, and \ref{lem:abelian gerbe}.
\end{proof}

\subsection{An alternate characterization of stability}
\label{S:alternate}
Let $B = \Sp (k)$ be the spectrum of an algebraically closed field and fix a prestable map $(\mathbf{C}, f)$ of type $(g, \ba)$ over $k$.  The induced map of coarse spaces $\bar f:C\rightarrow X$ is then a prestable map of type $(g, \ba)$ and by \cite[4.11]{OWI} admits a unique contraction 
$$
\xymatrix{
C\ar@/^2pc/[rr]^-{\bar f}\ar[r]^-q& C'\ar[r]^-g& X,}
$$
with $(C', \{s_i':=q\circ s_i\}_i, g)$ a stable map of type $(g, \ba)$.  Let $x'\in C'$ be the image of one of the marked points.  Then the preimage $q^{-1}(x')$ is a tree of rational curves with some marked points in the smooth locus.  We say that a rational curve $E\subset q^{-1}(x')$ is an \emph{interior branch} if it contains two nodes; it is an \emph{extremal branch} if it contains one node.

\begin{pg}\label{P:5.8}
For $x' \in C'$ let $G_{x'}$ denote the stabilizer group  of the point of $\mc X$ over $x'$, so that the fiber $\mc X_{x'}$ of $\mc X$ over $x'$ has maximal reduced substack isomorphic to $BG_{x'}$. Then if $E \subset q^{-1}(x')$ is a rational curve, the morphism $f$ restricts to a morphism $\mls E \to BG_{x'}$ since the stack $\mls E$ associated to $E$ is reduced. 
\end{pg}

 \begin{lem}\label{D:5.7} The prestable map $(\mathbf{C}, f)$ is stable if and only if $f$ is representable and 
 the following conditions hold:
 \begin{enumerate}
     \item [(i)] The map $q:C\rightarrow C'$ from $C$ to its stabilization is an isomorphism away from the marked points on $C'$. 
     \item [(ii)] For every marked point $x'\in C'$ and rational component $E\subset q^{-1}(x')$, we have:
     \begin{enumerate}
         \item [(a)]  If $E$ is an interior branch of $q^{-1}(x')$ then 
         $$
         \#\{\text{nodes on $E$}\}+\#\{\text{distinct marked points on $E$}\}>2.
         $$
         Note that $E$ has at least two nodes so this is equivalent to the assertion that either $E$ has three or more nodes or contains a marked point.
         \item [(b)] If $E$ is an extremal branch then  the map $\mls E \to \mc X$ does not have an abelian contraction.
     \end{enumerate}
 \end{enumerate}
\end{lem}
\begin{proof}
    First we show that if $(\mathbf{C}, f)$ satisfies (i) and (ii) then conditions (iv) and (v) in \ref{def:weighted twisted stable map} hold.  For this note that if $E$ is an interior branch then condition (ii) (a) implies that the automorphism group of $E$ is finite (see \ref{rem:finite auts}). 
If $E$ is an extremal branch with one distinct marked point (possibly the support of multiple $s_i$) then by \ref{P:newabelian} (i)  the map $\mls E \to \mc X$ has an abelian contraction, contradicting (ii)(b). So $E$ has at least two distinct marked points and hence a finite automorphism group.
    It follows that \ref{def:weighted twisted stable map} (iv) holds.
    To show that \ref{def:weighted twisted stable map} (v) holds let $E$ be a contracted component of $C$ for which $\mls E$ has an abelian contraction. By (ii) (b) we have that $E$ is not extremal, so $E$ is an interior branch and by (ii) (a) has at least three nodes or at least two nodes and one marking. It follows that \eqref{eq:numerics} holds.  Therefore the conditions in the lemma imply that $(\bC , f)$ is stable.
    
    Conversely if \ref{def:weighted twisted stable map} (iv) holds then 
    the preimage under $q$ of a non-marked smooth point is required to have trivial automorphism group.  Therefore (i) holds. 
 Furthermore any interior branch must have at least one marked point implying that (ii) (a) holds. Finally condition \ref{def:weighted twisted stable map} (v)  directly implies (ii) (b).
\end{proof}


 \begin{rem}\label{R:newcontraction}
 Condition (ii) (b) in \ref{D:5.7} has an equivalent formulation as follows. Let $E \subset q^{-1}(x')$ be an extremal branch, let $E^\circ \subset E$ be the complement of the node, and set $\mls E^\circ := E^\circ \times_E \mls E$. The map $\mls E \to BG_{x'}$ corresponds to a $G_{x'}$-torsor $Q \to \mls E$ and we let $Q^\circ$ be its restriction to $\mls E^\circ$. 
 The map $\mls E \to BG_{x'}$ factors through $BA \to BG_{x'}$ for some abelian subgroup $A \subset G_{x'}$ if and only if $Q$ is induced from an $A$-torsor. But if $Q^\circ \to \mls E^\circ$ is induced from an $A$-torsor $H^\circ \subset Q^\circ$, the closure $H \subset Q$ is an $A$-torsor with associated $G_{x'}$-torsor equal to $Q$. The fact that $H$ is, in fact, a torsor can be verified fppf locally on $\mls E$ when $Q$ is trivial. 
 Therefore condition (ii) (b) in \ref{D:5.7} is equivalent to the requirement that if $E\subset q^{-1}(x')$ is an extremal branch, the $G_{x'}$-torsor on $\mls E^\circ$ is not induced from an abelian subgroup $A \subset G_{x'}$. 
 \end{rem}

\section{The algebraic stack $\mls K_{g, \ba }(\cX)$}\label{S:section6b}

\begin{pg}
Let $\mc X$ be a separated tame Deligne-Mumford stack of finite type over a noetherian scheme $B$ and fix data $(g, \ba )$ consisting of a nonnegative integer $g$ and weights $a_i\in (0,1]\cap \mathbf{Q}$.
Let $\mls K_{g, \ba}(\cX)$ (resp. $\mls S_{g, \ba}(\cX)$) denote the fibered category over $B$ which to any $B$-scheme $T$ associates the groupoid of stable (resp. prestable) maps $(\mathbf{C}, f)$ to $\mc X$ over $T$ of type $(g, \ba )$. 
\end{pg}

The main result of this section is the following theorem.  

\begin{thm}\label{T:7.4} The stacks $\mls S_{g, \ba}(\mc X)$ and $\mls K _{g,  \ba }(\cX)$ are algebraic stacks locally of finite type and with quasi-compact and separated diagonals over $B$. 
\end{thm}
The proof of the theorem occupies the remainder of this section.  

\subsection{Reduction to openness of stability}
 Let $\fM _{g, n}^{glt}$ denote the stack of generalized log twisted curves defined in \cite{OWI} and let $\fC^{} \rightarrow \fM _{g, n}^{glt}$ be the stack associated to the universal generalized log twisted curve over $\fM _{g, n}^{glt}$.
Consider the base changes $\fM^{glt}_{g, n, B}$ and $\fC_B$ of $\fC$ and $\fM^{glt}_{g, n}$ to $B$ and the associated $\underline {\text{Hom}}$-stack
\[
\underline {\text{Hom}}_{\fM^{glt}_{g, n, B}}(\fC_B  , \mc X_{\fM^{glt}_{g, n, B}})
\]
classifying morphisms from $\fC_B $ to $\mc X$.
This $\underline {\text{Hom}}$-stack is algebraic and locally of finite presentation with quasi-compact and separated diagonal over $\fM _{g, n, B}^{glt}$ by \cite[Theorem C.2]{AOV}.
This stack contains the stacks $\mls S_{g, \ba}(\mc X)$ and $\mls K_{g, \ba}(\cX)$, and since $\fM_{g, n, B}^{glt}$ is locally of finite type with quasi-compact and separated diagonal over $B$ (by  \cite[2.31]{OWI}),  to prove \ref{T:7.4} it suffices 
to show that the inclusions \[\mls K _{g, \ba }(\cX)\hookrightarrow \mls S_{g, \ba}(\mc X) \hookrightarrow \underline {\text{Hom}}_{\fM _{g, n, B}^{glt}}(\fC_B , \mc X_{\fM _{ g, n, B}^{glt}})\] are representable by open immersions. 
By \ref{L:prestableopen} we know that the prestable locus is open, so this completes the proof for $\mls S_{g, \ba}(\mc X)$. Now, by \cite[1.6]{homstack} there is an open substack $\mls S_{g, \ba}^{\text{rep}} \subseteq \mls S_{g, \ba}$ of representable prestable maps, and the inclusion $\mls K_{g, \ba}(\mc X) \to \mls S_{g, \ba}$ factors through $\mls S_{g, \ba}^{\text{rep}}$. Thus, to complete the proof for $\mls K_{g, \ba}(\mc X)$ it suffices to show that if $(\mathbf{C}, f)$ is a representable prestable map to $\mc X/B$ of type $(g, \ba )$ over a $B$-scheme $S$ then there exists an open subscheme $S'\subset S$ such that the restriction of $(\mathbf{C}, f)$ to $S'$ is stable and any other morphism $T\rightarrow S$ for which the pullback of $(\mathbf{C}, f)$ to $T$ is stable factors through $S'$.  Since the notion of stability is defined on geometric fibers it therefore is enough to prove the following:

\begin{prop}\label{P:6.4b} Let $(\mathbf{C}, f)$ be a representable prestable map to $\mc X/B$ of type $(g, \ba )$ over a $B$-scheme $S$ and let $S'\subset S$ be the set of points $s\in S$ for which there exists a geometric point $\bar s\rightarrow S$ over $s$ for which the base change $(\mathbf{C}_{\bar s}, f_{\bar s})$ is stable.  Then $S'$ is an open subset of $S$.
\end{prop}

\begin{rem} Note that by \ref{L:5.13} if $s\in S$ is a point for which there exists a geometric point $\bar s\rightarrow s$ for which $(\mathbf{C}_{\bar s},f _{\bar s})$ is stable, then for every geometric point over $s$ the base change of $(\mathbf{C}, f)$ to that point is stable.
\end{rem}

To prove \ref{P:6.4b}, we let $S_1 \subset S$ (resp. $S_2 \subset S$) be the set of points $s \in S$ for which there exists a geometric point $\bar s \to S$ over $s$ for which the base change $(\bC_{\bar s}, f_{\bar s})$ satisfies condition (iv) (resp. condition (v)) in the definition \ref{def:weighted twisted stable map} of stability, so $S' = S_1 \cap S_2$. We show that both $S_1$ and $S_2$ are open. 

That $S_1$ is open follows from a general argument using only that the stack of representable maps $\Hom^{\mathrm{rep}}_{\fM^{glt}_{g, n, B}}(\fC_B, \cX_{\fM^{glt}_{g, n, B}})$ has quasi-compact and separated diagonal. Indeed, by \cite[\href{https://stacks.math.columbia.edu/tag/0DSM}{Tag 0DSM}]{stacks-project} and \cite[\href{https://stacks.math.columbia.edu/tag/0DSN}{Tag 0DSN}]{stacks-project} there is an open subscheme $S'' \to S$ such that $\bar s \to S$ factors through $S''$ if and only if the automorphism group algebraic space of the corresponding prestable map is locally quasi-finite. Since this group algebraic space is moreover separated and quasi-compact, by \cite[\href{https://stacks.math.columbia.edu/tag/03XX}{Tag 03XX}]{stacks-project} it is a scheme, hence by \cite[\href{https://stacks.math.columbia.edu/tag/02NH}{Tag 02NH}]{stacks-project} it is finite over the base field. Conversely, if $\bar s \to S$ corresponds to a stable map with finite automorphism group scheme, this group is (locally) quasi-finite by \cite[\href{https://stacks.math.columbia.edu/tag/02NU}{Tag 02NU}]{stacks-project}. So $S''$ is the desired set $S_1$.

To show that $S_2$ is open
we can replace $S$ by an open cover and therefore may assume that $S=\Sp (R)$ is affine.  Writing $R = \text{colim} R_i$ as a colimit of algebras of finite type over $\mathbf{Z}$ we then can ``spread out'' our data $(\mathbf{C}, f)$ to an object over one of the $R_i$ thereby further reducing to the case when $S$ is affine of finite type over $\mathbf{Z}$ (and hence Noetherian). Since $S$ is affine the underlying topological space is spectral, hence by  
\cite[\href{https://stacks.math.columbia.edu/tag/0903}{Tag 0903}]{stacks-project} to show that $S_2$ is open it is enough to show that it is constructible and stable under generization.

\subsection{Generization}\label{sec:generization} 
Let $\eta, s \in S$ be such that $\eta$ specializes to $s$, and assume that $(\bC_{\bar s}, f_{\bar s})$ satisfies condition (v) of \ref{def:weighted twisted stable map}. Our goal is to show that $(\bC_{\bar \eta}, f_{\bar \eta})$ also satisfies this condition.
Since $S$ is finite type over $\mathbf{Z}$, by \cite[\href{https://stacks.math.columbia.edu/tag/00PH}{Tag 00PH}]{stacks-project} (applied to the local ring at $s$ in the subscheme given by the closure of $\eta$) we can find a discrete valuation ring $V$ such that $\Sp(V) \to S$ sends the generic point of $V$ to $\eta$ and the closed point to $s$. Hence we can assume $S=\Sp(V)$. We will use $C_\eta$ and $C_s$ (resp. $C_{\bar \eta}, C_{\bar s}$) to denote the base change of $C$ to the respective residue fields (resp. algebraic closures of the residue fields). By \cite[\href{https://stacks.math.columbia.edu/tag/09E8}{Tag 09E8}]{stacks-project}, after replacing $V$ by a finite extension we may assume the irreducible components of $C_\eta$ biject with those of $C_{\bar \eta}$. By \ref{lem:abelian gerbe}, after passing to a further extension we may assume that this bijection identifies components of $C_\eta$ that have an abelian contraction with components of $C_{\bar \eta}$ that have an abelian contraction. 

Let $E_{\bar \eta} \subset C_{\bar \eta}$ be a genus-zero component for which the induced map $\mls E_{\bar \eta} \to \cX$ has an abelian contraction. The component $E_{\bar \eta}$ corresponds to a component $E_\eta$ of $C_\eta$, and the closure of this component is a subscheme $E$ of $C$ flat and proper and with geometrically connected fibers over $S$. In particular the fiber $E_{\bar s}$ is connected and a union of genus-zero irreducible components of $C_{\bar s}$. 
Moreover the map $E \to X$ factors through a section $\sigma: S \to X$: for this, it is enough to show that the scheme theoretic image of $E$ in $X$ maps isomorphically to $S$. But this can be verified fppf locally on $S$ where we can assume $E$ has a section in the smooth locus.

We show that $\mls E_{s} \to \cX$ has an abelian contraction. By \cite[3.2]{tame}, after replacing $V$ by a finite extension, we may write $\mc Y:= \mc X \times_{X, \sigma} S$ as $[U/G]$ where $G$ is a finite group of order invertible in $S$ and $U$ is a finite $S$-scheme. The map $\mls E \to \mc X$ factors through $\mc Y$ and corresponds to a $G$-torsor $P \to \mls E$ with equivariant map $P \to U$. Since  $\mls E_\eta \to \mc X$ has an abelian contraction, the restriction $P_\eta$ of $P$ to the generic fiber is induced from a torsor $R_\eta$ for an abelian subgroup $A \subset G$ (note that the stabilizer group of $\mc Y$ over the generic point of $V$ is not necessarily all of $G$ but a subgroup).   The closure $R$ of $R_\eta$ in $P$ is a reduction of $P$ to $A$: one can check that $R$ is an $A$-torsor \'etale locally on $\mls E$ where $P$ is a trivial torsor, where it is immediate. The fiber $R_s$ of $R$ over $s$ then provides a factorization of $\mls E_s \to \mc X$ through $BA$.

Since $( \mathbf{C}_{\bar s}, f_{\bar s})$ satisfies condition (v) of \ref{def:weighted twisted stable map} and $\mls E_{s} \to \cX$ has an  abelian contraction, every component of $E_{\bar s}$ satisfies \eqref{eq:numerics}. In particular, when we mark the $m$ nodes connecting $E_{\bar s}$ to the rest of $C_{\bar s}$, we see that $E_{\bar s} \to X$ is a weighted stable map in the sense of \ref{S:review} and 
\[
m + \sum_{i, s_i(\bar s) \in E_{\bar s}} a_i > 2,
\]
where $m$ is the number of points in $E_{\bar \eta}$ that are nodes in $C_{\bar \eta}$. Since $m$ is also the number of points connecting $E_{\bar s}$ to the rest of $C_{\bar s}$ and $s_i(\bar s) \in E_{\bar s}$ implies $s_i(\bar \eta) \in E_{\bar \eta}$, we conclude that $E_{\bar s}$ satisfies \eqref{eq:numerics}.

\subsection{Constructibility}

To show that $S_2$ is constructible,
by \cite[\href{https://stacks.math.columbia.edu/tag/053Z}{Tag 053Z}]{stacks-project} it is enough to assume $S$ is integral and show that if $S_2$ is nonempty, it contains a nonempty open subset of $S$. If $S_2$ is nonempty, then since it is stable under generization, it contains the generic point $\eta$ of $S$. Hence it is enough to show that if $S$ is integral and $S_2$ contains the generic point $\eta$, it contains a dense open.
We will repeatedly use that if $S' \to S$ is a dominant morphism of finite type of integral schemes and this result holds after base change to $S'$, then it holds on $S$ by Chevalley's theorem \cite[\href{https://stacks.math.columbia.edu/tag/054K}{Tag 054K}]{stacks-project}.

Making such a base if necessary we may therefore assume that the topological type of $C$ is constant. Indeed after replacing $S$ by a cover we may assume that the irreducible components of the generic fiber are geometrically irreducible, and then by restricting to an open subset of $S$ that the closure $E$ of  every irreducible component $E_\eta $  in the generic fiber  is a proper smooth curve over $S$. In other words, we have reduced to the situation where every irreducible component of $C_\eta$ is geometrically irreducible and the components of every geometric fiber of $C \to S$ biject with the components of the generic fiber.

Let $E_1, \ldots, E_m$ denote the irreducible components of $C \to S$ that have genus zero in some, and hence every, fiber, and such that the morphisms $E_i \to X$ factor through respective sections $\sigma_i: S \to X$ in the generic fiber. The closure of the image of the generic point of $S$ under $\sigma_i$ will map isomorphically to $S$ over a dense open. Hence after replacing $S$ by the intersection of these dense opens we may assume that the factorization through $\sigma_i$ holds on all of $S$.

By \ref{C:openprestable} there is a (possibly empty) open set $S_i \subset S$ where geometric fibers of $\mls E_i \to \mc X$ over $S$ admit an abelian contraction. Replace $S$ with the intersection of the nonempty $S_i$ and replace the list $E_1, \ldots, E_m$ with the list of those components for which $\mls E_i \to \mc X$ now has an abelian contraction in every fiber. (Note that, in every geometric fiber, these $E_i$ are also the only genus-zero components with abelian contractions.) Since condition (v) of \ref{def:weighted twisted stable map} is satisfied in the generic fiber, each $E_i$ satisfies $\# \text{nodes} + \sum a_i > 2$. In fact this inequality holds in every geometric fiber since the topological type is constant. Thus condition (v) of \ref{def:weighted twisted stable map} is satisfied in every geometric fiber.

This completes the proof of \ref{P:6.4b} and therefore also \ref{T:7.4}. \qed

\section{Contracting maps}\label{S:mapcontract}

\begin{pg}\label{S:mapcontract-setup}
Let $\mc X$ be a tame Deligne-Mumford stack over a noetherian base $B$. Let $(q, \phi') :\bC \rightarrow \bD$ be a contraction of generalized log twisted curves over a $B$-scheme $S$ and
assume that this contraction is initial in the sense of \cite[6.14]{OWI}.  Let $Q: \mls C \to \mls D$ be associated morphism of stacks (see \cite[6.7]{OWI}). Let $f:\mls C\rightarrow \mc X$ be a $B$-morphism such that the induced map on coarse moduli spaces $C\rightarrow X$ has a factorization $C \xrightarrow{q} D \to X$.

Let $T \to S$ be a morphism. The \textit{groupoid of factorizations} of $f_T: \mls C_T \to \mc X$ through $Q_T: \mls C_T \to \mls D_T$ is the category whose objects are pairs $(g: \mls D_T \to \cX, \alpha)$ such that the coarse moduli morphism of $g$ is the given map $D \to X$ and $\alpha$ is a 2-isomorphism from $g \circ Q_T$ to $f_T$. A morphism from $(g_1, \alpha_1)$ to $(g_2, \alpha_2)$ is a 2-isomorphism from $g_1$ to $g_2$ that is compatible with $\alpha_1$ and $\alpha_2$. 


If $q: C \to D$ has the property that for every $x \in D$, the fiber over $x$ meets at most one section $s_i$ of $C$, we say that $q$ is \textit{classical}.
Observe that $q$ is classical if and only if every geometric fiber $q_{\bar s}: C_{\bar s} \to D_{\bar s}$ can be factored as a sequence of contractions, each of which contracts a rational bridge with no markings or a rational tail with at most one marking.

The main purpose of this section is to prove the following theorem, which will play a key role in our discussion of stabilization. 
\end{pg}

\begin{thm}\label{T:9.2} 
Let $(q, \phi'): \bC \to \bD$ be an initial contraction over a $B$-scheme $S$ and let $f: \mls C \to \mc X$ be a morphism such that the induced map on coarse moduli spaces factors through $q$.
Assume that for every geometric point $\bar y\rightarrow D$ the induced map on the fiber $\mls C_{\bar y}\rightarrow \mc X$ admits an  abelian contraction.
\begin{enumerate}
\item For any $T \to S$ the groupoid of factorizations of $f_T$ through $Q_T$ is equivalent to a set.
\item
The functor which to any $S$-scheme $T\rightarrow S$ associates this set of factorizations is represented by a closed subscheme of $S$. 
\item If moreover $q$ is classical, or if $\cX = BG$ for some finite \'etale group scheme $G$ over $B$ of order invertible in $B$, then this functor is represented by $S$ itself.
\end{enumerate}
\end{thm}


\begin{rem}
In particular, the theorem asserts that a factorization of $\mls C_T \to \mc X$ through $\mls D_T$ is unique up to unique isomorphism, when it exists, and that a factorization will always exist if $\mc X = BG$. We note that when $q$ is classical the theorem is essentially \cite[9.1.1]{AV}, except that we do not require $\mc X$ to have projective coarse moduli.
\end{rem}

Theorem \ref{T:9.2} can be reformulated as follows. Let $\Hom(\mls C, \mc X)$ be the stack over $S$ that to $T \to S$ associates the groupoid of $T$-morphisms $\mls C_T \to \mc X_T$. By \cite[Theorem C.2]{AOV} the stack $\Hom(\mls C, \mc X)$ is algebraic and locally of finite presentation over $S$. Similarly we have algebraic stacks $\Hom(\mls D, \mc X)$, $\Hom(C, X)$, and $\Hom(D, X)$, each locally of finite presentation over $S$. On the other hand, let $\Hom^{\ab, \, D}(\mls C, \mc X)$ be the subcategory of $\Hom(\mls C, \mc X)$ that to $T \to S$ associates the groupoid of morphisms $\mls C_T \to \mc X_T$ whose coarse moduli morphisms factor as $C_T \xrightarrow{q_T} D_T \to X$ and that have an abelian contraction in every geometric fiber over $D$. We have a commuting diagram
\[
\begin{tikzcd}
\Hom^{\ab,\, D}(\mls C, \mc X) \arrow[r, hookrightarrow, "i"] & \mls H \arrow[r] \arrow[d] & \Hom(\mls C, \mc X) \arrow[d, "\pi"] \\
& \Hom(D, X) \arrow[r, "z\mapsto z\circ q"] & \Hom(C, X)
\end{tikzcd}
\]
where $\mls H$ is defined so the square is cartesian. By \ref{L:2.14b} the morphism $i$ is an open immersion and by \cite[C.2, C.3]{AOV} the morphism $\pi$ is of finite type. It follows that $\Hom^{\ab,\,D}(\mls C, \mc X)$ is an algebraic stack locally of finite presentation over $S$. Composition with $Q$ defines a morphism of algebraic stacks
\begin{equation}\label{eq:contract-iso}
\Hom(\mls D, \mc X) \to \Hom^{\ab,\, D}(\mls C, \mc X).
\end{equation}

\begin{rem}\label{R:reform}
Theorem \ref{T:9.2} is equivalent to the assertion that \eqref{eq:contract-iso} is a closed immersion, and an isomorphism if $\mc X = BG$ or if $q$ is classical. 
\end{rem}

We begin our proof of \ref{T:9.2} by proving a local version of this reformulation \ref{R:reform} in the case when $\mc X = BG$.

\begin{prop}\label{prop:5.2-for-BG}
Assume $S$ is noetherian and strictly henselian local, $q:C \to D$ contracts a single rational component of the closed fiber, and that $G$ is a tame \'etale group scheme over $S$.
 Let $\uTors^{G}_{\mls D/D}$ (resp. $\uTors^{G, \ab }_{\mls C/D}$) be the stack over the \'etale site of $D$ whose objects over $T \to D$ are $G$-torsors on $\mls D_T$ (resp. on $\mls C_T$, such that the induced morphism $\mls C_T \to BG$ has an abelian contraction in every geometric fiber over $T$).  Then the morphism of stacks 
\begin{equation}\label{eq:new-torsor}
\uTors^{G}_{\mls D/D} \to \uTors^{G, \ab }_{\mls C/D}
\end{equation}
defined by pullback along
$Q: \mls C \to \mls D$ is an equivalence. 
\end{prop}
\begin{proof}
For a geometric point $\bar y\rightarrow D$ let $\mls C_{(\bar y)}$ (resp. $\mls D_{(\bar y)}$) denote the base change $\mls C\times _D\Sp (\mls O_{D, \bar y})$ (resp. $\mls D\times _D\Sp (\mls O_{D, \bar y})$).  By a standard limit argument it suffices to show that the pullback functor
$$
\Tors ^G(\mls D_{(\bar y)})\rightarrow \Tors ^{ G, \ab }(\mls C_{(\bar y)})
$$
is an equivalence, where the right side is the category of $G$-torsors on $\mls C_{(\bar y)}$ whose restriction to the fiber of $\mls C_{(\bar y)}$ over the geometric point $\bar y$ has abelian reduction.

    Furthermore, it suffices to consider a geometric point $\bar y\rightarrow D$ such that the fiber $C_{\bar y}$ contains the single contracted rational component $P\subset C_{\bar y}$ (at other points the map $Q$ is an isomorphism).  Let $\mls P$ be the stacky projective line in $\mls C_{\bar y}$ lying over $P$, and let $D(X)$ be the stabilizer group scheme of $\mls D$ over $\bar y$, so we have an inclusion $BD(X)\hookrightarrow \mls D_{(\bar y)}$.

\begin{lem}\label{L:9.22b} The restriction functors 
\begin{equation}\label{E:9.22.1}
\Tors ^{G }(\mls D_{(\bar y)})\rightarrow \Tors ^{G}(BD(X)), \ \ 
\Tors ^{G, \ab }(\mls C_{(\bar y)})\rightarrow \Tors ^{G, \ab }(\mls P)
\end{equation}
are equivalences. 
\end{lem}
\begin{proof}

Let $R$ denote the strictly henselian local ring $\mls O_{D, \bar y}$ and let $\mathfrak{m}_R\subset R$ be the maximal ideal.
For an integer $n\geq 1$ let $\mls D_{(\bar y), n}\subset \mls D_{(\bar y)}$ (resp. $\mls C_{(\bar y), n}\subset \mls C_{(\bar y)}$) denote the closed substack defined by the pullback of the ideal $\mathfrak{m}_R^n\subset R$.  Note that each $\mls D_{(\bar y), n}$ (resp. $\mls C_{(\bar y), n}$) is a nilpotent thickening of $BD(X)$ (resp. $\mls P$).  By the invariance of the \'etale site under nilpotent thickenings it follows that the functors
$$
\varprojlim _n\Tors ^{G }(\mls D_{(\bar y), n})\rightarrow \Tors ^{G}(BD(X)), \ \ \varprojlim _n\Tors ^{G , \ab }(\mls C_{(\bar y), n})\rightarrow \Tors ^{G, \ab }(\mls P)
$$
are equivalences.  By the Grothendieck existence theorem the functors
$$
\Tors ^{G}(\mls D_{(\bar y), \hat R})\rightarrow \varprojlim _n\Tors ^{G}(\mls D_{(\bar y), n}), \ \ \Tors ^{G, \ab }(\mls C_{(\bar y), \hat R})\rightarrow \varprojlim _n\Tors ^{G, \ab }(\mls C_{(\bar y), n})
$$
are equivalences, where $\mls D_{(\bar y), \hat R}$ and $\Tors ^{G^\et , \ab }(\mls C_{(\bar y), \hat R})$ denote the base changes to the completion $\hat R$ of $R$ (note also that $\mls D_{(\bar y)}$ and $\mls C_{(\bar y)}$ are proper over $R$).  By the Artin approximation theorem we then deduce that the functors \eqref{E:9.22.1} are also equivalences (see also \cite[Proof of 3.1]{MR0268188} where a similar argument is made). 
\end{proof}

To complete the proof of \ref{prop:5.2-for-BG} it now suffices to observe that by the construction of the initial contraction in \cite[proof of 6.15]{OWI} the character group $X$ of $D(X)$ is precisely the group $X_m$ associated to $\mls P$ in \ref{S:section8}.  In particular, by \ref{P:3.15} the pullback functor
$$
\Tors ^G(BD(X))\rightarrow \Tors ^{G, \ab }(\mls P)
$$
is an equivalence, which implies the proposition.
\end{proof}

We now turn to the  proof of \ref{T:9.2}, beginning with various reduction steps.  Following a reduction step showing that it suffices to prove the theorem under assumption $\mathbf{P}$, we proceed in the following steps under the further assumption that $\mathbf{P}$ holds in our setup.

\subsection{Reduction to the case when $S$ is the spectrum of the strict henselization of a point on a finite type $\mathbf{Z}$-scheme} 
By descent theory we may replace $S$ by an \'etale cover.  Furthermore, by a limit argument similar to the one  in \cite[2.1]{Olssonproper} (and using \cite[2.7 and 2.16]{tame}) we may assume that $S$ is noetherian, even of finite type over $\mathbf{Z}$.  Applying a similar limit argument we then further reduce to the case of the strict henselization of such a scheme $S$ at a point.


\subsection{Reduction to the case when $q$ contracts a single rational component in the closed fiber}
When $S$ is strictly henselian local, by \cite[4.5]{OWI} we can factor $q$ as a sequence of contractions 
\[
C = C_k \to C_{k-1} \to \ldots \to C_1 \to C_0 = D
\]
with each step being the contraction of a rational bridge with no marked points or a rational tail in the closed fiber. It follows from the construction of initial contractions in \cite[6.13]{OWI} that the composition of initial contractions is initial, so the initial contraction $\bC \to \bD$ factors through the initial contractions $\bC_i \to \bC_{i-1}$. Then \eqref{eq:contract-iso} factors as
\[
\Hom(\mls C_0, \mc X) \to \Hom^{\ab,\, D}(\mls C_1, \mc X) \to \Hom^{\ab,\, D}(\mls C_2, \mc X) \to \ldots \to \Hom^{\ab,\, D}(\mls C_k, \mc X).
\]
and it is enough to show that each morphism in the composition is a closed embedding (and an isomorphism in appropriate situations).

The category $\Hom^{\ab,\,D}(\mls C_{i-1}, \mc X)$ is a full subcategory of $\Hom(\mls C_{i-1}, \mc X)$, and similarly 
\[{\Hom^{\ab,\, D}(\mls C_i, \mc X) \subseteq \Hom^{\ab,\, C_{i-1}}(\mls C_i, \mc X)}\] is a full subcategory. This means that, if \ref{T:9.2} holds for each of the contractions $\mls C_i \to \mls C_{i-1}$, we have a diagram
\begin{equation}\label{eq:reduce-single1}
\begin{tikzcd}
\Hom(\mls C_{i-1}, \mc X) \arrow[r, "\alpha"] & \Hom^{\ab,\, C_{i-1}}(\mls C_i, \mc X)\\
\Hom^{\ab,\, D}(\mls C_{i-1}, \mc  X) \arrow[u] \arrow[r, dashrightarrow] & \Hom^{\ab,\, D}(\mls C_i, \mc X)\arrow[u]
\end{tikzcd}
\end{equation}
where all solid arrows are known to be fully faithful and $\alpha$ is a closed embedding (and an isomorphism if $\mc X = BG$ or $q$ is classical). We will show that  if $g: \mls C_{i-1} \to \mc X$ has the property that its restriction $g|_{\mls C_i}: \mls C_i \to \mc X$ has an abelian contraction in every geometric fiber over $D$, then $g$ already has this property. From this it follows that the square \eqref{eq:reduce-single1} is fibered, and the dashed arrow in \eqref{eq:reduce-single1} is a closed embedding (and an isomorphism if $\mc X = BG$ or if $C_i \to C_{i-1}$ is classical). This will complete the current reductive step.

Assume $g|_{\mls C_i}: \mls C_i \to \mc X$ has an abelian contraction in every geometric fiber over $D$. Let $\bar y_0:\Sp(k) \to D$ be a geometric point mapping to the image of the unique contracted component of $C_i \to C_{i-1}$.  Away from $\bar y_0$, the fibers of $\mls C_i$ and $\mls C_{i-1}$ are equal, so it enough to show that  $(\mls C_{i-1})_{\bar y_0} \to \mc X$ has an abelian contraction. In other words, if $G$ is the automorphism group scheme of $\bar y_0 \to \mc X$, we have an abelian subgroup scheme $A \subseteq G$ and a solid commuting diagram
\begin{equation}\label{eq:reduce-single2}
\begin{tikzcd}
&&BA \arrow[d] \\
(\mls C_i)_{\bar y_0}\arrow[urr] \arrow[r] &(\mls C_{i-1})_{\bar y_0} \arrow[ur, dashrightarrow] \arrow[r, "g"] & BG
\end{tikzcd}
\end{equation}
and we want to show  the existence of a dashed arrow making the diagram commute.

For this there are two cases. 
If $(\mls C_{i-1})_{\bar y_0}$ has positive dimension then its maximal reduced substack is a prestable curve over $\Sp(k). $ Then we can  find a dashed arrow making the top triangle in \eqref{eq:reduce-single2} commute by \ref{prop:5.2-for-BG}. The bottom triangle commutes (via a unique 2-isomorphism) because $\alpha$ is a monomorphism. If $(\mls C_{i-1})_{\bar y_0}$ has dimension zero, we in fact do not use the diagram \eqref{eq:reduce-single2}: in this case the maximal reduced substack of $(\mls C_{i-1})_{\bar y_0}$ is an abelian gerbe, hence the map $(\mls C_{i-1})_{\bar y_0} \to BG$ factors through $BA'$ where $A' \subseteq G$ is the abelian subgroup equal to the image of the stabilizer group of $\mls C_{i-1}$ at $\bar y_0$.


This completes the reduction to the case when $q$ contracts a single rational component in the closed fiber. We split the remainder of the proof of \ref{T:9.2} into two cases: the case when this rational component is a bridge, and the case when it is a tail.

\subsection*{Case 1: When $q$ contracts a rational bridge}
It should be possible to derive the result in this case from the contraction theorem for twisted stable maps in \cite[9.1.1]{AV}. In order to highlight the difference between this case and the tails case in the next section we give an independent proof.



\begin{lem}\label{L:5.4.1}
The adjunction map $\mls O_{\mls D} \to RQ_*\mls O_{\mls C}$ is an isomorphism.
\end{lem}
\begin{proof}
    It suffices to check that the morphism is an isomorphism on \'etale stalks at geometric points of $D$. In fact, since $S$ and hence $D$ are locally Noetherian, the map from a local ring of $D$ to its completion is faithfully flat, and it is enough to show that $\mls O_{\mls D_T} \to Q_{T*}\mls O_{\mls C_T}$ is an isomorphism, when $T \to D$ is the spectrum of the completion of a strictly henselian local ring of $D$.
    Making the base change to the strict henselization of $S$ at the image of the closed point of $T$ in $S$ we may further assume that $S$ is local and the closed point of $T$ maps to the closed point $s$ of $S$.  Let $T_0\subset T$ be the closed subscheme defined by the maximal ideal in $\mls O_{S, s}$ so we have a commutative diagram with cartesian squares
    $$
    \xymatrix{
    \mls C_{T_0}\ar[d]\ar@{^{(}->}[r]& \mls C_T\ar[d]^-{Q_T}\\
    T_0\ar@{^{(}->}[r]\ar[d]& T\ar[d]\\
    s\ar@{^{(}->}[r]& S.}
    $$
    Since $\mls C_T\rightarrow S$ is flat, the top square is tor-independent (defined as in \cite[\href{https://stacks.math.columbia.edu/tag/08IA}{Tag 08IA}]{stacks-project}), and so by \cite[\href{https://stacks.math.columbia.edu/tag/08IR}{Tag 08IR}]{stacks-project}, which generalizes immediately to algebraic stacks, the derived pullback of the cone of the map $\mls O_{\mls D_T}\rightarrow RQ_{T*}\mls O_{\mls C_T}$ is isomorphic to the cone of the map $\mls O_{\mls D_{ T_0}}\rightarrow RQ_{T_0*}\mls O_{\mls C_{T_0}}$.
    By the derived Nakayama lemma (combine \cite[\href{https://stacks.math.columbia.edu/tag/0A05}{Tag 0A05}]{stacks-project} and \cite[\href{https://stacks.math.columbia.edu/tag/0G1U}{Tag 0G1U}]{stacks-project}) it is enough to show that the cone of the map $\mls O_{\mls D_{ T_0}}\rightarrow RQ_{T_0*}\mls O_{\mls C_{T_0}}$ is zero.
    In particular, we can reduce to the case when
    $S = \Sp (k)$ is the spectrum of a separably closed field.

    In this case, let $\widetilde {\mls D}$ (resp. $\widetilde D$) denote the normalization of $\mls D$ (resp. $D$) at the node given by the image of the contracted bridge.  
We then have two points
$$
0, \infty :S\rightarrow \widetilde D
$$
with gerbes $B\mu _{c}$ over them in $\widetilde {\mls D}$, and  a short exact sequence of sheaves on $\mls D$
$$
\xymatrix{
0\ar[r]& \mls O_{\mls D}\ar[r]& \mls O_{\widetilde {\mls D}}\oplus \mls O_{B\mu _{c}}\ar[r]^-{(f, \lambda )\mapsto (f(0)-\lambda , f(\infty )-\lambda )}& \mls O_{B\mu _{c}}\oplus \mls O_{B\mu _{c}}\ar[r]& 0.}
$$

Similarly, the normalization of $\mls C$ is a disjoint union $\mls C'\coprod \mls P$, where $\mls C'\rightarrow \widetilde {\mls D}$ is a rigidification map and $\mls P$ is a stacky $\mathbf{P}^1$, with two stacky points $0:B\mu _a\hookrightarrow \mls P$ and $\infty :B\mu _b\hookrightarrow \mls P$.  We also have a short exact sequence of sheaves on $\mls C$
$$
0\rightarrow \mls O_{\mls C}\rightarrow \mls O_{\mls C'}\oplus \mls O_{\mls P}\rightarrow \mls O_{B\mu _a}\oplus \mls O_{B\mu _b}\rightarrow 0.
$$
Applying $RQ_*$ to the second sequence and combining it with the first we get a morphism of distinguished triangles
$$
\xymatrix{
\mls O_{\mls D}\ar[r]\ar[d]& \mls O_{\widetilde {\mls D}}\oplus \mls O_{B\mu _{c}}\ar[r]^-{}\ar[d]& \mls O_{B\mu _{c}}\oplus \mls O_{B\mu _{c}}\ar[r]\ar[d]&\mls O_{\mls D}[1]\ar[d]\\
RQ_*\mls O_{\mls C}\ar[r]& RQ_*\mls O_{\mls C'}\oplus RQ_*\mls O_{\mls P}\ar[r]& RQ_*\mls O_{B\mu _a}\oplus RQ_*\mls O_{B\mu _b}\ar[r]& RQ_*\mls O_{\mls C}[1].}
$$
Since $\mls C' \to \widetilde{\mls D}$, $B\mu_a \to B\mu_c$, and $B\mu_b \to B\mu_c$ are rigidification maps, the corresponding adjunction maps are isomorphisms. It therefore suffices to show that the adjunction map
\[
\mls O_{B\mu_c} \to RQ_*\mls O_{\mls P}
\]
is an isomorphism. For this note that $\mls P \to B\mu_c$ factors as $\mls P \to \mls P_{c, c} \to B\mu_c$ where $\mls P_{c, c}$ is the relative coarse space and equal to $[\mathbf{P}^1/\mu_c]$. The morphism $\mls P \to \mls P_{c, c}$ is a rigidification, so the corresponding adjunction map is an isomorphism, and the morphism $\mls P_{c, c} \to B\mu_c$ is the quotient of $\mathbf{P}^1 \to S$ so the adjunction map is an isomorphism in this case.
  
\end{proof}

The proof of \ref{T:9.2} in the present case follows from \ref{prop:5.2-for-BG}, \ref{L:5.4.1}, and the following lemma. We will use this lemma again in the next section, and so state it here in more generality than currently needed (for the current application, we can take $\bC^\dagger = \bD$).

\begin{lem}\label{L:5.7}
Let $S$ be the spectrum of a strictly henselian local ring, let $C \to C^\dagger \to D$ be contractions of curves over $S$, and let $\bC \to \bC^\dagger$ be a contraction of generalized log twisted curves over $S$ (not necessarily initial). Let $Q: \mls C \to \mls C^\dagger$ denote the associated morphism of stacks. Assume that for every tame \'etale group scheme $G$ over $S$, the restriction morphism $\uTors^{G,\ab}_{\mls C^\dagger/D} \to \uTors^{G, \ab}_{\mls C/ D}$ is an equivalence, and that $\mls O_{\mls C^\dagger} \to Q_*\mls O_{\mls C}$ is an isomorphism. 
Then the morphism $\Hom^{\ab, D}(\mls C^\dagger, \mc X) \to \Hom^{\ab,\,D}(\mls C, \mc X)$ is an equivalence. 
\end{lem}
\begin{proof}
That $\Hom^{\ab,\,D}(\mls C^\dagger, \mc X) \to \Hom^{\ab,\,D}(\mls C, \mc X)$ is an equivalence may be verified \'etale locally on $D$. That is, for an arbitrary geometric point $\bar x \to D$, it is enough to verify the result after replacing $D$ by the spectrum of the strictly henselian local ring at $\bar x$, which we denote $D_{(\bar x)}$. We replace $\mls C^\dagger$, $\mls C$, and $\mc X$ by $\mls C^{\dagger}_{(\bar x)}:= \mls C^\dagger \times_D D_{(\bar x)}$, $\mls C_{(\bar x)}:=\mls C \times_D D_{(\bar x)}$ and $\mc X_{(\bar x)}:= \mc X \times_X D_{(\bar x)}$, respectively. We note that since the coarse space commutes with arbitrary base change, $D_{(\bar x)}$ is the coarse space of $\mc X_{(\bar x)}.$ 

In this case, $\mc X$ is equal to $[W/G]$ for a finite $D_{(\bar x)}$-scheme $W$ and a tame \'etale group scheme $G$ over $D_{(\bar x)}$. In particular $W$ is affine. A morphism $f: \mls C_{(\bar x)} \to \mc X_{(\bar x)}$ is the data of a $G$-torsor $p: P \to \mls C_{(\bar x)}$ and a $G$-equivariant morphism $P \to W$. If $f$ has an abelian contraction in every geometric fiber over $D$, then by assumption, $P$ descends to a unique $G$-torsor $p_0: P_0 \to \mls C^\dagger_{(\bar x)}.$ Writing $W = \Spec(R)$ the morphism $P \to W$ is given by a $G$-equivariant ring homomorphism $R \to H^0(P, \mls O_P)$. The sheaf $p_{0, *} \mls O_{P_0}$ is locally free since $P_0 \to \mls C^\dagger_{(\bar x)}$ is finite flat, so by the projection formula and our assumption that $\mls O_{\mls C^\dagger} \to Q_*\mls O_{\mls C}$ is an isomorphism, the canonical map
\[
p_{0, *} \mls O_{P_0} \to Q_* p_*\mls O_P
\]
is an isomorphism.
It follows that $H^0(P, \mls O_P) = H^0(P_0, \mls O_{P_0})$, so the map $P \to W$ descends to a unique $G$-equivariant map $P_0 \to W$. 
\end{proof}



\qed


\subsection*{Case 2: When $q$ contracts a rational tail}\label{sec:case2}

This case is substantially harder because the analog of \ref{L:5.4.1} is false (this is also what gives rise to the closed subscheme in \ref{T:9.2} (2), as we will see).

\subsection{Reduction to the case when $\mls C \to \mls D$ is representable}
We first reduce to the case when the morphism $\mls C \to \mls D$ is representable. This morphism arises from an initial contraction $\bC \to \bD$. By \cite[6.15]{OWI}, this contraction factors as contractions $\bC \to \bC^\dagger \to \bD$, where $\bC \to \bC^\dagger$ induces an isomorphism of coarse spaces and $\bC^\dagger \to \bD$ induces a representable morphism $\mls C^\dagger \to \mls D$ of associated stacks. We note that since $\bC \to \bD$ was initial, the contraction $\bC^\dagger \to \bD$ is also initial: this follows from the fact that in the category of generalized log twisted curves with contraction morphisms, all morphisms are epimorphisms, and this in turn boils down to the obeservation that the morphism $\phi'$ in a contraction $(q, \phi')$ is an injective sheaf homomorphism. 
Then the following lemma implies that we may replace $\bC \to \bD$ with $\bC^\dagger \to \bD$.

\begin{lem}
The functor $\Hom^{\ab,\,D}(\mls C^\dagger, \mc X) \to \Hom^{\ab,\,D}(\mls C, \mc X)$ is an equivalence.
\end{lem}
\begin{proof}
Let $Q^\dagger$ denote the induced morphism $\mls C \to \mls C^\dagger$. By \cite[6.11]{OWI}, the factorization $\mls C \to \mls C^\dagger \to \mls D$ realizes $Q^\dagger$ as the relative coarse moduli space of $\mls C \to \mls D$.

We wish to apply \ref{L:5.7} to $Q^\dagger$. To do so, we must verify that 
\begin{equation}\label{eq:repable1}
\uTors^{G,\ab}_{\mls C^\dagger/D} \to \uTors^{G,\ab}_{\mls C/D}
\end{equation}
is an equivalence and that $\mls O_{\mls C^\dagger} \to Q^\dagger_*\mls O_{\mls C}$ is an isomorphism. The latter holds since $\mls C \to \mls C^\dagger$ is a relative coarse moduli space. To see that \eqref{eq:repable1} holds note the existence of a commuting triangle
\[
\begin{tikzcd}
&\uTors^G_{\mls D/D} \arrow[dl, "r_1"'] \arrow[dr, "r_2"] \\
\uTors^{G,\ab}_{\mls C^\dagger/D} \arrow[rr, "{\text{\ref{eq:repable1}}}"] &&\uTors^{G,\ab}_{\mls C/D}
\end{tikzcd}
\]
where $r_1$ and $r_2$ are equivalences by \ref{prop:5.2-for-BG}, since $\bC \to \bD$ and $\bC^\dagger \to \bD$ are both initial.

\end{proof}


\subsection{Explicit description of $\mls C \to \mls D$} 
We can now assume $\mls C \to \mls D$ is representable and contracts a single rational tail in the closed fiber. Our proof of \ref{T:9.2} in this case will follow our proof in the bridges case as closely as possible: in particular, the main task is to describe properties of the morphism $\mls O_{\mls D} \to Q_*\mls O_{\mls C}$, which we do in \ref{S:tail-functions}. To obtain those results we describe the contraction $\mls C \to \mls D$ very explicitly, at least near the image of the contracted tail. This comprises paragraphs \ref{P:5.10}-\ref{P:5.14}. 

\begin{pg}\label{P:5.10}
Write $S = \Sp (\A)$ with $\A$ strictly henselian local with closed point $s$ and residue field $k$. We are assuming that $C\rightarrow D$ contracts a rational tail $P\subset C_s$. For notational simplicity we will assume all the sections $s_i$ meet $P$.
Let $\bar y\rightarrow D$ be a geometric point in the closed fiber over the image point of $P$.  Let $R$ denote the strictly henselian local ring $\mls O_{D, \bar y}$ and (as in the proof of \ref{L:5.7}) define
$$
C_{(\bar y)}:= C\times _D\Sp (R), \ \ \mls C_{(\bar y)}:= \mls C\times _D\Sp (R),  \ \ \mls D_{(\bar y)}:= \mls D\times _{D}\Sp (R).
$$
For each of the marked points meeting $P$ we have a monoid $N_j$  as in \cite[2.17 (iv)]{OWI}.  Let $N$ denote $\oplus _jN_j$, and let $X$ denote the quotient  $N^\gp /\mathbf{Z}^n$.  
The isotropy group at the unique node of $P \subset C_s$ is $\mu_m$ for some integer $m$. Let $N_m \subset N$ denote the submonoid of elements whose coordinate sum is in $\frac{1}{m}\mathbf{Z}$ and let $X_m = N^{gp}_m/\mathbf{Z}^n$ (note that this definition of $X_m$ agrees with the one in \ref{S:section8}). 

It follows from the definition of an initial contraction, specifically from \cite[(6.19.1)]{OWI}, that 
\[
\mls D_{(\bar y)} \simeq [\Sp (R\otimes _{\mathbf{Z}[\mathbf{N}^n]}\mathbf{Z}[N_m])/D(X_m)].
\]
Let $\mls P\subset \mls C_{(\bar y)}$ be the stacky $\mathbf{P}^1$ lying over $P$. Since we have reduced to the case that $\mls C \to \mls D$ is representable, 
the  morphism $\rho: \mls P \to BD(X_m)$ of \ref{S:generalcase} is representable. Therefore the fiber product $\mc Q$ in the Cartesian diagram
\[
\begin{tikzcd}
\mc Q \arrow[r] \arrow[d] & \Sp (R\otimes _{\mathbf{Z}[\mathbf{N}^n]}\mathbf{Z}[N_m]) \arrow[d] \\
\mls C_{(\bar y)} \arrow[r] & \mls D_{(\bar y)}
\end{tikzcd}
\]
is a scheme finite over $C_{(\bar y)}$. Note that $\mc Q \to \mls C_{(\bar y)}$ is a $D(X_m)$-torsor.
\end{pg}

\begin{pg}\label{P:5.12b}
To describe $\mc Q$ it is convenient to also consider the log structure $M_{\mc Q}$ obtained by pullback from the tautological log structure $M_{\mls C_{(\bar y)}}$ on $\mls C_{(\bar y)}$.  Also let $M_{N_m}$ be the log structure on $\Sp (R\otimes _{\mathbf{Z}[\mathbf{N}^n]}\mathbf{Z}[N_m])$ associated to the natural map $N_m\rightarrow \mathbf{Z}[N_m]$, so $M_{N_m}$ is the pullback of the tautological log structure $M_{\mls D_{(\bar y)}}$ on $\mls D_{(\bar y)}$.  By \cite[6.8]{OWI} we then get a $D(X_m)$-equivariant morphism of log schemes
$$
\rho :(\mc Q, M_{\mc Q})\rightarrow (\Sp (R\otimes _{\mathbf{Z}[\mathbf{N}^n]}\mathbf{Z}[N_m]), M_{N_m})
$$
over the mophism of log schemes $(C_{(\bar y)}, M_{C_{(\bar y)}})\rightarrow (D_{(\bar y)}, M_{D_{(\bar y)}})$.

The advantage of incorporating the log structure here is that by \cite[6.8]{OWI} if $\mc Q'\rightarrow \mls C_{(\bar y)}$ is any $D(X_m)$-torsor equipped with a $D(X_m)$-equivariant log map $\rho ': (\mc Q', M_{\mc Q'})\rightarrow (\Sp (R\otimes _{\mathbf{Z}[\mathbf{N}^n]}\mathbf{Z}[N_m]), M_{N_m})$ over $(C_{(\bar y)}, M_{C_{(\bar y)}})\rightarrow (D_{(\bar y)}, M_{D_{(\bar y)}})$, where $M_{\mc Q'}$ is the pullback of $M_{\mls C_{(\bar y)}}$,  then $(\mc Q, \rho ) \simeq (\mc Q', \rho ')$.  To describe $\mc Q$ we therefore simply have to write down a torsor with such a log map, which we now do directly. 

\end{pg}

\begin{pg} \label{P:5.15} In this paragraph we construct an open cover of $C_{(\bar y)}$. Let $x\in R$ denote a local parameter so that $R$ is isomorphic to the strict henselization of $\A[x]$ at the maximal ideal $(\mathfrak{m}_{\A}, x)$.  Then
$$
C_{(\bar y)}\simeq \text{Proj}(R[U, V]/(xU-tV))
$$
for some element $t\in \mathfrak{m}_\A$.  
 Note that the simple extension 
 $\MS{S}{C} \hookrightarrow \MS{S}{C}'$ 
 corresponds to an $m$-th root $t^{1/m}\in \A$ of $t$. We now construct an open cover of $C_{(\bar y)}.$

 Setting $u=U/V$ and $v=V/U$ the scheme $C_{(\bar y)}$ is covered by two open sets
$$
\Sp (R[u]/(xu-t)), \ \ \Sp (R[v]/(x-vt)).
$$

Since the sections $s_i$ have image in the closed fiber contained in the smooth locus of $P$, they are contained in the open subset 
$$
\Sp (R[v]/(x-tv)).
$$
The section $s_i$ is therefore induced by a ring map $R[v]/(x-tv) \to \A$, and this is determined by an element $a_i \in \A$ (the image of $v$). It follows that the scheme-theoretic image of $s_i$ in $\Sp (R)$ is given by the ideal equal to the kernel of the map
$$
\xymatrix{
R\ar[r]& R[v]/(x-tv)\ar[r]^-{v\mapsto a_i}&\A.}
$$
Hence the image of $s_i$ in $\Sp (R)$ is given by the ideal $(x-ta_i)$.

Observe also that the ideal generated by $x-ta_i$ in $R[u]/(xu-t)$ is equal to $x(1-ua_i)$. 
Rather than work with the standard open cover of $C_{(\bar y)}$ it will be convenient to remove the marked points from the first open set.  Let $F_u$ (resp. $F_v$) denote the polynomial $\prod _i(1-ua_i)$ (resp. $\prod _i(v-a_i)$).  We then cover $C_{(\bar y)}$ by the two open sets
\begin{equation}\label{E:10.24.1}
\Omega^{\nd}:=\Sp (R[u]_{F_u}/(xu-t)), \ \ \Omega^{\mk}:=\Sp (R[v]/(x-vt)),
\end{equation}
noting that $\Omega^{\nd}$ contains the node and $\Omega^{\mk}$ contains the marked points of the closed fiber of $C_{(\bar y)}$.
\end{pg}

\begin{pg}
The scheme $\mc Q$ is finite over $C_{(\bar y)}$, so 
the fiber products
\[
\mc Q|_{\Omega^{\nd}} := \mc Q \times_{C_{(\bar y)}} \Omega^{\nd} \quad \quad \quad  \ \ \mc Q|_{\Omega^{\mk}} := \mc Q \times_{C_{(\bar y)}} \Omega^{\mk}
\]
are affine schemes over the open sets $\Omega^{\nd}$ and $\Omega^{\mk}$. In this paragraph we compute the coordinate ring of $\mc Q|_{\Omega^{\nd}}$ and a chart for its log structure. To begin, we have
$$
\mls C_{(\bar y)}\times _{C_{(\bar y)}}\Omega^{\nd}\simeq [\Sp (R[x^{1/m}, u^{1/m}]_{F_u}/(x^{1/m}u^{1/m}-t^{1/m}))/\mu _m].
$$
Here we abuse notation slightly, writing $R[x^{1/m}, u^{1/m}]_{F_u}$ for the ring obtained by adjoining $m$-th roots of $x$ and $u$ to $R[u]_{F_u}$.   The group $\mu _m$ acts on this localized ring by $\zeta * x^{1/m} = \zeta x^{1/m}$, $\zeta *u^{1/m} = \zeta ^{-1}u^{1/m}$ ($\zeta \in \mu _m$) since $F_u$ is invariant under the action.  
The $D(X_m)$-torsor $\mc Q|_{\Omega^{\nd}} \to \mls C_{(\bar y)} \times_{C_{(\bar y)}} \Omega^{\nd}$ is the pushout of the $\mu_m$-torsor
$$
\Sp (R[x^{1/m}, u^{1/m}]_{F_u}/(x^{1/m}u^{1/m}-t^{1/m}))\rightarrow  [\Sp (R[x^{1/m}, u^{1/m}]_{F_u}/(x^{1/m}u^{1/m}-t^{1/m}))/\mu _m]
$$
along the map $\mu _m\rightarrow D(X_m)$ 
induced by the map $\chi :X_m\rightarrow \frac{1}{m}\mathbf{Z}/\mathbf{Z}$ 
in \eqref{E:chidef}.  We find that $\mc Q|_{\Omega^{\nd}}$ has 
coordinate ring
\begin{equation}\label{E:5.14.1}
((R[x^{1/m}, u^{1/m}]_{F_u}/(x^{1/m}u^{1/m}-t^{1/m}))[X_m])^{\mu _m}.
\end{equation}
Here $\mu _m$ acts through the above action on $(R[x^{1/m}, u^{1/m}]_{F_u}/(x^{1/m}u^{1/m}-t^{1/m}))$ and on basis elements coming from $X_m$ through the homomorphism $\mu _m\to D(X_m)$ induced by $-\chi $.  

To understand the log structure on $\mc Q|_{\Omega^{\nd}}$ observe that the pushout construction of $\mc Q|_{\Omega^{\nd}}$ is captured by the fiber diagram (with the coordinate ring of $\mc Q|_{\Omega^{\nd}}$ appearing in the top right corner)
$$
\xymatrix{
 \Sp (R[x^{1/m}, u^{\pm1/m}]_{F_u}/(x^{1/m}u^{1/m}-t^{1/m})[X_m])\ar[r]^-{\mu _m}\ar[d]_-{D(X_m)}&\Sp (((R[x^{1/m}, u^{\pm1/m}]_{F_u}/(x^{1/m}u^{1/m}-t^{1/m}))[X_m])^{\mu _m})\ar[d]_-{D(X_m)} \\
\Sp (R[x^{1/m}, u^{\pm1/m}]_{F_u}/(x^{1/m}u^{1/m}-t^{1/m}))\ar[r]_-{\mu _m}&\mls C_{(\bar y)},
}
$$
where each arrow is a torsor with the indicated group. The log structure on $R[x^{1/m}, u^{\pm1/m}]_{F_u}/(x^{1/m}u^{1/m}-t^{1/m})$ is described using the standard chart: It is the one given by $\mathbf{N}^2$ with generators mapping to $x^{1/m}$ and $u^{1/m}$.  Note that the associated log structure (but not the chart) is naturally $\mu _m$-invariant.  Let $M_1$ be this log structure on $R[x^{1/m}, u^{\pm1/m}]_{F_u}/(x^{1/m}u^{1/m}-t^{1/m})$.  Let $M_2$ be the pullback of this log structure to $R[x^{1/m}, u^{\pm1/m}]_{F_u}/(x^{1/m}u^{1/m}-t^{1/m})[X_m]$.  The log structure $M_2$ is then equivariant with respect to the $\mu _m\times D(X_m)$-action and therefore descends to a $D(X_m)$-equivariant log structure $M_3$ on $((R[x^{1/m}, u^{\pm1/m}]_{F_u}/(x^{1/m}u^{1/m}-t^{1/m}))[X_m])^{\mu _m}$, which is the log structure we are trying to describe.

Adding some units to a chart doesn't change the associated log structure so a chart for $M_2$ is given by
$$
X_m\oplus \mathbf{N}^2\rightarrow R[x^{1/m}, u^{\pm1/m}]_{F_u}/(x^{1/m}u^{1/m}-t^{1/m})[X_m], \ \ (\bar n, a, b)\mapsto x^{a/m}u^{b/m}e^{\bar n}.
$$
Choose a map $\tilde \chi :\mathbf{N}^2\rightarrow X_m$ lifting the map $\mathbf{N}^2\rightarrow \mathbf{Z}/(m)$ sending $(a, b)$ to $a-b$ (since $\mathbf{N}^2$ is a free monoid we just need to lift the generators) and let $\beta :\mathbf{N}^2\rightarrow X_m\oplus \mathbf{N}^2$ be the graph of $-\tilde \chi $.  Then the induced map
$$
\mathbf{N}^2\rightarrow X_m\oplus \mathbf{N}^2\rightarrow R[x^{1/m}, u^{\pm1/m}]_{F_u}/(x^{1/m}u^{1/m}-t^{1/m})[X_m]
$$
is another chart for $M_2$ with image in the $\mu _m$-invariants, and therefore defines a chart for $M_3$.  Note also that we can describe the $D(X_m)$-action on $M_3$ in terms of this chart.

\begin{conclusion}\label{S:5.14}
    The restriction of the log structure $M_{\mc Q}$ to $\mc Q|_{\Omega^{\nd}}$
    is given by the chart
    \begin{equation}\label{eq:5.14c}
    \mathbf{N}^2\rightarrow ((R[x^{1/m}, u^{\pm1/m}]_{F_u}/(x^{1/m}u^{1/m}-t^{1/m}))[X_m])^{\mu _m}, \ \ (a, b)\mapsto e^{-\tilde \chi (a, b)}u^{a/m}x^{b/m}.
    \end{equation}
    This log structure has the following properties.
    \begin{itemize}
\item[(i)] It is pulled back from $\mls C$. Indeed this follows from the fact that it has a natural $D(X_m)$-linearization.
\item[(ii)] Its restriction to the open set where $u^{1/m}$ is a unit is given by
\[
\mathbf{N} \to ((R[x^{1/m}, u^{\pm1/m}]_{F_u}/(x^{1/m}u^{1/m}-t^{1/m}))[X_m])^{\mu _m}, \ \ 1\mapsto t^{1/m}.
\]
This follows from the fact that on this open set the image of the first generator of \eqref{eq:5.14c} restricts to a unit, and the image of the second restricts to a unit multiple of $t^{1/m}$.  
\end{itemize}
\end{conclusion}

\end{pg}

\begin{pg} In this paragraph we compute the coordinate ring of $\mc Q|_{\Omega^{\mk}}$  and a chart for its log structure, as well as a gluing isomorphism for these affine log schemes.
Fortunately the restriction of $\mc Q$ to $\Omega^{\mk}$ is easier to describe: Its coordinate ring is
\begin{equation}\label{E:5.14.2}
R[v]/(x-tv)\otimes _{\mathbf{Z}[\mathbf{N}^n]}\mathbf{Z}[N_m],
\end{equation}
where the map $\mathbf{N}^n\rightarrow R[v]/(x-tv)$ sends the $i$-th generator $e_i$ to $v-a_i$. The log structure on this piece has the chart 
$$
\mathbf{N}\oplus N_m\rightarrow R[v]/(x-tv)\otimes _{\mathbf{Z}[\mathbf{N}^r]}\mathbf{Z}[N_m]
$$
where $(1,0)$ maps to $t^{1/m}$ and $n\in N_m$ maps to $1\otimes e^n$. 

The gluing isomorphism of the two $D(X_m)$-torsors $\mc Q|_{\Omega^{\nd}} \to \mls C_{(\bar y)} \times_{C_{(\bar y)}} \Omega^{\nd}$ and $\mc Q|_{\Omega^{\mk}} \to \mls C_{(\bar y)} \times_{C_{(\bar y)}} \Omega^{\mk}$ is given by the ring isomorphism 
\begin{align}\notag
R[v^{\pm }]_{F_v}/(x-tv)\otimes _{\mathbf{Z}[\mathbf{N}^r]}\mathbf{Z}[N_m]& \simeq ((R[x^{1/m}, u^{\pm 1/m}]_{F_u}/(x^{1/m}u^{1/m}-t^{1/m}))[X_m])^{\mu _m}, \\
1\otimes e^n& \mapsto u^{-\chi (n)}e^{\bar n},\label{E:5.14.3}
\end{align}
 over the coarse space isomorphism given by $v\mapsto u^{-1}$.  Note that the log structures on both open sets are isomorphic to the one given by $t^{1/m}$ so this extends to an isomorphism of log schemes.  Hence we have a well-defined $D(X_m)$-torsor $\mc Q \to \mls C_{(\bar y)}$ with a log structure $M_{\mc Q}$.

\end{pg}

\begin{pg}\label{P:5.14}
Finally, in this paragraph we compute the $D(X_m)$-equivariant log map 
\begin{equation}\label{eq:5.14}(\mc Q, M_{\mc Q}) \rightarrow (\Sp (R\otimes _{\mathbf{Z}[\mathbf{N}^n]}\mathbf{Z}[N_m]), M_{N_m}).\end{equation}
This morphism is defined on our two open sets $\mc Q|_{\Omega^{\nd}}$ and $\mc Q|_{\Omega^{\mk}}$ as follows. On $\mc Q|_{\Omega^{\nd}} = \Sp(((R[x^{1/m}, u^{1/m}]_{F_u}/(x^{1/m}u^{1/m}-t^{1/m}))[X_m])^{\mu _m})$, the map is induced by 
the map on rings sending $1\otimes e^n$ to $x^{\chi (n)}e^{\bar n}$ and the map on monoids
$$
N_m\rightarrow (((R[x^{1/m}, u^{\pm1/m}]_{F_u}/(x^{1/m}u^{1/m}-t^{1/m}))[X_m])^{\mu _m})^*\oplus \mathbf{N}^2,
$$
$$
n\mapsto (e^{\tilde \chi (0, m\chi (n))+\bar n}, 0, m\chi (n)).
$$
(Note that adding units to the chart for the log structure on $\mc Q|_{\Omega^{\nd}}$ does not change the associated log structre.) On $\mc Q|_{\Omega^{\mk}} = \Sp(R[v]/(x-tv)\otimes _{\mathbf{Z}[\mathbf{N}^n]}\mathbf{Z}[N_m])$, the map is induced by the map on rings sending $1 \otimes e^n$ to $t^{\chi(n)} \otimes e^n$ and the map on monoids 
\[
N_m \to \mathbf{N} \oplus N_m, \quad \quad \quad \quad n \mapsto (m\chi(n), n).
\]
One checks that these formulas define morphisms of log schemes
\[
(\mc Q
_{\Omega^{\square}}, M_{\mc Q|_{\Omega^{\square}}}) \rightarrow (\Sp (R\otimes _{\mathbf{Z}[\mathbf{N}^n]}\mathbf{Z}[N_m]), M_{N_m})
\]
where $\square$ equals $\mk$ or $\nd$. To check that these glue to define a morphism as in \eqref{eq:5.14}, note that the scheme maps are compatible with the gluing isomorphism \eqref{E:5.14.3} and therefore defines the desired $D(X_m)$-equivariant morphism of schemes.
To verify that the morphisms of log structures are compatible it is enough to observe that the log structure on the target pulls back to the  log structure $t^{1/m}$ on the intersection of $\mc Q|_{\Omega^{\mk}} \cap \mc Q|_{\Omega^{\nd}}$, and for this it suffices to verify that the elements of $N_m$ map to units. But this is immediate.

    Lastly, we should verify that our morphism \eqref{eq:5.14} is a morphism of log schemes over $(C_{(\bar y)}, M_{C_{(\bar y)}}) \to (D_{(\bar y)}, M_{D_{(\bar y)}})$, as required in \ref{P:5.12b}. This can be verified over the two open sets $\mc Q|_{\Omega^{\nd}}$ and $\mc Q|_{\Omega^{\mk}}$. For $\mc Q|_{\Omega^{\nd}}$ this follows from the construction of the log structure: see \ref{S:5.14}(i).
    For $\mc Q|_{\Omega^{\mk}}$ we can verify it directly by checking that the diagram of monoid morphisms
    \[
    \begin{tikzcd}[column sep = 70pt]
    \mathbf{N} \oplus N_m & N_m \arrow[l, "{(m\chi(n), n)}\mapsfrom n"']\\
    \mathbf{N} \oplus \mathbf{N}^n \arrow[u, "{(\cdot m, i)}"] & \mathbf{N}^n \arrow[u, "i"] \arrow[l, "{(\Sigma \ba, \ba)}\mapsfrom \ba"]
    \end{tikzcd}
    \]
    commutes, where $i: \mathbf{N}^n \to N_m$ is the inclusion and $\Sigma \ba $ is the coordinate sum of $\ba \in \mathbf{N}^n$.

    This concludes our description of the data in \ref{P:5.12b}. 
\end{pg}

\subsection{Properties of $\mls O_{\mls D} \to Q_*\mls O_{\mls C}$}\label{S:tail-functions}

\Rachel{On p. 30 of the arxiv version of our first paper, in the commuting diagram of log structures, should $\mathbf{N}^2$ in the bottom left corner be $\mathbf{N}$?}\Martin{Yes, good catch!}

In this section, we continue with the notation in \ref{P:5.10} and \ref{P:5.15}. We use our description of $\mls C_{(\bar y)} \to \mls D_{(\bar y)}$ to derive the following result (compare to Lemma \ref{L:5.4.1}, which holds when $q$ contracts a rational bridge).

\begin{lem}\label{L:tails}
The morphism $\mls O_{\mls D_{(\bar y)}} \to Q_*\mls O_{\mls C_{(\bar y)}}$ is injective, and an isomorphism if $C_{\bar y}$ contains at most one marked point. The cokernel is flat if $S = \Sp(\beta)$ is integral.
\end{lem}

We note that when $C_{\bar y}$ contains no marked points, the image of $\bar y \to D$ is not marked, and hence $\mls D_{(\bar y)} = D_{\bar y}$. Then since $\mls C \to \mls D$ is representable we have that $\mls C_{(\bar y)} = C_{(\bar y)}$, and that $\mls O_{D_{(\bar y)}} \to q_* \mls O_{C_{(\bar y)}}$ is an isomorphism follows from the fact that $q$ is a contraction. Hence from now on we assume that $C_{(\bar y)}$ has at least one marked point, or in other words $n \geq 1$. In this case Lemma \ref{L:tails} is proved as \ref{L:injective}, \ref{C:classical}, and \ref{C:5.13c} below. We can prove injectivity immediately.

\begin{lem}\label{L:injective}
The morphism $\mls O_{\mls D_{(\bar y)}} \to Q_*\mls O_{\mls C_{(\bar y)}}$ is injective.
\end{lem}
\begin{proof}
We can check this on the \'etale cover $\Sp(R \otimes_{\mathbf{Z}[\mathbf{N}^n]} \mathbf{Z}[N_m]) \to \mls D_{(\bar y)}$: it is enough to show that $R \otimes_{\mathbf{Z}[\mathbf{N}^n]} \mathbf{Z}[N_m] \to H^0(\mc Q, \mls O_\mc Q)$ is injective.

For this, observe that on the open subscheme of $D_{(\bar y)}$ where $x \neq 0$, the map $\mls C_{(\bar y)} \to \mls D_{(\bar y)}$ is an isomorphism. Let $\mc Q_x$ be the restriction of $\mc Q$ to this open set. So in the square of restriction maps
\[
\begin{tikzcd}
H^0(\mc Q_x, \mls O_{\mc Q_x}) & \arrow[l] R[1/x] \otimes_{\mathbf{Z}[\mathbf{N}^n]} \mathbf{Z}[N_m] \\
H^0(\mc Q, \mls O_\mc Q) \arrow[u] & \arrow[l] \arrow[u] R \otimes_{\mathbf{Z}[\mathbf{N}^n]} \mathbf{Z}[N_m] 
\end{tikzcd}
\]
the top arrow is an isomorphism. Since the right vertical map is injective, the desired morphism (bottom arrow) is also injective.
\end{proof}

To prove the remainder of \ref{L:tails}, we need to calculate $H^0(\mc Q, \mls O_{\mc Q})$. For $\lambda \in X_m$ let $H^0(\mc Q, \mls O_{\mc Q)})_\lambda$ denote the submodule of $H^0(\mc Q, \mls O_{\mc Q})$ on which $D(X_m)$ acts through the character $\lambda$. Let $n_\lambda \in N_m$ be the minimal lift of $\lambda$, and decompose $\chi(n_\lambda) \in (1/m)\mathbf{N}$ into its integral and fractional parts, writing
\[
\chi(n_\lambda) = s_\lambda + {w_\lambda}/{m} \quad \quad \quad 0 \leq w_\lambda < m, \quad s_\lambda \in \mathbf{N}.
\]
Recall the open cover $\mc Q|_{\Omega^{\mk}} \cup \mc Q|_{\Omega^{\nd}}$ of $\mc Q$.

\begin{lem}\label{L:9.14} 
The $R$-module $H^0(\mc Q, \mls O_\mc Q)$ is finitely generated. The restriction map
\begin{equation}\label{eq:r}
H^0(\mc Q, \mls O_\mc Q) \to H^0(\mc Q|_{\Omega^{\nd}}, \mls O_{\mc Q|_{\Omega^{\nd}}})
\end{equation}
is injective with image the $\beta$-submodule
\begin{equation}\label{E:0.9.1}
H = \oplus_{\lambda \in X_m} H_\lambda, \quad \quad \quad H_\lambda:= (\oplus _{j=0}^{s_\lambda-1}M_j\cdot x^{j+(w_\lambda/m)}e^{\lambda})\oplus Rx^{s_\lambda + (w_\lambda/m)}e^{\lambda},
\end{equation}
where $M_j:= \text{\rm Ker}(\cdot t^{j+(w_n/m)}:\A\rightarrow \A)$.

\end{lem}


\begin{proof}
The $R$-module $H^0(\mc Q, \mls O_Q)$ is finitely generated because $\mc Q$ is proper over $R$.
Injectivity of \eqref{eq:r} follows from the observation that the restriction map
\begin{equation}\label{eq:r2}
H^0(\mc Q|_{\Omega^{\mk}}, \mls O_{\mc Q|_{\Omega^{\mk}}}) \to H^0(\mc Q|_{\Omega^{\mk} \cap \Omega^{\nd}}, \mls O_{\mc Q|_{\Omega^{\mk}\cap \Omega^{\nd}}})
\end{equation}
is injective.  Indeed a function on $\mc Q$ is determined by its restrictions to $\mc Q|_{\Omega^{\mk}}$ and $\mc Q|_{\Omega^{\nd}}$, and it follows from injectivity of \eqref{eq:r2} that a function on $\mc Q$ is determined just by its restriction to $\mc Q|_{\Omega^{\nd}}$. 

To prove compute the image of the restriction \eqref{eq:r}, observe first that the image of $H^0(\mc Q, \mls O_{\mc Q})_\lambda $ contains $H_\lambda $.  
To see this we have to show that the image of $H_\lambda $ under the restriction map from $\mc Q|_{\Omega_{\nd}}$ to $\mc Q|_{\Omega^{\mk} \cap \Omega^{\nd}}$ lies in the image of \eqref{eq:r2}. This restriction map is a ring homomorphism
\begin{align}
((R[x^{1/m}, u^{1/m}]_{F_u}/(x^{1/m}u^{1/m}-t^{1/m}))[X_m])^{\mu _m}&\rightarrow ((R[x^{1/m}, u^{\pm 1/m}]_{F_u}/(x^{1/m}u^{1/m}-t^{1/m}))[X_m])^{\mu _m}\notag \\
& \simeq R[v^{\pm }]_{F_v}/(x-tv)\otimes _{\mathbf{Z}[\mathbf{N}^r]}\mathbf{Z}[N_m], \label{eq:r3}
\end{align}
whereas \eqref{eq:r2} is the inclusion
\begin{equation}\label{E:restrictionb}
 R[v]/(x-tv)\otimes _{\mathbf{Z}[\mathbf{N}^r]}\mathbf{Z}[N_m]\rightarrow R[v^{\pm }]_{F_v}/(x-tv)\otimes _{\mathbf{Z}[\mathbf{N}^r]}\mathbf{Z}[N_m].
 \end{equation}
One checks that the image of $H_\lambda$ under \eqref{eq:r3} lies in the image of \eqref{E:restrictionb} by noting that the image under \eqref{eq:r3} of a term $ax^\alpha e^{\lambda }$ is $at^\alpha v^{\alpha -\chi (n_\lambda)}\otimes e^{n_\lambda}$. 

To prove that the resulting inclusion of finitely generated $R$-modules
\begin{equation}\label{E:inclusionc}
H_\lambda \hookrightarrow H^0(\mc Q, \mls O_{\mc Q})_\lambda 
\end{equation}
is an isomorphism, it suffices to show that it becomes an isomorphism after first making the faithfully flat base extension $\A\rightarrow \widehat \A$, and then after making the faithfully flat base extension $R\rightarrow \widehat R$, where $\widehat{\A}$ and $\widehat{R}$ are the completions of these rings along their respective maximal ideals.   
Furthermore, reducing modulo powers of the maximal ideal of $\widehat{\A}$ we may further assume that $t$ is a nilpotent element in $\widehat{\A}$. 

Now consider an element 
$$
((\sum _{j\geq 0}a_jx^j\cdot x^{w_\lambda/m})+ (\sum _{b\geq 0}c_bu^bu^{1-w_\lambda/m}))\cdot e^{\lambda}\in ((\A[[x^{1/m}]][u^{1/m}]_{F_u}/(x^{1/m}u^{1/m}-t^{1/m}))[X_m])^{\mu_m}_\lambda ,
$$
where the first series is infinite and the second finite, noting that the right hand side is $H^0(\mc Q|_{\Omega^{\mk} \cap \Omega^{\nd}}, \mls O_{\mc Q|_{\Omega^{\mk} \cap \Omega^{\nd}}})$ after the base extensions in the previous paragraph.
The image of this element in the isomorphic ring 
$$
R[v^{\pm }]_{F_v}/(x-tv)\otimes _{\mathbf{Z}[\mathbf{N}^r]}\mathbf{Z}[N_m]\simeq \A[v^{\pm }]_{F_v}\otimes _{\mathbf{Z}[\mathbf{N}^r]}\mathbf{Z}[N_m]
$$
(where the above isomorphism follows from the fact that $t$ is nilpotent) is the element
$$
(\sum _{j\geq 0}a_jt^{j+w_\lambda/m}v^{j-\chi (n_\lambda)+w_\lambda/m}+\sum _{b\geq 0}c_bv^{-b-1-\chi (n_\lambda)+w_\lambda/m})\otimes e^{n_\lambda}.
$$
If this element comes from an element of $H^0(\mc Q, \mls O_{\mc Q})$, then the coefficients of negative powers of $v$ must be zero, which implies that the $c_b$ are all $0$ and that $a_j$ is annihilated by $t^{j+w_\lambda/m}$ for $j+w_\lambda/m<\chi (n_\lambda)$.

From this 
it follows that \eqref{E:inclusionc} is an isomorphism.
\end{proof}

\begin{cor}\label{C:classical}
If $C_{\bar y}$ contains a unique marked point, then $\mls O_{\mls D_{(\bar y)}} \to Q_*\mls O_{\mls C_{(\bar y)}}$ is an isomorphism.
\end{cor}
\begin{proof}
It is enough to show that $R \otimes_{\mathbf{Z}[\mathbf{N}^n]} \mathbf{Z}[N_m]$ is an isomorphism in this case. We check this using \ref{L:9.14}, noting that the image of the $\lambda$-eigenspace of this module homomorphism is the last summand $R x^{\chi(n_\lambda)} e^\lambda$ of $H_\lambda$. So it is enough to show that $s_\lambda=0$. But when $C_{\bar y}$ contains a unique marked point, we have $N_m = (1/m)\mathbf{N}$, and a minimal lift of $\lambda$ can always be found in the interval $[0, 1)$. So $s_\lambda=0$ as required. 
\end{proof}
\begin{cor}\label{C:5.13c} If $\A$ is an integral domain then the cokernel of the map $R\otimes _{\mathbf{Z}[\mathbf{N}^n]}\mathbf{Z}[N_m]\rightarrow H^0(\mc Q, \mls O_{\mc Q})$ is flat over $\beta $.
\end{cor}
\begin{proof}
The image of the $\lambda$-eigenspace of this module homomorphism is the last summand $Rx^{\chi(n_\lambda)}e^\lambda$ of the $\A$-module $H_\lambda$ in \ref{L:9.14}, so the cokernel is a direct sum of copies of the $M_j$'s. Hence
it is enough to show that either $M_j = \beta $ or $M_j = 0$ for all $j=0 \ldots s_\lambda-1$. This follows from the definition of $M_j$ and the assumption that $\A$ is an integral domain, so in particular $t^{1/m}$ 
    is either $0$ or not a zerodivisor.
\end{proof}

\subsection{Remainder of the proof of the tails case}
We can now finish the argument for Case 2 (when $q$ contracts a tail). Since we have been working locally near $\bar y$ for some time we recall the notation given at the outset of Case 2 (see \ref{sec:case2}): we have $S = \Spec(\beta)$ where $\beta$ is Noetherian and strictly Henselian local, and $Q:\mls C \to \mls D$ is a contraction of curves over $S$ whose underlying morphism $q: C \to D$ of coarse spaces contracts a single rational tail of $C$ to the image of a geometric point $\bar y \to D$. 

We use the properties of $\mls O_{\mls D_{(\bar y)}} \to Q_*\mls O_{\mls C_{(\bar y)}}$ proved in \ref{L:tails} to obtain the following geometric result. Compare to Lemma \ref{L:5.7} and its proof, which holds when $q$ contracts a rational bridge.

\begin{lem}\label{L:6.22b} Let $\mls F$ be the stack over the small \'etale site of $D$ which to any $U\rightarrow D$ associates the groupoid of factorizations $g_U:\mls D_U\rightarrow \mc X$ of $f_U:\mls C_U\rightarrow \mc X$. 

\begin{itemize}
\item[(i)] For every \'etale $U\rightarrow D$ the groupoid $\mls F(U)$ is either empty or equivalent to the one-element set. If $C_{\bar y}$ contains at most one marking the latter always occurs.

\item[(ii)] Suppose $S = \Sp (\A)$ with $\A$ an integral domain with field of fractions $K$ and that $U\rightarrow D$ is an \'etale morphism such that 
$\mls F(U_K) \neq \emptyset$. Then $\mls F(U)\neq \emptyset $.
\end{itemize}
\end{lem}
\begin{proof}
    Note that for a given $U$ both statements can be verified \'etale locally on $U$, since $\mls F$ is evidently a stack for the \'etale topology, and by a standard limit argument it even suffices to verify the analogous statements for the stalk at a geometric point $\bar y\rightarrow D$:

    $(i)_{\bar y}$ The groupoid of factorizations of $f_{\bar y}: \mls C_{(\bar y)} \to \mc X$  is empty or equivalent to the unital set.  That is, if $g_{\bar y}:\mls D_{(\bar y)}\rightarrow \mc X$   of $f_{\bar y}$  is a factorization of $f_{\bar y}$, there are no nontrivial automorphisms of $g_{\bar y}$ that induce the identity on $f_{\bar y}$, and at most one $g_{\bar y}$ exists.

$(ii)_{\bar y}$ Suppose $S = \Sp (\A)$ with $\A$ an integral domain with field of fractions $K$, and that the restriction $f_{\bar y, K}$ of $f_{\bar y}$ to the generic fiber of $\mls C_{(\bar y)}$ factors through a map $g_{\bar y, K}:\mls D_{(\bar y), K}\rightarrow \mc X$.  Then there exists a unique factorization $g_{\bar y}:\mls D_{(\bar y)}\rightarrow \mc X$ of $f_{\bar y}$  extending $g_{\bar y, K}$.

To prove $(i)_{\bar y}$ and $(ii)_{\bar y}$ we can replace $\mc X$ with $\mc X_{(\bar y)} := \mc X \times_X D_{(\bar y)}$. The coarse space of $\mc X_{(\bar y)}$ is the affine scheme $D_{(\bar y)} = \Sp(R)$, so $\mc X_{(\bar y)} = [W/G]$ for a finite $D_{(\bar y)}$-scheme $W = \Sp(\widetilde R)$ and a tame \'etale group scheme $G$ over $D_{(\bar x)}$. The morphism $f_{(\bar y)}: \mls C_{(\bar y)} \to \mc X_{(\bar y)}$ is equivalent to a $G$-torsor $p: P \to \mls C_{(\bar y)}$ and a $G$-equivariant map $P \to W$. Since $\mls C_{\bar y} \to \mc X$ has abelian reduction, by \ref{prop:5.2-for-BG} the torsor $P$ descends to a $G$-torsor $p_0: P_0 \to \mls D_{(\bar y)}$ (unique up to unique isomorphism). The equivariant map $P \to W$ corresponds to a ring homomorphism 
\begin{equation}\label{eq:end1}
\widetilde{R} \to H^0(P, \mls O_P),\end{equation} 
and the category of factorizations of $f_{(\bar y)}$ is equivalent to the \textit{set} of factorizations of \eqref{eq:end1} through the canonical map $H^0(P_0, \mls O_{P_0}) \to H^0(P, \mls O_P)$.

To compute this set of factorizations, let 
$\mc F$ denote the cokernel of the map $\mls O_{\mls D_{(\bar y)}}\rightarrow Q_*\mls O_{\mls C_{(\bar y)}}$.  Since $p_0$ is finite flat we have that $p_{0, *}\mls O_{P_0}$ is locally free. Tensoring the short exact sequence
$$
0\rightarrow \mls O_{\mls D_{(\bar y)}}\rightarrow Q_*\mls O_{\mls C_{(\bar y)}}\rightarrow \mc F\rightarrow 0
$$
with $p_{0, *}\mls O_{P_0}$, we obtain the short exact sequence 
$$
0\rightarrow \mls O_{P_0}\rightarrow Q_*p_*\mls O_{P}\rightarrow p_{0, *}\mls O_{P_0}\otimes _{\mls O_{\mls D_{(\bar y)}}}\mc F\rightarrow 0.
$$
of sheaves on $\mls D_{(\bar y)}$, where the computation of the middle term follows from the projection formula. Taking global sections gives an exact sequence that forms the bottom row of the following diagram.
\begin{equation}\label{eq:end2}
\begin{tikzcd}
&&\widetilde{R} \arrow[d] \arrow[dl, dashrightarrow] & \\
0 \arrow[r] & H^0(P^0, \mls O_{P^0}) \arrow[r] & H^0(P, \mls O_P) \arrow[r] & H^0(\mls D_{(\bar y)}, p_{0, *}\mls O_{P_0}\otimes _{\mls O_{\mls D_{(\bar y)}}}\mc F)  
\end{tikzcd}
\end{equation}
As explained above, the set of factorizations of $f_{\bar y}$ is the set of dashed arrows that makes the triangle commute. Now use \ref{L:tails}: since $H^0(P^0, \mls O_{P^0}) \to H^0(P, \mls O_P)$ is injective (resp. an isomorphism if $C_{\bar y}$ has at most one marked point) we see that such an arrow is unique if it exists (resp. exists and is unique), proving $(i)_{\bar y}$. For $(ii)_{\bar y}$ assume $S = \Sp(\beta)$ with $\beta$ an integral domain with field of fractions $K$, so we have a restriction homomorphism
\begin{equation}\label{eq:end3}
H^0(\mls D_{(\bar y)}, p_{0, *}\mls O_{P_0}\otimes _{\mls O_{\mls D_{(\bar y)}}}\mc F) \hookrightarrow H^0(\mls D_{(\bar y), K}, p_{0K, *}\mls O_{P_0, K} \otimes_{\mls O_{\mls D_{(\bar y), K}}} \mc F_K)
\end{equation}
which is injective since by \ref{L:tails} the sheaf $\mc F$ is flat. Now a dotted arrow in \eqref{eq:end2} exists if and only if the composition $\downarrow {}_{\to}$ in that diagram is zero. But $\downarrow {}_{\to}$ is zero if and only if its composition with the injective homomorphism \eqref{eq:end3} is zero, and this triple composition is zero exactly when $f_{\bar y, K}$ has a factorization.

\end{proof}

We now complete the proof of \ref{T:9.2} in the tails case.
 Let $\Sigma$ be the fiber product of the diagram
$$
\xymatrix{
& S\ar[d]^-f\\
\Hom (\mls D, \mc X)\ar[r]& \Hom (\mls C, \mc X).}
$$
By setting $U=D$ in \ref{L:6.22b} we obtain the following:

\begin{itemize}
\item[(i)] The map $\Sigma \rightarrow S$ is a monomorphism, and if $q$ is classical then it is an equivalence.

\item[(ii)] For a morphism $\Sp (A)\rightarrow S$, with $A$ an integral domain with field of fractions $K$, the set $\Sigma (A)$ is nonempty if and only if $\Sigma (K)$ is nonempty.
\end{itemize}

By (i) and \cite[\href{https://stacks.math.columbia.edu/tag/0463}{Tags 0463 and 0418}]{stacks-project} we conclude that $\Sigma \to S$ is separated and that $\Sigma$ is a scheme (note that $f$ is locally of finite type by \cite[\href{https://stacks.math.columbia.edu/tag/06U9}{Tags 06U9}]{stacks-project}).  

The map $\Sigma \to S$ is quasi-compact.  Indeed, there is a commuting square
\[
\begin{tikzcd}
\Hom(\mls D, \mc X) \arrow[r] \arrow[d] & \Hom(\mls C, \mc X) \arrow[d] \\
\Hom(D, X) \arrow[r] & \Hom(C, X),
\end{tikzcd}
\]
where 
the vertical maps are quasicompact and quasiseparated by \cite[C.2(ii), C.3]{AOV} and 050M(4). It follows that the map
$$
\Hom(\mls D, \mc X)\rightarrow \Hom(\mls C, \mc X)\times _{\Hom(C, X)}\Hom (D, X)
$$
is quasi-compact.  The map $\Sigma \rightarrow S$ is the base change of this map along the morphism $S\rightarrow \Hom(\mls C, \mc X)\times _{\Hom(C, X)}\Hom (D, X)$ and therefore is also quasi-compact.

 Then (ii) implies that the map $\Sigma \rightarrow S$ satisfies the valuative criterion for properness, and therefore is a proper monomorphism.  By \cite[\href{https://stacks.math.columbia.edu/tag/04XV}{Tag 04XV}]{stacks-project} we conclude that $\Sigma \hookrightarrow S$ is a closed immersion (or an isomorphism if $q$ is classical) which completes the proof of \ref{T:9.2} in this case, and therefore also completes the proof of \ref{T:9.2}. \qed

\section{Stabilization}\label{S:section11}

\begin{pg}
Let $\mc X$ be a separated tame Deligne-Mumford stack of finite type over a noetherian base $B$. Fix nonnegative integers $g, n$ and a collection of weights $\{a_i\}_{i=1}^n$ satisfying 
$a_i \in (0, 1]\cap \mathbf{Q}.$
Let $(\bC, f:\mls C \to \mc X)$ be a prestable map of type $(g, \ba)$ over a $B$-scheme $S$. Let $T \to S$ be a morphism. The \textit{category of stabilizations} of $f_T: \mls C_T \to \mc X$ is the category of triples 
\[((q, \phi'): \bC_T \to \bD, g: \mls D \to \mc X, \alpha)\]
where $(q, \phi'):\bC_T \to \bD$ is a contraction (with associated stack morphism $Q:\mls C_T \to \mls D$) such that $\bar f: C \to X$ has a factorization $C \xrightarrow{q} D \xrightarrow{\bar g} X$ (this is unique since contractions of ordinary curves are epimorphisms of separated algebraic spaces \cite[4.3(v)]{OWI}), the pair $(\bD, g)$ is a stable map of type $(g, \ba)$, and $\alpha$ is a 2-isomorphism from $g \circ Q$ to $f_T$. A morphism from $((q_1, \phi'_1), g_1, \alpha_1)$ to $((q_2, \phi'_2), g_2, \alpha_2)$ exists only if $(q_1, \phi'_1) = (q_2, \phi'_2) =: (q, \phi')$, and in this case a morphism
from $((q, \phi'), g_1, \alpha_1)$ to $((q, \phi'), g_2, \alpha_2)$ is a 2-isomorphism from $g_1$ to $g_2$ that is compatible with $\alpha_1$ and $\alpha_2$.


\begin{lem}
The category of stabilizations is equivalent to a set.
\end{lem}
\begin{proof}
Let $(g_1, \alpha_1)$ and $(g_2, \alpha_2)$ be two stabilizations along the contraction $\bC \to \bD$. This contraction factors by \cite[6.15]{OWI} through an initial contraction $\bC \to \bD^\dagger \to \bD$ and by \cite[6.11]{OWI}, since the maps $g_i$ are representable, the induced map of stacks $\mls D^\dagger \to \mls D$ is the relative coarse space of both the compositions
\begin{equation}\label{eq:set1}
\mls D^\dagger \to \mls D \xrightarrow{g_i} \mc X.
\end{equation}
By the universal property of coarse moduli spaces the set of 2-isomorphisms of $g_1$ and $g_2$ is isomorphic to the set of 2-isomorphisms of the compositions \eqref{eq:set1}. But since $\bC \to \bD^\dagger$ is initial,  by \ref{T:9.2} there is at most one 2-isomorphism of the compositions in \eqref{eq:set1} compatible with the $\alpha_i$.
\end{proof}

Thanks to this lemma we will hereafter speak of the \textit{set of stabilizations}. 
The main result of this section is the following.
\end{pg}

\begin{thm}\label{T:stabilize}
Let $(\bC, f:\mls C \to \mc X)$ be a prestable map of type $(g, \ba)$ over a $B$-scheme $S$. Then the functor which to any $S$-scheme $T$ associates the set of stabilizations is represented by a closed subscheme of $S$. If $\mc X = BG$ for some finite \'etale group scheme $G$ over $B$ of order invertible in $B$, then this functor is represented by $S$ itself.
\end{thm}

\begin{proof}
Without loss of generality we may assume that  $S=B$ (this simplifies the notation), and by a standard limit argument we may also assume that $S$ is noetherian. Let $\Psi$ denote the functor in the theorem statement.  The functor $\Psi$ is a sheaf with respect to the big \'etale topology on $B$, and therefore by descent for closed subschemes \cite[\href{https://stacks.math.columbia.edu/tag/0247}{Tag 0247}]{stacks-project} it suffices to prove the theorem after replacing $B$ by an \'etale cover. 

The proof now proceeds in several steps (\ref{step1}--\ref{SS:end}).

\subsection{Theorem \ref{T:stabilize} holds when $B = \Sp(k)$ for $k$ an algebraically closed field}\label{step1}

Let $k$ be an algebraically closed field.
Fix a prestable map $(\bC, f: \mls C \to \mc X)$  over $k$ of type $(g, \ba)$ as in the theorem. Recall that if $C$ is a nodal curve over $k$, a rational tail (resp. rational bridge) is a component with exactly one node (resp. two nodes).
We say that a rational tail or bridge $E$ of $C$ is \textit{unstable for $f$} if the morphism of stacks $\mls E \to \mc X$ has an abelian contraction and
\[
\sum_{i=1}^n a_i + \#\text{nodes of }E \leq 2.
\]
Note that if $E$ is a rational bridge this simply means that $E$ has no marked points, and if $E$ is a rational tail then $\sum _{i=1}^na_i\leq 1$.

We will show that the following algorithm produces the unique element of $\Psi(k)$, or returns the empty set if $\Psi(k) = \emptyset$. As motivation we refer to \cite[4.14(v)]{OWI}.
 
\begin{alg}\label{A:only}
\color{white} create a new line \color{black}
\begin{itemize}
\item[(i)] Let $C=C_1 \to C_2$ contract all unstable rational tails of $C$ and let $\bC_1 \to \bC_2$ be the associated initial contraction. By \ref{T:9.2}, a factorization of $f:\mls C_1 \to \mc X$ through $\mls C_2$ is unique if it exists. If no factorization exists then $\Psi(k)=\emptyset$; otherwise proceed.
\item[(ii)] Repeat (i) as many times as necessary: at each stage, contract all unstable rational tails of $C_i$ to form $C_{i+1}$ and factor the morphism $\mls C_i \to \mc X$ through $\mls C_{i+1}$ if possible. 
\item[(iii)] Once all rational tails of $C_{\ell-1}$ are stable, let $C_{\ell-1} \to C_\ell = C^{stab}$ contract all unstable rational bridges (note that all rational bridges are unstable by \ref{P:newabelian}), 
and let $\bC_{\ell-1} \to \bC_\ell$ be the associated initial contraction. The morphism $\mls C_{\ell-1} \to \mc X$ factors through a morphism $\mls C_{\ell} \to \mc X$ by \ref{T:9.2}. 
\item[(iv)] By \cite[6.11]{OWI} there is a unique contraction $\bC_\ell \to \bC^{stab}$ inducing an isomorphism of coarse spaces such that $\mls C_\ell \to \mc X$ factors through a representable morphism $f^{stab}:\mls C^{stab} \to \mc X$. The pair $(\bC^{stab}, f^{stab})$ is the unique element of $\Psi(k)$.
\end{itemize}
\end{alg}

We note that in step (iii), contracting rational bridges will not introduce any new rational tails. Also, once the contraction $\bC \to \bC^{stab}$ has been constructed, $f^{stab}$ is the unique factorization of $f$ through $\mls C^{stab}$ coming from \ref{T:9.2}.

 If Algorithm \ref{A:only} produces an object it is stable by \ref{D:5.7}. That is, \ref{A:only} produces an element of $\Psi(k)$.
We must show that, if $(\bD, g: \mls D \to \mc X)$ is a stable map of type $(g, \ba)$ and $\mls C \to \mls D \xrightarrow{g} \mc X$ is a factorization of $f$ through a contraction $\bC \to \bD$, then $\Psi (k)\neq \emptyset $ and $(\bD, g) = (\bC^{stab}, f^{stab})$.
For this, let
\[
C = D_1 \to D_2 \to \ldots \to D_p =D
\]
be the greedy factorization of $C \to D$ (see \cite[4.14(v)]{OWI}). Associated to this factorization we have a factorization of $\bC \to \bD$ as a sequence of initial contractions $\bC = \bD_1 \to \bD_2 \to \ldots \to \bD_p$ followed by a contraction $\bD_p \to \bD$ that induces an isomorphism of coarse curves. We argue via induction on $i$ that $D_i = C_i$ and hence $\mls D_i = \mls C_i$ (note, however, that this will not yet imply $\mls D = \mls C^{stab}$). 

First assume that $i$ is less than $\ell-1$ and $p-1$. An extremal branch of $C_i = D_i$ that is contracted in $D_{i+1}$ is  unstable, hence contracted in $C_{i+1}$. Conversely, if $E \subset D_i$ is an extremal branch not contracted in $D_{i+1}$, its image in $D$ is an extremal branch. Although the associated stack $\mls E$ may not be isomorphic to its image in $\mls D$, by \ref{L:coarse-abelian} below the image will still have an abelian contraction and hence be unstable.
So we must have $C_{i+1} = D_{i+1}$. Similarly if $p > \ell$ then the component of $C_{\ell-1} = D_{\ell-1}$ contracted in $D_\ell$ must be an unstable rational tail, a contradiction, so $p=\ell$. Finally, when $i=\ell-1$, we see that $D_{\ell-1} \to D_\ell$ must contract all the rational bridges since these are all unstable by \ref{P:newabelian}. This shows that $\ell=p$ and $C_\ell=D_\ell$.

We have moreover two factorizations $\mls C_\ell \to \mls D \to \mc X$ and $\mls C_\ell \to \mls C^{stab} \to \mc X$, where $\mls C_\ell \to \mls D, \mls C^{stab}$ induce isomorphisms of coarse curves and $\mls D, \mls C^{stab} \to \mc X$ are representable. By \cite[6.11]{OWI}, both these factorizations are relative coarse spaces of $\mls C \to \mc X$. By uniqueness of the relative coarse space we obtain $(\bD, g) = (\bC^{stab}, f^{stab}).$

This completes the proof of the case when $B =\Sp (k)$ is the spectrum of a field, modulo the following lemma.

\begin{lem}\label{L:coarse-abelian}
Let $\bE$ and $\bE^\dagger$ be generalized log twisted curves over an algebraically closed field, with coarse spaces both equal to $\mathbf{P}^1$, and let $\bE \to \bE^\dagger$ be a contraction over $\mathbf{P}^1$. Let $\mls E^\dagger \to \mc X$ be a morphism. If the restriction $\mls E \to \mc X$ has an abelian contraction, so does $\mls E^\dagger \to \mc X$.
\end{lem}
\begin{proof}
The map $\pi :\mls E\rightarrow \mls E^\dag $ is its own relative coarse moduli space: Let 
$$
\xymatrix{
\mls E\ar@/^2pc/[rr]^-{\pi }\ar[r]^-a& \mls E'\ar[r]^-b& \mls E^\dag }
$$
be the relative coarse moduli space of $\pi $, so in particular $b$ is representable, and note that by \cite[6.11]{OWI} there is a factorization of generalized log twisted curves
$$
\xymatrix{
\mathbf{E}\ar[r]& \mathbf{E}'\ar[r]& \mathbf{E}^\dag }
$$
inducing this factorization of stacks. By  \cite[6.12]{OWI} we have that $\mathbf{E}'\rightarrow \mathbf{E}^\dag $ and hence $\mls E'\rightarrow \mls E^\dag $ are isomorphisms.
For a geometric point $\bar x \to \mathbf{P}^1$ let $\Gamma_{\bar x}$ (resp. $\Gamma_{\bar x}^\dagger$) denote the stabilizer group at $\bar x$ in $\mls E$ (resp. $\mls E^\dagger$), so we have a surjection $\Gamma _{\bar x}\rightarrow \Gamma _{\bar x}^\dagger $.

Since $\mls E \to \mls E^\dagger$ is its own relative coarse moduli space, the category of quasi-coherent sheaves on $\mls E^\dag $ is equivalent to the category of quasi-coherent sheaves $\mls F$ on $\mls E$ such that for all geometric points $\bar x\rightarrow \mathbf{P}^1$ the action of $\text{Ker}(\Gamma _{\bar x}\rightarrow \Gamma _{\bar x}^\dagger )$ on $\bar x^*\mls F$ is trivial (see for example \cite[2.3.4]{AV}).  This equivalence is monoidal (compatible with tensor products), and therefore it induces an equivalence between categories of torsors under \'etale group schemes.

Turning now to the proof of the lemma, assume that $\mls E\rightarrow \mc X$ has an abelian contraction.   Then the map $\mathbf{P}^1\rightarrow X$ factors thorugh a point.  Replacing $\mc X$ by the residual gerbe of this point we may assume that $\mc X = BG$ for a tame group $G$.  The map $\mls E^\dag \rightarrow \mc X$ is therefore given by a $G$-torsor $P^\dag \rightarrow \mls E^\dag $ whose pullback $P\rightarrow \mls E$ admits a reduction $Q$ to an abelian subgroup $A\subset G$.  Note that there is a natural inclusion $Q \subset P$, and since the action of the groups $\text{Ker}(\Gamma _{\bar x}\rightarrow \Gamma _{\bar x}^\dagger )$ on the stabilizers $\bar x^*P$ is trivial, the same is true for the actions on the schemes $\bar x^*Q$.  We conclude that $Q$ descends to an $A$-torsor $Q^\dag \rightarrow \mls E^\dag $ which induces $P^\dag $.

\end{proof}

We now turn to the proof of \ref{T:stabilize} in the general case.  Fix $(\bC , f)$ as in \ref{T:stabilize}. 

\subsection{Subfunctors $\Psi _{D}\subseteq \Psi$}

The induced map of coarse spaces $\bar f: C \to X$ is prestable of type $(g, \ba)$. Let 
$$
\xymatrix{
C\ar[r]^-p& C^*\ar[r]^-{}& X}
$$
be its stabilization \cite[4.11]{OWI}.
For every factorization of $p$ into a sequence of contractions 
\begin{equation}\label{E:6.7.1}
\xymatrix{
C\ar[r]^-{}\ar@/^2pc/[rr]^-p& D\ar[r]^-{}& C^*}
\end{equation}
define $\Psi _{D}\subseteq \Psi $ be the subfunctor 
of objects $((q, \phi'): \bC \to \bD, g, \alpha)$ for which $q$ is equal to the given contraction $C \to D$. 

\begin{lem}\label{L:6.8a}
Let $\bar b \to B$ be a geometric point. There exists an \'etale neighborhood $U$ of $\bar b$ and a finite set $J$ of factorizations of $p$ as a composition of contractions
\[
C_U \to D^{(j)}_U \to C^*_U, \quad \quad \quad \quad j \in J,
\]
such that, for each geometric point $\bar u \to U$, we have
\[
\Psi(\bar u) = \bigcup_{j \in J} \Psi_{D^{(j)}}(\bar u).
\]
\end{lem}
\begin{proof}
Given a geometric point $\bar b\rightarrow B$, let $J$ be the set of factorizations
$$
\xymatrix{
C_{\bar b}\ar[r]& D^{(j)}_{\bar b}\ar[r]& C^*_{\bar b}}
$$
in the fiber over $\bar b$.  Note that the set $J$ is finite.  By \cite[4.17]{OWI}, there is an \'etale neighborhood $U'$ of $\bar b$ where there exists for each $j\in J$ a factorization
$$
\xymatrix{
C_{U'}\ar[r]& D^{(j)}_{U'}\ar[r]& C^*_{U'}}
$$
inducing the factorization corresponding to $j$ in the fiber over $\bar b$. Moreover $U'$ has a finite stratification such that the topological type of $C_{U'}$ over each stratum is constant. We obtain $U$ by removing from $U'$ the closures of those strata whose closure does not contain the image of $\bar b$.
\end{proof}

\begin{lem}\label{L:6.8.2} Fix a factorization \eqref{E:6.7.1} and let $\bar b\rightarrow B$ be a geometric point with $\Psi _D(\bar b)\neq \emptyset $.  Then there exists an \'etale neighborhood $U$ of $\bar b$ such that the restriction of $\Psi _D$ to $U$ is represented by a closed subscheme of $U$.
\end{lem}
\begin{proof}
 By \ref{L:2.14b} we can find an \'etale neighborhood  $U$ of $\bar b$ so that for every geometric point $\bar y\rightarrow D$ the restriction of $f$ to the fiber $\mls C_{\bar y}$ admits an abelian contraction. 
  By \ref{T:9.2} the groupoid of factorizations of $\mls C \to \mc X$ through the initial contraction over $C \to D$ is represented by a closed subscheme of $U$. By openness of stability \ref{P:6.4b}, by passing to a further \'etale neighborhood we can ensure that this closed subscheme represents $\Psi_D$.
\end{proof}

\subsection{Proof of \ref{T:stabilize} in two steps}
We complete the proof of \ref{T:stabilize} in the following steps, which we will show in turn.
\begin{enumerate}
    \item Let $Z\subset B$ be the set of points $b\in B$ such that for a geometric point $\bar b\rightarrow B$ over $b$ the set $\Psi (\bar b)$ is nonempty.  Then $Z$ is a closed subset of $B$.
    \item For any geometric point $\bar b\rightarrow B$ with $\Psi (\bar b) \neq \emptyset $ there exists an \'etale neighborhood $U$ of $\bar b$ and a contraction $C_U\rightarrow D$ over $U$ such that $\Psi|_U = \Psi_D$.
    
\end{enumerate}

\begin{rem}
At this point we can show directly that $\Psi$ is a subfunctor of $B$, or equivalently that for any $T \to B$ the set $\Psi(T)$ is either empty or has a unique element. To see this,
fix a scheme $T\rightarrow B$ and two stabilizations over $T$
$$
((q_i, \phi _i'): \bC_T \to \bD _i, g_i: \mls D _i \to \mc X, \alpha_i), \ \ i=1,2.
$$
By \ref{step1} the two contractions $C_T\rightarrow D_i$ agree in every geometric fiber, and therefore by \cite[4.9]{OWI} are equal; say $D_1 = D_2 = D$.  Then these two stabilizations define two elements of $\Psi _D(T)$ and therefore they are equal by \ref{L:6.8.2}.
\end{rem}

\begin{lem} Statements (1)-(2) imply \ref{T:stabilize}.
\end{lem}
\begin{proof}
Note that $\Psi $ is a sheaf for the \'etale topology.  By descent for closed immersions \cite[\href{https://stacks.math.columbia.edu/tag/03I0}{Tag 03I0}]{stacks-project}, it therefore suffices to show that for every geometric point $\bar b\rightarrow B$ there exists an \'etale neighborhood $U$ of $\bar b$ such that $\Psi |_U$ is represented by a closed subscheme of $U$.  If $\Psi (\bar b) = \emptyset $, using (1), let $U$ be an \'etale neighborhood whose image (which is open since $U\rightarrow B$ is \'etale) does not meet $Z$. Then $\Psi |_U$ is represented by the empty scheme.  If $\Psi (\bar b)\neq \emptyset$ let $U$ be as in (2). Then the result follows from \ref{L:6.8.2}.

When $\mc X = BG$ for a finite group $G$ of order invertible in $B$, the case of an algebraically closed field (in \ref{step1}) implies that $Z = B$, and then \ref{T:9.2} implies that $\Psi $ is represented by $B$ itself.
\end{proof}

\subsection{Proof of (1)}
We first claim that $Z$ is constructible. By \ref{L:6.8a} it suffices to prove the analogous statement for the functors $\Psi _D$ which follows from \ref{L:6.8.2}.
Since $Z$ is constructible, to prove (1) it suffices to show that $Z$ is stable under specialization. This follows from the following lemma.
\begin{lem}\label{L:6.13} Let $A$ be a discrete valuation ring with separably closed residue field, let $a:\Sp (A)\rightarrow B$ be a morphism, and let $\bar \eta \rightarrow \Sp (A)$  be a geometric generic  point.  If $\Psi (\bar \eta )\neq \emptyset $ then $\Psi (\Sp (A))\neq \emptyset .$
\end{lem}
\begin{proof}
Without loss of generality we may replace $B$ by $\Sp (A)$, so can assume $\Sp (A) = B$.
Denote by $\bar b\in B$ the closed point.
  Let $(\bD_{\bar \eta }, g_{\bar \eta })$ be the unique object of $\Psi(\bar \eta )$, so $(\bD_{\bar \eta }, g_{\bar \eta })$ is a stable map of type $(g, \ba)$ through which $(\bC_{\bar \eta }, f_{\bar \eta })$ factors.

 The following algorithm will produce an element of $\Psi(B)$.
\begin{alg}\label{A:ok-there's-two} 
\begin{itemize}
\item[(i)] Let $C = C_1 \to C_2$ be the unique contraction over $B$ that contracts all unstable rational tails of $C_{1, \bar b}$ and let $\bC_1 \to \bC_2$ be the associated initial contraction. The morphism $f:\mls C_1 \to \mc X$ always factors through $\mls C_2$: indeed, by the Claim proved below, $C_{\bar \eta } \to C_{2, \bar \eta }$ factors $C_{\bar \eta } \to D_{\bar \eta }.$ Granting this, the stable map $(\bD_{\bar \eta }, g_{\bar \eta })$ restricts to a generic factorization of $f : \mls C \to \mc X$ through $\mls C_2$, and hence a factorization over all of $B$ exists since the functor of factorizations in \ref{T:9.2} is represented by a closed subscheme of $B$.

\item[(ii)] Repeat (i) as many times as necessary: at each stage, $C_i \to C_{i+1}$ is the unique contraction over $B$ that contracts all unstable rational tails of $C_{i, \bar b}$.

\item[(iii)] Once all rational tails of $C_{\ell-1, \bar b}$ are stable, let $C_{\ell-1} \to C_\ell = C^{stab}$ contract all rational bridges, and let $\bC_{\ell-1} \to \bC_\ell$ be the associated initial contraction. The morphism $\mls C_{\ell-1} \to \mc X$ factors through a morphism $\mls C_\ell \to \mc X$ by \ref{T:9.2}.
\item[(iv)] There is a unique contraction $\bC_\ell \to \bC^{stab}$ inducing an isomorphism of coarse spaces such that $\mls C_\ell \to \mc X$ factors through a representable morphism $f^{stab}: \mls C^{stab} \to \mc X$. The pair $(\bC^{stab}, f^{stab})$ is stable because its fiber over $\bar b$ is stable by construction (see \ref{D:5.7}) and the stable locus is open by \ref{P:6.4b}.
\end{itemize}
\end{alg}


We prove the Claim in step (i): namely, we show that each time we form $C_i$, the fiber $C_{\bar \eta } \to C_{i, \bar \eta }$ factors $C_{\bar \eta } \to D_{\bar \eta }.$ We do this by induction on $i$, so we assume $C_{\bar \eta } \to C_{i-1, \bar \eta }$ factors $C_{\bar \eta } \to D_{\bar \eta }$. Using the uniqueness in \ref{step1} (the field case),  it is enough to show that $C_{i-1, \bar \eta } \to C_{i, \bar \eta }$ contracts 
rational tails that are unstable for the composition $\mls C_{i-1, \bar \eta } \to \mls D_{\bar \eta } \to \mc X$. Let $E_{\bar \eta } \subset C_{i-1, \bar \eta }$ be a component contracted by $C_{i-1, \bar \eta} \to C_{i, \bar \eta}$. Its closure is a subscheme $E$ of $C_{i-1}$ flat and proper over $B$ such that $E_{\bar b}$ is contracted in $C_{i, \bar b}$. 
Since $C_{i-1, \bar b} \to C_{i, \bar b}$ contracts rational tails, and these do not meet each other, it follows that  $E_{\bar b}$ is one of the components contracted by $C_{i-1, \bar b} \to C_{i, \bar b}$; i.e., $E_{\bar b}$ is a rational tail unstable for $\mls C_{i-1, \bar b} \to \mc X$. In particular $\mls E_{\bar b} \to \mc X$ has an abelian contraction and it follows from \ref{L:2.14b} that $\mls E_{ \bar \eta } \to \mc X$ has an abelian contraction. Likewise by computing on $E_{\bar b}$ we see that $\sum_{i=1}^n a_i \leq 1$, so $E_{\bar \eta }$ is unstable.
\end{proof}

\subsection{Proof of (2)}\label{SS:end}
Fix an \'etale neighborhood $U$ of $\bar b$ over which the contraction corresponding to $\Psi(\bar b) \neq \emptyset$ extends to a contraction $C_U\rightarrow D$ over $U$ (see \cite[4.17]{OWI}).  
Let $W\subset Z\times_B U$ be the subset of points such that a geometric point $\bar y\rightarrow Z\times_B U$ is in $W$ if and only if the fiber $D_{\bar y}$ is not equal to the contraction that occurs from the fact that $\Psi (\bar y)\neq \emptyset $.  
\begin{lem} After possibly replacing $U$ by another \'etale neighborhood of $\bar b$ we can arrange that the set $W$ is closed.
\end{lem}
\begin{proof}
Replacing $U$ by a smaller \'etale neighborhood of $\bar b$ if necessary, we may by \ref{L:6.8a} assume that there exists a finite set $J$ of contractions $C_U \to D'$ 
such that, if $\bar u \to U$ is a geometric point such that $\Psi(\bar u) \neq \emptyset$, then the associated contraction is the fiber of one of the contractions indexed by $J$.
For each contraction $C_U\rightarrow D'$ in $J$, the set of points where $D'$ agrees with $D$ is open by \cite[4.9]{OWI}. Hence the intersection of its complement (which is closed) with the support of $\Psi _{D'}$ (which is locally closed by \ref{L:6.8.2}) is a locally closed subset $W_{D'}\subset Z\times_B U$.  The union of these locally closed subsets $W_{D'}$ for $C_U \to D'$ in $J$ is equal to $W$, implying that $W$ is at least locally closed.

To conclude that $W$ is closed it
 suffices to show that $W$ is closed under specialization.   Let $A$ be a discrete valuation ring with separably closed residue field, let $a:\Sp (A)\rightarrow Z$ be a morphism, and let $\bar \eta \rightarrow \Sp (A)$  be a geometric generic  point.  Assume that $\bar \eta $ maps to $W$.  We show that the closed point $s\in \Sp (A)$ also maps to $W$.  For this note that by \ref{L:6.13} we get a contraction $C_A\rightarrow D_A'$ over $\Sp (A)$ whose fibers are the contractions arising from the fact that the $\Psi (\bar \eta )$ and $\Psi (s)$ are nonempty, and for which $D'_{\bar \eta }\neq D_{\bar \eta }$.  By \cite[4.9]{OWI} again it follows that $D'_s\neq D_s$ implying that $s$ maps to $W$.
\end{proof}

 We return to the proof of (2). Since $W$ is closed and does not contain $\bar b$, we can replace $U$ by $U \setminus W$. On this neighborhood we have $\Psi=\Psi_D$ on all geometric points $\bar u \to U$, and hence $\Psi = \Psi_D$ by \cite[4.9]{OWI}. This proves (2). 

This completes the proof of \ref{T:stabilize}. 
\end{proof}

\subsection{Applications of stabilization}
We note some applications of \ref{T:stabilize}. Let $\{b_i\}_{i=1}^n$ be a second collection of weights satisfying $b_i \in (0, 1]\cap \mathbf{Q}$, and assume $a_i \leq b_i$ for all $i$. We say that a prestable map to $\mc X$ of type $(g, \bb)$ \textit{admits an $\ba$-stabilization} if it has a factorization through a stable map of type $(g, \ba)$ (the factorization is unique when it exists by \ref{T:stabilize}). In Section \ref{S:section6b} we defined the stacks $\mls S_{g, \bb}(\mc X)$ of prestable maps; we now let $\mls S_{g, \bb}^{\ba-\mathrm{stab}}(\mc X) \subseteq \mls S_{g, \bb}(\mc X)$ be the subcategory of objects admitting an $\ba$-stabilization.

\begin{cor}\label{C:stab}
The subcategory $\mls S_{g, \bb}^{\ba-\mathrm{stab}}(\mc X) \subseteq \mls S_{g, \bb}(\mc X)$ is a closed substack, and there is a morphism
\[
\mathrm{stab}: \mls S_{g, \bb}^{\ba-\mathrm{stab}}(\mc X) \to \mls K_{g, \ba}(\mc X)
\]
sending a $\bb$-prestable map to its $\ba$-stabilization.
\end{cor}
\begin{proof}
This is immediate from \ref{T:stabilize}, noting that we have an open immersion $\mls S_{g, \bb}(\mc X) \hookrightarrow \mls S_{g, \ba}(\mc X)$.
\end{proof}

\begin{rem} When $\cX$ is a scheme the morphism in \ref{C:stab} was discussed in \cite[3.1]{AlexeevGuy} and \cite[1.2.1]{bayermanin}. In this case the substack $\mls S_{g, \bb}^{\ba-\mathrm{stab}}(\mc X)$ is equal to $ \mls S_{g, \bb}(\mc X)$.
\end{rem}

Similarly, we let $\mls K_{g, \bb}^{\ba-\mathrm{stab}}(\mc X) \subseteq \mls K_{g, \bb}(\mc X)$ denote the subcategory of objects admitting an $\ba$-stabilization. Note that this inclusion fits into a fiber square
\[
\begin{tikzcd}
\mls K_{g, \bb}^{\ba-\mathrm{stab}}(\mc X) \arrow[r] \arrow[d] &  \mls K_{g, \bb}(\mc X)\arrow[d]\\
\mls S_{g, \bb}^{\ba-\mathrm{stab}}(\mc X) \arrow[r]  &  \mls S_{g, \bb}(\mc X)
\end{tikzcd}
\]
hence is a closed immersion by \ref{C:stab}. The morphism in \ref{C:stab} induces an analogous morphism
\begin{equation}\label{eq:stab1}
\mathrm{stab}: \mls K_{g, \bb}^{\ba-\mathrm{stab}}(\mc X) \to \mls K_{g, \ba}(\mc X).
\end{equation}
Let $\bone^n$ denote the vector $(1, 1, \ldots, 1)$ with $n$ entries.
\begin{cor}\label{C:small}
There is a closed substack $\mls K^{\mathrm{stab}}_{g, \bone^n}(\mc X) \subseteq \mls K_{g, \bone^n}(\mc X)$ of stable maps with distinct markings that admit $\ba$-stabilizations for every weight vector $\ba \in (0, 1]^n\cap \mathbf{Q}^n$. In particular the substack $\mls K^{\mathrm{stab}}_{g, \bone^n}(\mc X)$ is a compactification of the stack of maps with smooth source curve.
\end{cor}
\begin{proof}
Even though there are infinitely many possibilities for rational weight vectors $\ba \in (0, 1]^n\cap \mathbf{Q}^n$, we will show that there are only finitely many distinct substacks $\mls K^{\ba-\mathrm{stab}}_{g, \bone^n}(\mc X)$ of $\mls K^{\mathrm{stab}}_{g, \bone^n}(\mc X).$ Then we can take $\mls K^{\mathrm{stab}}_{g, \bone^n}(\mc X)$ to be the intersection of these finitely many substacks.

The claimed finiteness holds by an argument similar to that used in \cite[5.1]{Hassett} and \cite[2.5]{AlexeevGuy}: namely, suppose that $\ba$ and $\bb$ are two weight vectors that satisfy, for each $I \subseteq\{1, \ldots, n\}$ of size at least two, 
\begin{equation}\label{eq:compact1}
\sum_{i \in I} a_i \leq 1 \quad \quad \quad \text{if and only if} \quad \quad \quad \sum_{i \in I} b_i \leq 1.
\end{equation}
We will show that a prestable map $(\bC, f)$ over $S$ of type $(g, \ba)$ has an $\ba$-stabilization if and only if it has a $\bb$-stabilization (and that if they exist, the stabilizations are equal). Granting this, we define a \textit{chamber} to be a maximal subset $\mc C \subset (0, 1]^n\cap \mathbf{Q}^n$ such that for each $I \subseteq \{1, \ldots, n\}$ of size at least two, we have either $\sum_{i \in I} x_i \leq 1$ for all $\mathbf{x} \in \mc C$ or $\sum_{i \in I} x_k > 1$ for all $\mathbf{x} \in \mc C$. 
Then the number of distinct substacks $\mls K^{\ba-\mathrm{stab}}_{g, \ba}(\mc X)$ is bounded by the number of distinct chambers.

Therefore, to conclude the proof, it is enough to assume $S$ is strictly Henselian local and that $(\bC, f)$ admits an $\ba$-stabilization, and show that it also admits a $\bb$-stabilization. In this case the proof of \ref{T:stabilize} step (b) gives an explicit construction of the $\ba$-stabilization: it arises from the algorithm \ref{A:ok-there's-two}. Each time we iterate step (i), the desired factorization exists (over all of $S$) by assumption that $(\bC, f)$ admits an $\ba$-stabilization. Now the key point is that because of \eqref{eq:compact1}, a rational tail or bridge is unstable for $\ba$ if and only if it is unstable for $\bb$. So the algorithm \ref{A:ok-there's-two} for $\bb$ runs exactly as for $\ba$, and produces the same stabilization (in particular, not the empty set).
\end{proof}

\section{Properness of the moduli of stable maps}\label{S:proper}
 
Let $\mc X$ be a separated tame Deligne-Mumford stack of finite type over a noetherian base scheme $B$.
If the coarse moduli space $X$ of $\mc X$ is projective over $B$, we moreover fix a relatively ample invertible sheaf $\mls O_X(1)$ on $X$.  In this case we define the degree of an object $(\mathbf{C},f)\in \mls K_{g,  \ba }(\mc X)$ to be the degree with respect to $\mls O_X(1)$ of the induced map $C\rightarrow X$ on coarse spaces.  Let $\mls K_{g, \ba }(\mc X, d)\subset \mls K_{g, \ba }(\mc X)$ be the substack of stable maps of degree $d$.

\begin{thm}\label{T:proper}
The morphism $\mls K_{g, \ba}(\mc X) \to B$ has finite diagonal and satisfies the existence part of the valuative criterion if $\mc X$ is proper. If moreover the coarse space $X$ admits a $B$-relatively ample line bundle $\mls O_X(1)$, the substacks $\mls K_{g, \ba}(\mc X, d)$ are proper.

\end{thm}

The proof of \ref{T:proper} occupies the remainder of this section (see, in particular, \ref{C:findiag}, \ref{C:prop}, and \ref{C:qc}). We will  use the stabilization morphism constructed in \ref{T:stabilize} in multiple places.

\subsection{Valuative criterion for separatedness}
Let $R$ be a discrete valuation ring with generic point $\eta \in \Sp(R)$ and closed point $s \in \Sp(R)$. Let 
\[
\bC_j = (\,C_j/R,\; \{s_i^{(j)}\}_{i=1}^n, \;\mls N_j,\; \ell_i: \MS{S}{C_j} \to \MS{S}{C_j}'\,) \quad \quad \quad \text{for } j=1, 2
\]
be two generalized log twisted curves over $R$ that are generically isomorphic.

\begin{lem}\label{lem:compare-curves}
There is a generalized log twisted curve $\bC_3$ and contractions $(q, \phi'_i): \bC_3 \to \bC_i$ for $i=1, 2$.
\end{lem}
\begin{proof}
Let $(C_\eta, \{s_{i, \eta}\}_{i=1}^n)$ denote the marked prestable curve equal to the common generic fiber of $C_1$ and $C_2$. By considering the closure of $C_\eta \hookrightarrow C_1 \times_R C_2$ and applying the semistable reduction theorem we can, after possibly replacing $R$ by a finite extension, find a semistable model $C_3$ of $C_\eta$ and a diagram
\[
\begin{tikzcd}[row sep=3pt]
&C_3 \arrow[dl, "q_1"'] \arrow[dr, "q_2"] \\
C_1 && C_2
\end{tikzcd}
\]
of semistable models for $C_\eta.$ Since $C_3$ is proper over $R$, by the valuative criterion there exist unique extensions $s_i^{(3)}: \Sp(R) \to C_3$ of $s_{i, \eta}$, and $q_1$ and $q_2$ are morphisms of marked prestable curves.

We now construct an admissible sheaf $\mls N_3 \subset \oplus s^{(3)}_{i, *} \mathbf{Q}_{\geq 0}$ on $C_3$. By \cite[2.22]{OWI}, since $R$ is local, the sheaves $\mls N_1$ and $\mls N_2$ are the images of admissible submonoids $N_1$ and $N_2$ of $\mathbf{Q}^n_{\geq 0}$. If $\bar x$ is a geometric point of the generic fiber $C_\eta$ and $I \subset \{1, \ldots, n\}$ is the set of sections containing the image of $\bar x$, then by \cite[2.20]{OWI} we have, for $j=1, 2$, an equality
\[
\mls N_{ j,\, \bar x} = N_j^I
\]
where the left hand side is the stalk of $\mls N_j$ at $\bar x$ and the right hand side is the projection of $N_j \subset \mathbf{Q}^n_{\geq 0}$ to $\mathbf{Q}^I_{\geq 0}.$
Hence the quotients $N_j^I$ are independent of $j$. It follows that the quotients $N_j^{\{i\}}$ are also independent of $j$ for all $i =1, \ldots, n$, and by \cite[2.11]{OWI} we have an admissible submonoid $N_3:= \oplus_{i=1}^n N_j^{\{i\}}$ of $\mathbf{Q}^n_{\geq 0}$ containing both $N_1$ and $N_2$. We define $\mls N_3$ to be the associated admissible sheaf.

Lastly we construct an admissible inclusion $\ell_3: \MS{S}{C_3} \to \MS{S}{C_3}'$. The contractions $q_j: C_3 \to C_j$ give canonical injective morphisms of log structures
\[
\phi_j: \MS{S}{C_j} \to \MS{S}{C_3}.
\]
Since $\MS{S}{C_j}'$ is a subsheaf of $(\MS{S}{C_j})_{\mathbf{Q}}$, these induce injective morphisms $\phi_j^{\mathbf{Q}}: \MS{S}{C_j}' \to (\MS{S}{C_3})_{\mathbf{Q}}$, and we define $\MS{S}{C_3}'$ to be the smallest subsheaf of $(\MS{S}{C_j})_{\mathbf{Q}}$ containing the subsheaves 
\[
\MS{S}{C_3}, \quad \quad \phi_1^{\mathbf{Q}}(\MS{S}{C_1}'), \quad \quad \text{and} \quad \quad \phi_2^{\mathbf{Q}}(\MS{S}{C_2}').
\]
Let $\ell_3: \MS{S}{C_3} \to \MS{S}{C_3}'$ be the natural inclusion. We claim $\ell_3$ is simple. This claim can be verified by considering the stalks at the generic and closed points of $\Sp(R)$. 

We have constructed the data of a generalized log twisted curve $\bC_3 = (C_3, \{s_i^{(3)}\}_{i=1}^n, \mls N_3, \ell_3: \MS{S}{C_3} \to \MS{S}{C_3}').$ It follows from the construction that we have contraction morphisms $\bC_3 \rightarrow  \bC_j$ for $j=1, 2$. 

\end{proof}

Now we incorporate morphisms to $\mc X$: let $(\bC_1, f_1)$ and $(\bC_2, f_2)$ be two stable maps to $\mc X$ over $R$ of type $(g, \ba)$ such that the restrictions to the generic fiber $\eta$ are isomorphic. Let $\bC_3$ be as in \ref{lem:compare-curves} and let $Q_j: \mls C_3 \to \mls C_j$ be the morphisms of stacks induced by the contractions in that lemma. 

\begin{lem}\label{L:morphs}
The compositions $\mls C_3 \xrightarrow{Q_j} \mls C_j \xrightarrow{f_j} \mc X$ agree, for $j=1,2$. 
\end{lem}
\begin{proof}
We first consider the induced morphisms of coarse spaces 
\begin{equation}\label{eq:morphs1}
C_3 \to C_j \to X \quad \quad \text{for }j=1, 2.
\end{equation}
If we restrict this composition to the generic point, we obtain the same morphism for $j=1, 2$. Since $C_3$ is reduced and $X$ is separated, by \cite[\href{https://stacks.math.columbia.edu/tag/02KB}{Tag 02KB}]{stacks-project} 
the two morphisms in \eqref{eq:morphs1} agree. Then the two morphisms $Q_j \circ f_j$
agree away from the marked points and nodes in the closed fiber $\mls C_{3,s}$. By purity \cite[2.4.1]{AV}
these morphisms in fact agree on the open complement $\mls C_3^\circ$ of the marked points in the closed fiber. 

Let $X' \to X$ be an \'etale morphism from an affine scheme $X'$ such that $X' \times_X \mc X$ is isomorphic to $[Y/G]$ for $Y/X'$ finite and $G$ a finite group. Let $V \to C$ be the \'etale morphism obtained by pulling back $X'$ and let $\mls U \to \mls C_3$ (resp. $\mls U^\circ \to \mls C^\circ_3$) denote $V \times_C \mls C_3$ (resp. $V \times_C \mls C^\circ_3$). We then have a map 
\[
f^\circ: \mls U^\circ \to [Y/G],
\]
and it suffices (by descent) to show that this map extends to at most one morphism $f: \mls U \to [Y/G].$ 
For this, we may replace $V$ by an open subset whose image in $C$ does not contain points that are nodes in their respecive fibers. After making this replacement we have that the coarse moduli morphism $\mls U \to V$ is flat by \cite[3.8(2)]{OWI}.

Such a morphism corresponds to a $G$-torsor $p:P \to \mls U$ and a $G$-equivariant map $P \to Y$. By \ref{P:extend} below the torsor $P$ is determined by its restriction $P_\eta$. Since $X'$, hence also $Y$, is affine, we can write $Y = \Sp(A)$ for some ring $A$ finite over $\Gamma(X', \mls O_{X'})$. Given the torsor $P$ the data needed to specify $f$ is a lift of the $G$-equivariant ring map induced by $f^\circ$:
\begin{equation}\label{eq:morphs2}
\begin{tikzcd}
A \arrow[r, dashrightarrow]\arrow[dr, "f^\circ"'] &\Gamma(P, \mls O_P) \arrow[d] \\
& \Gamma(P_{\mls U^\circ}, \mls O_{P_{\mls U^\circ}}).
\end{tikzcd}
\end{equation}
Since $P \to \mls U$ is finite and flat, $p_*\mls O_P$ is a sheaf of finite flat $\mls O_{\mls U}$-modules, hence a finite rank locally free sheaf by \cite[\href{https://stacks.math.columbia.edu/tag/02KB}{Tag 02KB}]{stacks-project}. 
Moreover and we have $\Gamma(P, \mls O_P)=\Gamma(\mls U, p_* \mls O_P)$ and $\Gamma(P_{\mls U^\circ}, \mls O_{P_{\mls U^\circ}})=\Gamma(\mls U^\circ, (p_* \mls O_P)|_{\mls U^\circ})$.  The uniqueness of the dashed arrow in \eqref{eq:morphs2} now follows from \ref{L:sections} below.
\end{proof}

\begin{lem}\label{P:extend} The restriction functor
$$
\text{\rm Tors}^G(\mls U)\rightarrow \text{\rm Tors}^G(\mls U_\eta )
$$
from the category of $G$-torsors on $\mls U$ to the category of torsors over $\mls U_\eta $ is fully faithful.
\end{lem}
\begin{proof}
Fix two torsors $P_1, P_2\in \text{\rm Tors} ^G(\mls U)$ and an isomorphism $\sigma _\eta :P_{1, \eta }\rightarrow P_{2, \eta }$ between their generic fibers.  First, an extension of $\sigma _{\eta }$ is unique. This is because $P_1$ and $P_2$ are flat over $R$: If $\sigma :P_1\rightarrow P_2$ is an isomorphism then $\sigma $ is determined by its graph $\gamma :P_1\hookrightarrow P_1\times _{\mls U}P_2$ and by flatness of $P_1$ over $R$ this graph is the scheme-theoretic closure in $P_1\times _{\mls U}P_2$ of the graph of $\sigma _\eta $ over the generic fiber.  

For the existence of the extension,  we claim that taking the closure of the graph of $\sigma _{\bar \eta}$ in $P_1\times _{\mls U}P_2$ defines an extension $\sigma: P_1\rightarrow P_2$.  To see this we may work \'etale locally on $\mls U$ so it suffices to consider the case when $P_1$ and $P_2$ are trivial. In this case the result amounts to the statement that $\mls U_\eta $ is schematically dense in $\mls U$ since in this case $P_1$ and $P_2$ are disjoint unions of copies of $\mls U$. 
\end{proof}

\begin{lem} \label{L:sections} Let $\mls E$ be a vector bundle on $\mls U$.  Then the restriction map
$$
\Gamma (\mls U, \mls E)\rightarrow \Gamma (\mls U^\circ , \mls E|_{\mls U^\circ} )
$$
is an isomorphism.
\end{lem}
\begin{proof}
The proof depends critically on the fact that the image of $V$ in $C$ contains no nodes, so $V$ is regular and the coarse moduli morphism $\mls U \to V$ is flat.

We may work \'etale locally $V$ so it suffices to consider the case when $\mls U = [U/H]$ with $U$ and $V$ strictly henselian local and $H$ an abelian group 
The scheme $U$ is finite and flat over the regular $V$ and therefore is Cohen-Macaulay \cite[18.17]{Eisenbud}.  In particular, $U$ is S2 so the map $\Gamma (U, \mls E)\rightarrow \Gamma (U^\circ , \mls E|_{U^\circ})$ is an isomorphism.  Taking $H$-invariants we get the result.
\end{proof}

The valuative criterion for separatedness referred to below is defined in e.g. \cite[\href{https://stacks.math.columbia.edu/tag/0E80}{Tag 0E80}]{stacks-project} or \cite[7.8]{LMB}.
\begin{cor}\label{C:sep}
The morphism $\mls K_{g, \ba}(\mc X) \to B$ satisfies the valuative criterion for separatedness.
\end{cor}
\begin{proof}
The stable maps $(\bC_1, f_1)$ and $(\bC_2, f_2)$ are two $\ba$-stabilizations of the $\ba$-prestable map $(\bC_3, Q_j \circ f_j)$ (note that by \ref{L:morphs} the morphism $Q_j \circ f_j$ is independent of $j=1, 2$). The result now follows from the uniqueness of $\ba$-stabilizations \ref{T:stabilize}.
\end{proof}

\begin{cor}\label{C:findiag}
The morphism $\mls K_{g, \ba}(\mc X) \to B$ has finite diagonal.
\end{cor}
\begin{proof}
By \cite[\href{https://stacks.math.columbia.edu/tag/0E80}{Tag 0E80}]{stacks-project}, the morphism $\mls K_{g, \ba}(\mc X) \to B$ is separated. Hence the diagonal is proper. Since the automorphism group scheme of a stable map is finite, the diagonal is finite.
\end{proof}

\subsection{Valuative criterion for properness}
We now make the additional assumption that $\mc X$ is proper.
Let $\bone^n$ denote the weight vector $(1, 1, \ldots, 1)$ with $n$ entries. Recall from \eqref{eq:stab1} the stabilization morphism \begin{equation}\label{eq:qc1}
    \mathrm{stab}: \mls K^{\ba-\mathrm{stab}}_{g, \bone^n}(\mc X) \to \mls K_{g, \ba}(\mc X).
    \end{equation}

\begin{lem}\label{L:surj}
The morphism \eqref{eq:qc1} is surjective.
\end{lem}
\begin{proof}
It is enough to show that, given an $\ba$-stable map $(\bD, f)$  over an algebraically closed field $k$, there is a generalized log twisted curve $\bC$ and a contraction $\bC \to \bD$ such that the composition $\mls C \to \mls D \to \mc X$ is $\bone^n$-stable. For this it is enough to consider the case where there is a unique geometric point $\bar x \to \mls D$ whose image is equal to $w$ markings, where $w$ is at least two. 

Let $N \subseteq \mathbf{Q}^w_{\geq 0}$ be the stalk of the admissible sheaf of $\bD$ at $\bar x$. For $i=1, \ldots, w$ define $r_i \in \mathbf{Z}_{>0}$ (resp. $m \in \mathbf{Z}_{>0}$) to be the integer such that the image of $N$ under the projection to the $i$th coordinate (resp. summation map)
\[
\mathbf{Q}^w_{\geq 0} \to \mathbf{Q}_{\geq 0}
\]
is $\frac{1}{r_i} \mathbf{Z}_{\geq 0}$ (resp. $\frac{1}{m}\mathbf{Z}_{\geq 0}$). We construct $\bC$ as follows. Define $\bE$ to be the generalized log twisted curve whose associated stack is $\mathbf{P}^1$ rooted at (any choice of) $w$ distinct points with orders $r_1, \ldots r_w$, respectively. Define $\bC'$ to be the generalized log twisted curve naturally associated to the rigidification of $\mls D$ at $\bar x$. Glue one of the non-stacky points of $\bE$ to the image of $\bar x$ in $\bC'$. Since $k$ is algebraically closed, we can take $\bC$ to be a 
generalized log twisted curve obtained from this glued one by adding orbifold structure to order $m$ at the node joining $\bE$ to $\bC'$ (note that $\bC$ is not unique). There is then a contraction $\bC \rightarrow \bD$. 
\end{proof}


\begin{cor}\label{C:prop}
If $\mc X$ is proper, the  morphism $\mls K_{g, \ba}(\mc X) \to B$ satisfies the valuative criterion for properness.
\end{cor}
\begin{proof}
This follows from the surjectivity in \ref{L:surj} after noting that $\mls K^{\ba-\mathrm{stab}}_{g, \ba}(\mc X)$ is a closed substack of $\mls K_{g, \ba}(\mc X)$ by \ref{T:stabilize} and $\mls K_{g, \bone^n}(\mc X)$ satisfies the valuative criterion for properness by \cite[1.4.1]{AV} (and the following remark).
\end{proof}
\begin{rem} In most of the literature the valuative criterion for properness for $\mls K_{g, \bone ^n}(\mc X)$ is verified under the further assumption that the coarse space of $\mc X$ is projective, and not just proper.  One can reduce the proper case to the projective case as follows.  Let $A$ be a discrete valuation ring over $B$ with field of fractions $K$, and let $\Sp (K)\rightarrow \mls K_{g, \bone ^n}(\mc X)$ denote a map corresponding to a twisted stable map $\mls C_K\rightarrow \mc X$. We show that after possibly replacing $A$ by an \'etale extension we can extend the given map to a map $\Sp (A)\rightarrow \mls K_{g, \bone ^n}(\mc X)$.   Let $Z\subset X_A$ be the scheme-theoretic image of the coarse space  $C_\eta $.  Then $Z$ is flat over $A$ of relative dimension $1$ and therefore, after possibly replacing $A$ by an \'etale cover (which is permissible in order to verify the valuative criterion), we can arrange that $Z$ is projective over $A$ \cite[\href{https://stacks.math.columbia.edu/tag/0E6F}{Tag 0E6F}]{stacks-project}.  Our stable map over $K$ can be viewed as a stable map to $Z\times _X\mc X$, which is a stack with projective coarse moduli space.  By the projective case we then get an extension over $A$ of the induced stable map to $Z\times _X\mc X$, and therefore also an extension of the stable map to $\mc X$. 
\end{rem}

\subsection{Quasicompactness and the proof of \ref{T:proper}}

We now impose the additional assumption that
  $X$ is projective over $B$ and fix a $B$-relatively ample invertible sheaf $\mls O_X(1)$.

\begin{cor}\label{C:qc} The stack $\mls K_{g, \ba }(\mc X, d)$ is a quasi-compact  open and closed substack of $\mls K_{g,  \ba }(\mc X)$. In particular $\mls K_{g, \ba }(\mc X, d)$ is proper over $B$.
\end{cor}
\begin{proof}
    That $\mls K_{g,  \ba }(\mc X, d)$ is open and closed in $\mls K_{g,  \ba }(\mc X)$ follows from the fact that the degree of a line bundle on a flat family of nodal curves is a locally constant function. Quasicompactness follows from the cover \ref{L:surj} and \cite[\href{https://stacks.math.columbia.edu/tag/050X}{Tag 050X}]{stacks-project}, noting that $\mls K^{\ba-\mathrm{stab}}_{g, \bone^n}(\mc X, d)$ is a closed substack of $\mls K_{g, \bone^n}(\mc X)$ by \ref{T:stabilize} and $\mls K_{g, \bone^n}(\mc X)$ is the Abramovich-Vistoli stack of twisted stable maps and hence quasicompact by \cite[1.4.1]{AV}. Now $\mls K_{g, \ba}(\mc X, d)$ is proper by \ref{C:sep} and \ref{C:prop}, see e.g. \cite[7.12]{LMB}.
    (Note that $\mls K_{g, \ba}(\mc X)$ and hence $\mls K_{g, \ba}(\mc X, d)$ are quasiseparated by \ref{T:7.4}.)
\end{proof}

\section{Global quotient structure}\label{S:section14}

Following the strategy in \cite{AGOT} we can prove that under appropriate assumptions our stacks $\mls K_{g, \ba}(\mc X, d)$ are global quotient stacks.  

\begin{pg}
Assume that $B$ is quasi-compact and that $\mc X = [M/G]$, where $M$ is a quasi-projective $B$-scheme and $G$ is an algebraic group over $B$ acting on $M$.  Assume further that the coarse space $X$ of $\mc X$ is projective over $B$ and fix a relatively ample invertible sheaf $\mls O_X(1)$ on $X$.

In this case the stack $\mc X$ has a generating sheaf and we fix one such sheaf $\mc V$ on $\mc X$ (see \cite[Theorem 5.5]{OlssonStarr}).
\end{pg}

\begin{thm}\label{T:14.2} Under these assumptions every connected component of $\mls K_{g,  \ba }(\mc X, d)$ admits a closed immersion into a stack of the form $[\mc P/\mathbf{G}]$, where $\mc P$ is a scheme smooth and quasi-projective over $B$ and $\mathbf{G}$ is an algebraic group acting on $\mc P$. 
\end{thm}

The proof of this theorem occupies the remainder of this section.

 Let $S$ be a $B$-scheme and let $(\mathbf{C}, f)$ be an object of $\mls K_{g,  \ba }(\mc X, d)(S)$.  So we have an induced diagram
$$
\xymatrix{
\mls C\ar[r]^-{f}\ar@/_2pc/[dd]_-{g}\ar[d]_-\pi & \mc X\ar[d]\\
C\ar[d]_-{\bar g}\ar[r]^-{\bar f}& X\\
S.&}
$$

\begin{lem} The line bundle 
 $\mc L:= \omega _{C/S}(\Sigma _is_i)\otimes \bar f^*\mls O_X(3)$ is relatively ample on $C$. 
\end{lem}
\begin{proof}
    It suffices to consider the case when $S = \Sp (k)$ is the spectrum of an algebraically closed field.  In this case it suffices to show that any rational component $P\subset C$ which is contracted by $\bar f$ contains at least three marked points or nodes. This follows from the definition of stability which requires the automorphism group of $C$ to be finite. 
\end{proof}
 For a sheaf $\mc F$ on $\mls C$ and integer $m$ let $\mc F(m)$ denote the sheaf $\mc F\otimes \pi ^*\mc L^{\otimes m}$.

\begin{pg} By the same argument proving \cite[2.2.1]{AGOT} there exists an integer $m_0$ such that the following hold for all $m\geq m_0$ and for any $B$-scheme $S$ and object $(\bC, f)$ of $\mls K_{g, \ba }(\mc X, d)(S)$.
\begin{enumerate}
    \item [(i)] $R^ig_*(f^*\mc V^\vee (m)) = 0$ for all $i>0$; 
    \item [(ii)] The sheaf $g_*(f^*\mc V^\vee (m))$ is locally free and its formation commutes with arbitrary base change $S'\rightarrow S$. 
    \item [(iii)] The natural map
\begin{equation}\label{E:9.4.1}
\bar g^*(g_*f^*\mc V^\vee (m))\rightarrow \pi _*f^*\mc V^\vee (m)
\end{equation}
is surjective.  
Note that since $f^*\mc V^\vee $ is generating this implies that the map
    $$
g^*g_*f^*\mc V^\vee (m)\otimes f^*\mc V(-m)\rightarrow \mls O_{\mls C}
$$
is also surjective. To see it observe that $\mls V$ being generating implies that $(\pi ^*\pi _*f^*\mc V^\vee (m))\otimes \mc V(-m)=\pi ^*\pi _*(\mls Hom(f^*\mc V(-m), \mls O_{\mls C})\otimes f^*\mc V(-m)\rightarrow \mls O_{\mls C}$ is surjective which then combines with \eqref{E:9.4.1} to give the assertion. 
\item [(iv)] $\mc L^{\otimes m}$ is relatively very ample on $C$, $R^i\bar g_*\mc L^{\otimes m}$ vanishes for all $i>0$ and the same holds after arbitary base change $S'\rightarrow S$, $\bar g_*\mc L^{\otimes m}$ is locally free and its formation commutes with base change, and the adjunction map $\bar g^*\bar g_*\mc L^{\otimes m}\rightarrow \mc L^{\otimes m}$ is surjective.
\end{enumerate}
\end{pg}

\begin{pg}
    Fix such an integer $m$ and for integers $l_1$ and $l_2$ let $\mls K_{l_1, l_2}\subset \mls K_{g,  \ba }(\mc X, d)$ denote the open and closed substack classifying log twisted stable maps $(\mathbf{C}, f)$ of degree $d$ for which $g_*f^*\mc V^\vee (m)$ has rank $l_1$ and $\bar g_*\mc L^{\otimes m}$ has rank $l_2$.    To prove \ref{T:14.2}, it is enough to show that $\mls K_{l_1, l_2}$ has a closed immersion into a global quotient stack. Let $\widetilde \Gamma _{l_1, l_2}$ be the stack which associates to a scheme $S$ the collection of data  $(\mathbf{C}, f, \iota _1, \iota _2)$, where $(\mathbf{C}, f)$ is an object of $\mls K_{l_1, l_2}(S)$ and 
    $$
    \iota _1:\mls O_S^{\oplus l_1}\rightarrow g_*f^*\mc V^\vee (m), \ \ \iota _2:\mls O_S^{\oplus l_2}\rightarrow \bar g_*\mc L^{\otimes m}
    $$
    are isomorphisms of locally free sheaves on $S$.  

    Property (iv) above implies that for such a collection of data over a scheme $S$ the composition, which we denote by $j_C$,
    \[
    \begin{tikzcd}
C \arrow[rr, bend left, "j_c"] \arrow[r] & \mathbf{P}(\bar g_*\mc L^{\otimes m}) \arrow[r, "\iota_2"] &\mathbf{P}_S^{l_2-1}
    \end{tikzcd}
    \]
is a closed immersion.  
Let $h:\mls C\rightarrow \mc X_S\times \mathbf{P}_S^{l_2-1}$ be the map $(f, j_C\circ \pi )$.   The map $h$ is proper, representable (since $f$ is representable), and quasi-finite (since the map on coarse spaces is an immersion), and therefore $h$ is finite and $\mls C$ is realized as the relative spectrum over $\mc X_S\times \mathbf{P}_S^{l_2-1}$ of $h_*\mls O_{\mls C}$.

The sheaf $\mc V^\vee \boxtimes \mls O_{\mathbf{P}^{l_2-1}_S}(1)$ on $\mc X_S\times \mathbf{P}_S^{l_2-1}$ (where as usual we write $\boxtimes $ for the operation of pulling the two sheaves back to the product and tensoring them together) is a generating sheaf since $\mc V^\vee $ is generating on $\mc X$.  It follows that if $a:\mc X_S\times \mathbf{P}_S^{l_2-1}\rightarrow X_S\times \mathbf{P}_S^{l_2-1}$ is the coarse moduli space map then the map
\begin{equation}\label{E:iso1}
a^*a_*((h_*\mls O_{\mls C})\otimes (\mc V^\vee \boxtimes \mls O_{\mathbf{P}^{l_2-1}_S}(1)))\otimes (\mc V\boxtimes \mls O_{\mathbf{P}^{l_2-1}_S}(-1))\rightarrow h_*\mls O_{\mls C}
\end{equation}
is surjective.  Now we have isomorphisms
\begin{equation}\label{E:iso2}
a_*((h_*\mls O_{\mls C})\otimes (\mc V^\vee \boxtimes \mls O_{\mathbf{P}^{l_2-1}_S}(1)))\simeq a_*h_*h^*(\mc V^\vee \boxtimes \mls O_{\mathbf{P}^{l_2-1}_S}(1))\simeq a_*h_*(f^*\mc V^\vee (m))\simeq \bar h_*\pi _*(f^*\mc V^\vee (m)),
\end{equation}
where $\bar h:C\rightarrow X\times \mathbf{P}^{l_2-1}_S$ is the closed immersion (since $j_C$ is a closed immersion) defined by $h$. By (iii)  the map 
\begin{equation}\label{eq:surj2}
\mls O_{C}^{\oplus l_1}\rightarrow \pi _*(f^*\mc V^\vee (m))
\end{equation}
(equal to $\iota_1$ followed by \eqref{E:9.4.1}) is surjective, 
and since $\bar h$ is a closed immersion it follows that the map
$$
\mls O_{X_S\times \mathbf{P}_S^{l_2-1}}^{\oplus l_1}\rightarrow \bar h_*\pi _*(f^*\mc V^\vee (m))
$$
(equal to the surjection $\mls O_{X_S\times \mathbf{P}_S^{l_2-1}}^{\oplus l_1}\rightarrow \bar h_*\mls O_C^{\oplus l_1}$ followed by $\bar h_*$ applied to \eqref{eq:surj2})
is also surjective.  
Combining this with \eqref{E:iso1} and \eqref{E:iso2} we find that the natural map
$$
(\mc V\boxtimes \mls O_{\mathbf{P}_S^{l_2-1}}(-1))^{\oplus l_1}\rightarrow h_*\mls O_{\mls C}
$$
is surjective.
 The induced map of quasi-coherent sheaves of algebras on $\mc X_S\times \mathbf{P}_S^{l_2-1}$
$$
\text{Sym}^\bullet( (\mc V\boxtimes \mls O_{\mathbf{P}_S^{l_2-1}}(-1))^{\oplus ^{l_1}})\rightarrow h_*\mls O_{\mls C}
$$
is therefore also surjective, and
taking relative spectra we get a  closed immersion
 \[j_{\mls C}:\mls C\hookrightarrow \mathbf{A}\] lifting $j_C$.  Here $\mathbf{A}$ is the affine bundle over  $\cX_S \times \mathbf{P}_S^{l_2-1}$ associated to $\mc V\boxtimes \mls O_{\mathbf{P}_S^{l_2-1}}(-1)^{\oplus l_1}$. 
\end{pg}

\begin{pg}\label{P:7.6b}
From this and \cite[7.18]{OWI} it follows that any automorphism of an object $(\mathbf{C}, f, \iota _1, \iota _2)$ of $\widetilde{\Gamma}_{l_1, l_2}$ must be the identity, since the construction of the closed embedding $j_{\mls C}$ implies that such an automorphism induces the identity on $\mls C$. 
Moreover the stack $\widetilde \Gamma _{l_1, l_2}$ is the product of frame bundles over $\mls K_{l_1, l_2}$ (or equivalently a stack of sections) and therefore algebraic.  It follows that $\widetilde \Gamma_{l_1, l_2}$ is an algebraic space.  Furthermore  $\widetilde \Gamma _{l_1, l_2}$ has an action of $GL_{l_1}\times GL_{l_2}$, obtained by changing the isomorphisms $(\iota _1, \iota _2)$, for which
    $$
    \mls K_{l_1, l_2}\simeq [\widetilde \Gamma _{l_1, l_2}/GL_{l_1}\times GL_{l_2}].
    $$
\end{pg}

\begin{pg}\label{P:9.7} Let $\mathbf{G}$ denote the quotient of $GL_{l_1}\times GL_{l_2}$ by $\mathbf{G}_m$ embedded diagonally and let $\Gamma _{l_1, l_2}$ denote the quotient of $\widetilde \Gamma _{l_1, l_2}$ by the $\mathbf{G}_m$-action.  So $\Gamma _{l_1, l_2}$ is a $\mathbf{G}$-torsor over $\mls K_{l_1, l_2}$ and $\mls K_{l_1, l_2}\simeq [\Gamma _{l_1, l_2}/\mathbf{G}]$.  The stack $\Gamma _{l_1, l_2}$ classifies data $(\mathbf{C}, f, (\lambda _1, \lambda _2))$, where $(\mathbf{C}, f)\in \mls K_{l_1, l_2}$ and $(\lambda _1, \lambda _2)$ is a section of the sheaf quotient
$$
(\underline {\text{Isom}}(\mls O_S^{\oplus l_1}, g_*f^*\mc V^\vee (m))\times \underline {\text{Isom}}(\mls O_S^{\oplus l_2},\bar g_*\mc L^{\otimes m}))/\mathbf{G}_m
$$
where $\mathbf{G}_m$ acts diagonally.

The embedding $j_C$ and the map $h$ depend only on the projectivization of $\iota _2$. 
Furthermore the preceding construction of the embedding $j_{\mls C}:\mls C\hookrightarrow \mathbf{A}$ depends only on the pair $(\lambda _1, \lambda _2)$ so we get a morphism
$$
\rho :\Gamma _{l_1, l_2}\rightarrow \underline {\text{Hilb}}(\mathbf{A})
$$
to the Hilbert scheme of closed substacks of $\mathbf{A}$. This Hilbert scheme is a disjoint union of schemes quasi-projective over $B$ by \cite[1.5]{OlssonStarr}.

The fact that $j_{\mls C}$ only depends on the pair $(\lambda _1, \lambda _2)$ implies, as in \ref{P:7.6b}, that any automorphism of an object $(\mathbf{C}, f, (\lambda_1, \lambda_2))$ must be the identity, since such an automorphism will induce the identity on $\mls C$. So $\Gamma_{l_1, l_2}$ is an algebraic space.

Note also that there is a natural action of $GL_{l_1}\times GL_{l_2}$ on $\mathbf{A}$, 
which in fact factors through an action of $\mathbf{G}$, and the map $\rho $ is $\mathbf{G}$-equivariant. 
Furthermore, each of the sections $s_i$ define (upon composition with $j_C$) a point of $\mathbf{P}^{l_2-1}$, so we get a map
$$
\tau := \rho \times (j_C\circ s_1)\times \cdots \times (j_C\circ s_n):\Gamma _{l_1, l_2}\rightarrow \underline {\text{Hilb}}(\mathbf{A})\times \mathbf{P}_B^{l_2-1}\times \cdots \times \mathbf{P}_B^{l_2-1}
$$
which is equivariant with respect to the $\mathbf{G}$-actions. 
\end{pg}

\begin{lem}\label{L:9.8}
(i) The map $\tau $ is separated and quasi-finite.

(ii) $\Gamma _{l_1, l_2}$ is a scheme quasi-projective over $B$. 
\end{lem}
\begin{proof}
The algebraic space $\Gamma _{l_1, l_2}$ is separated over $B$ being a $\mathbf{G}$-torsor over a $B$-separated algebraic stack.  From this it follows that $\tau $ is separated.  Furthermore, the fibers of $\tau $ classify log twisted stable maps with a fixed associated stack with marked points.  Therefore $\tau$ has finite fibers by \cite[7.28]{OWI} and is quasi-finite by \cite[\href{https://stacks.math.columbia.edu/tag/06RT}{Tag 06RT}]{stacks-project} (since $\Gamma_{l_1, l_2}$ and hence $\tau$ are locally of finite type). This proves (i).

It follows from (i) that $\Gamma _{l_1, l_2}$ is a separated algebraic space  quasi-finite over a scheme (the target of $\tau$).  By \cite[\href{https://stacks.math.columbia.edu/tag/03XX}{Tag 03XX}]{stacks-project} and \cite[\href{https://stacks.math.columbia.edu/tag/02LR}{Tag 02LR}]{stacks-project} it follows that $\Gamma _{l_1, l_2}$ is a scheme, quasi-affine over a scheme quasi-projective over $B$ and hence itself is quasi-projective over $B$. 
\end{proof}

\begin{pg} 
Since $\tau $ is quasi-finite and separated by \ref{L:9.8}, and therefore quasi-affine by \cite[\href{https://stacks.math.columbia.edu/tag/02LR}{Tag 02LR}]{stacks-project},  and the target admits a $\mathbf{G}$-equivariant $B$-relatively ample line bundle  (see for example \cite[\S 2.5]{AGOT}),  the scheme $\Gamma _{l_1, l_2}$ also admits a $\mathbf{G}$-equivariant $B$-relatively ample line bundle $\mls M$ by \cite[\href{https://stacks.math.columbia.edu/tag/0892}{Tag 0892}]{stacks-project}. 
Therefore $\Gamma _{l_1, l_2}$ admits an equivariant closed immersion into a smooth quasi-projective  $B$-scheme $\mc P$ 
with $\mathbf{G}$-action 
and we get a closed immersion
$$
\mls K_{l_1, l_2} \simeq [\Gamma _{l_1, l_2}/\mathbf{G}]\hookrightarrow [\mc P/\mathbf{G}].
$$
This completes the proof of \ref{T:14.2}. \qed
\end{pg}

\bibliographystyle{amsplain}
\bibliography{bibliography}{}

\end{document}